%% file: Ivrii-BMS.tex
\documentclass[12pt,oneside,openany,article]{memoir}
\usepackage{mempatch}

\nouppercaseheads
\usepackage{amsthm}
\usepackage{etex}
\usepackage[english]{babel}

\usepackage{caption}

\usepackage[leqno]{mathtools}

\usepackage[scr]{rsfso}
\usepackage{upgreek}
\usepackage[compress,square,comma,numbers]{natbib}
\usepackage{hypernat} %required if hyperref is used as well

\usepackage{sfmath}

\usepackage{microtype}

\usepackage{array}

\usepackage{graphicx,xcolor}

\usepackage{hyperxmp}
\usepackage{xr-hyper}
\usepackage{nameref}
\usepackage[pdftex,bookmarks,pdfnewwindow,plainpages=false,unicode]{hyperref}

\usepackage{bookmark}
\usepackage{enumitem}
\usepackage{amssymb}

\hypersetup{
colorlinks=true,
linkcolor=black,
citecolor=black,
urlcolor=blue!80!cyan,
pdfauthor={Victor Ivrii},
pdftitle={100 years of Weyl's law},
pdfsubject={Sharp Spectral Asymptotics},
pdfkeywords={Microlocal Analysis,  Sharp Spectral Asymptotics},
bookmarksdepth={4}
}

\theoremstyle{plain}
\newtheorem{theorem}{Theorem}[section]

\newtheorem{proposition}[theorem]{Proposition}
\newtheorem{corollary}[theorem]{Corollary}

\theoremstyle{definition}
\newtheorem{definition}[theorem]{Definition}

\theoremstyle{remark}
\newtheorem{remark}[theorem]{Remark}
\newtheorem{example}[theorem]{Example}

\numberwithin{equation}{section}

\newenvironment{phantomequation}[1][]{\refstepcounter{equation}}{}

\newenvironment{claim}[1][{\textup{(\theequation)}}]{\refstepcounter{equation}\vglue10pt
\begin{trivlist}
\item[{\hskip\labelsep#1}]}{\vglue10pt\end{trivlist}}

\newcommand{\corr}{\mathsf{corr}}
\newcommand{\const}{\mathsf{const}}
\newcommand{\TF}{{\mathsf{TF}}}
\newcommand{\MW}{{\mathsf{MW}}}
\newcommand{\W}{{\mathsf{W}}}
\newcommand{\E}{{\mathsf{E}}}
\newcommand{\N}{{\mathsf{N}}}
\newcommand{\bound}{\mathsf{bound}}
\newcommand{\Scott}{\mathsf{Scott}}
\newcommand{\Schwinger}{\mathsf{Schwinger}}
\newcommand{\Dirac}{\mathsf{Dirac}}
\newcommand{\R}{\mathsf{R}}
\newcommand{\I}{\mathsf{I}}

\newcommand{\x}{{\mathsf{x}}}
\newcommand{\y}{{\mathsf{y}}}

\newcommand{\bR}{\mathbb{R}}
\newcommand{\bC}{\mathbb{C}}
\newcommand{\bZ}{\mathbb{Z}}
\newcommand{\bH}{\mathbb{H}}

\newcommand{\fN}{\mathfrak{N}}
\newcommand{\fH}{\mathfrak{H}}
\newcommand{\fD}{\mathfrak{D}}

\newcommand{\blangle}{{\boldsymbol{\langle}}}
\newcommand{\brangle}{{\boldsymbol{\rangle}}}

\newcommand{\cA}{\mathcal{A}}
\newcommand{\cB}{\mathcal{B}}
\newcommand{\cE}{\mathcal{E}}
\newcommand{\cK}{\mathcal{K}}
\newcommand{\cN}{\mathcal{N}}
\newcommand{\cF}{\mathcal{F}}
\newcommand{\cL}{\mathcal{L}}

\newcommand{\cY}{\mathcal{Y}}
\newcommand{\cZ}{\mathcal{Z}}

\newcommand{\sC}{\mathscr{C}}
\newcommand{\cQ}{\mathcal{Q}}
\newcommand{\sL}{\mathscr{L}}
\newcommand{\sH}{\mathscr{H}}

\newcommand{\D}{\mathsf{D}}
\newcommand{\sfH}{{\mathsf{H}}}

\newcommand{\dist}{\operatorname{dist}}
\newcommand{\Hess}{\operatorname{Hess}}
\newcommand{\vol}{\operatorname{vol}}
\newcommand{\supp}{\operatorname{supp}}
\newcommand{\rank}{\operatorname{rank}}
\newcommand{\WF}{\operatorname{WF}}
\newcommand{\Char}{\operatorname{Char}}
\newcommand{\Spec}{\operatorname{Spec}}
\newcommand{\Tr}{\operatorname{Tr}}
\newcommand{\tr}{\operatorname{tr}}

\newcommand{\eff}{{\mathsf{eff}}}

\newcommand\1{|\hskip-1.3pt|}

\newcommand{\Def}{\mathrel{\mathop:}=}
\newcommand{\boldrho}{\boldsymbol{\rho}}
\newcommand{\boldgamma}{\boldsymbol{\gamma}}
\newcommand{\boldupsigma}{{\boldsymbol{\upsigma}}}

\renewcommand{\Im}{\operatorname{Im}}
\renewcommand{\Re}{\operatorname{Re}}

\addtopsmarks{headings}{}{%
\createmark{chapter}{right}{shownumber}{}{. \ }
}
\title{100 years of Weyl's law}
\author{Victor Ivrii}

\begin{document}

\maketitle

\begin{abstract}
We discuss the asymptotics of the eigenvalue counting function for partial differential operators and related expressions paying the most attention to the sharp asymptotics. We consider Weyl asymptotics,  asymptotics with Weyl principal parts and correction terms and asymptotics with non-Weyl principal parts. Semiclassical microlocal analysis, propagation of singularities and related dynamics play crucial role.

We start from the general theory, then consider Schr\"odinger and Dirac operators with the strong magnetic field and, finally, applications to the asymptotics of the ground state energy of heavy atoms and molecules with or without a magnetic field.

\end{abstract}

\renewcommand\printtoctitle[1]{\section*{}}
\renewcommand\aftertoctitle{}
\tableofcontents

\chapter{Introduction}
\label{sect-1}

\section{A bit of history}
\label{sect-1-1}
In 1911, \href{https://en.wikipedia.org/wiki/Hermann_Weyl}{Hermann Weyl}, who at that time was a young German mathematician specializing in partial differential and integral equations, proved the following remarkable asymptotic formula describing distribution of (large) eigenvalues of the Dirichlet Laplacian in a bounded domain $X\subset \bR^d$:
\begin{equation}
\N(\lambda) = (2\pi)^{-d}\omega_d \vol (X) \lambda^{d/2} (1+o(1))\qquad \text{as\ \ } \lambda\to +\infty,
\label{1-1-1}
\end{equation}
where $\N(\lambda)$ is the number of eigenvalues of the (positive) Laplacian, which are less than $\lambda$\,\footnote{\label{foot-1} $\N(\lambda)$ is called the \emph{eigenvalue counting function}.}, $\omega_d$ is a volume of the unit ball in $\bR^d$, $\vol (X)$ is the volume of $X$. This formula was actually conjectured independently by \href{https://en.wikipedia.org/wiki/Arnold_Sommerfeld}{Arnold~Sommerfeld} \cite{sommerfeld:weyl} and \href{https://en.wikipedia.org/wiki/Hendrik_Lorentz}{Hendrik~Lorentz} \cite{lorentz:weyl} in 1910 who stated the \emph{Weyl's Law\/} as a conjecture based on the book of \href{https://en.wikipedia.org/wiki/John_William_Strutt,_3rd_Baron_Rayleigh}{Lord Rayleigh} ``\href{https://archive.org/details/theoryofsound02raylrich}{The Theory of Sound}'' (1887) (for details, see  \cite{arendt:weyl}).

H.~Weyl  published several papers \cite{weyl:asymp1,weyl:asymp,weyl:membrane, weyl:1913,weyl:1915}(1911--1915) devoted to the eigenvalue asymptotics for the Laplace operator (and also the elasticity operator) in a bounded domain with  regular boundary. In \cite{weyl:1913},  he published what is now known as \emph{Weyl's conjecture\/}
\begin{multline}
\N(\lambda) = (2\pi)^{-d}\omega_d \vol (X) \lambda^{d/2} \mp
\frac{1}{4}(2\pi)^{1-d}\omega_{d-1} \vol' (\partial X) \lambda^{(d-1)/2}\\ \text{as\ \ } \lambda\to +\infty
\label{1-1-2}
\end{multline}
for Dirichlet and Neumann boundary conditions respectively where
$\vol'(\partial X)$ is the $(d-1)$-dimensional volume of
$\partial X\in \sC^\infty$. Both these formulae appear in the toy model of a rectangular box $X=\{0<x_1<a_1,\ldots ,0<x_d<a_d\}$ and then $\N(\lambda)$ is the number of integer lattice points in the part of ellipsoid
$\{z_1^2 /a_1^2 +\ldots +z_d^2/a_d^2< \pi ^2\lambda \}$ with $z_j>0$ and $z_j\ge 0$ for Dirichlet and Neumann boundary conditions respectively\footnote{\label{foot-2} Finding sharp asymptotics of the number of the lattice points in the inflated domain is an important problem of the number theory.}.

After his pioneering  work, a huge number of papers devoted to spectral asymptotics were published.  Аmong the authors  were numerous prominent mathematicians.

After H.~Weyl, the next big step was made by \href{https://en.wikipedia.org/wiki/Richard_Courant}{Richard Courant} \cite{courant:eigen}(1920), who further developed the variational method and recovered the remainder estimate $O(\lambda^{(d-1)/2}\log \lambda)$. The variational method was developed further by many mathematicians, but it lead to generalizations rather than to getting sharp remainder estimates and we postpone its discussion until Section~\ref{sect-3-2}. Here we mention only \href{https://ru.wikipedia.org/wiki/\%D0\%91\%D0\%B8\%D1\%80\%D0\%BC\%D0\%B0\%D0\%BD\%2C_\%D0\%9C\%D0\%B8\%D1\%85\%D0\%B0\%D0\%B8\%D0\%BB_\%D0\%A8\%D0\%BB\%D1\%91\%D0\%BC\%D0\%BE\%D0\%B2\%D0\%B8\%D1\%87}{Mikhail Birman}, \href{https://en.wikipedia.org/wiki/Elliott_H._Lieb}{Elliott Lieb} and \href{https://en.wikipedia.org/wiki/Barry_Simon}{Barry Simon} and  their schools.

The next development was due to \href{https://en.wikipedia.org/wiki/Torsten_Carleman}{Torsten Carleman} \cite{carleman:membrane, carleman:asymp}(1934, 1936) who invented the \emph{Tauberian method\/} and  was probably the first to consider an arbitrary spacial dimension (H.~Weyl and R.~Courant considered only dimensions $2$ and $3$) followed by \href{https://en.wikipedia.org/wiki/Boris_Levitan}{Boris Levitan} \cite{levitan:asymp}(1952) and V.~G.~Avakumovi{\v{c}} \cite{avak:eigen}(1956) who, applied \emph{hyperbolic operator method\/} (see Section~\ref{sect-1-2}) to recover the remainder estimate $O(\lambda^{(d-1)/2})$, but only for closed manifolds and also for $e(x,x,\lambda)$ away from the boundary\footnote{\label{foot-3} Where here and below $e(x,y,\lambda)$ is the Schwartz kernel of the spectral projector.}.

After this,  \href{https://en.wikipedia.org/wiki/Lars_H\%C3\%B6rmander}{Lars~H\"ormander} \cite{hoermander:spectral,hoermander:riesz}(1968, 1969) applied Fourier integral operators in the framework of this method. \href{https://en.wikipedia.org/wiki/Hans_Duistermaat}{Hans~Duistermaat} and Victor~Guillemin \cite{duistermaat:period}(1975) recovered the remainder estimate $o(\lambda^{(d-1)/2})$ under the assumption that
\begin{claim}\label{1-1-3}
The set of all periodic geodesics has measure $0$
\end{claim}
observing that for the sphere neither this assumption nor (\ref{1-1-2}) hold. Here, we consider the phase space $T^*X$ equipped with the standard measure $dxd\xi$ where $X$ is a manifold\footnote{\label{foot-4} In fact the general scalar pseudodifferential operator  and Hamiltonian trajectories of its principal symbol were considered.}. This was a very important step since it connected the sharp spectral asymptotics with  classical dynamics.

The main obstacle was the impossibility to construct the parametrix of the hyperbolic problem near the boundary\footnote{\label{foot-5} Or even inside for elliptic systems with the eigenvalues of the principal symbol having the variable multiplicity.}. This obstacle was partially circumvented by Robert~Seeley~\cite{seeley:sharp, seeley:boundary}(1978, 1980) who recovered remainder estimate $O(\lambda^{(d-1)/2})$; his approach we will consider in Subsection~\ref{sect-4-1-2}. Finally the Author~\cite{Ivr1}(1980), using very different approach, proved (\ref{1-1-2}) under assumption that
\begin{claim}\label{1-1-4}
The set of all periodic geodesic billiards has measure $0$,
\end{claim}
which obviously generalizes (\ref{1-1-3}). Using this approach, the Author in~\cite{Ivr2} (1982) proved (\ref{1-1-1}) and (\ref{1-1-2}) for elliptic systems on manifolds without  boundary; (\ref{1-1-2}) was proven under certain assumption similar to (\ref{1-1-3}).

The new approaches were further developed during the 35 years to follow and many new ideas were implemented. The purpose of this article is to provide a brief and rather incomplete survey of the results and techniques. Beforehand, let us mention that the field was drastically transformed.

First, at that time, in addition to the problem that we described above, there were similar but distinct problems which we describe by examples:
\begin{enumerate}[label=(\alph*), wide, labelindent=0pt]
\setcounter{enumi}{1}
\item\label{sect-1-1-b}
Find the asymptotics as $\lambda\to +\infty$ of $\N(\lambda)$  for the Schr\"odinger operator $\Delta +V(x)$  in $\bR^d$ with potential  $V(x)\to +\infty$ at infinity;
\item\label{sect-1-1-c}
Find the asymptotics as $\lambda\to -0$ of $\N(\lambda)$   for the Schr\"odinger operator in $\bR^d$ with potential  $V(x)\to -0$ at infinity (decaying more slowly than $|x|^{-2}$);
\item\label{sect-1-1-d}
Find the asymptotics as $h\to +0$ of $\N^-(h)$ the number of the negative eigenvalues for the Schr\"odinger operator $h^2\Delta +V(x)$.
\end{enumerate}

These four problems  were being studied separately albeit by rather similar methods. However, it turned out that the latter problem \ref{sect-1-1-d} is more fundamental than the others which could be reduced to it by the variational \emph{Birman-Schwinger principle\/}.

Second, we should study the \emph{local semiclassical spectral asymptotics\/}, i.e. the asymptotics of $\int e(x,x,0)\psi (x) \,dx$ where $\psi\in \sC_0^\infty$ supported in the ball of radius $1$ in which\footnote{\label{foot-6} Actually, in the proportionally larger ball.} $V$ is of magnitude $1$\,\footnote{\label{foot-7} Sometimes, however, we consider \emph{pointwise semiclassical spectral asymptotics\/},  i.e. asymptotics of $e(x,x,0)$.}. By means of scaling we generalize these results for $\psi$ supported in the ball of radius $\gamma$ in which\footref{foot-6} $V$ is of magnitude $\rho$ with $\rho\gamma\ge h$ because in scaling $h\mapsto h/\rho \gamma$. Then in the general case we apply partition of unity with \emph{scaling functions\/} $\gamma(x)$ and $\rho(x)$.

Third, in the \emph{singular zone\/}  $\{x:\, \rho(x)\gamma(x)\le h\}$b we can apply variational estimates and combine them with the semiclassical estimates in the \emph{regular zone\/} $\{x:\, \rho(x)\gamma(x)\ge h\}$. It allows us to consider domains and operators with singularities.

Some further developments will be either discussed or mentioned in the next sections. Currently, I am working on the Monster book~\cite{futurebook} which summarizes this development. It is almost ready and is available online and we will often refer to it for details, exact statements and proofs.

Finally, I should mention that in addition to the variational methods and method of hyperbolic operator, other methods were developed: other Tauberian methods (like the method of the heat equation or the method of resolvent) and the almost-spectral projector method \cite{ST}. However, we will neither use nor even discuss them; for survey of different methods, see \cite{RSS}.

\section{Method of the hyperbolic operator}
\label{sect-1-2}

The method of the hyperbolic operator is one of the Tauberian methods proposed by T.~Carleman. Applied to the Laplace operator, it was designed as follows: let $e(x,y,\lambda)$ be the Schwartz kernel of a spectral projector and let
\begin{equation}
u(x,y,t)= \int_0^\infty  \cos (\lambda t)\, d_\lambda e(x,y,\lambda^2);
\label{1-2-1}
\end{equation}
observe, that now $\lambda^2$ is the spectral parameter. Then, $u(x,y,t)$ is a \emph{propagator\/} of the corresponding wave equation and satisfies
\begin{align}
&u_{tt}+\Delta u=0,
\label{1-2-2}\\
&u|_{t=0}=\delta (x-y), \quad u|_{t=0}=0
\label{1-2-3}
\end{align}
(recall that $\Delta$ is a positive Laplacian).

Now we need to construct the solution of (\ref{1-2-2})--(\ref{1-2-3}) and recover $e(x,y,t)$ from (\ref{1-2-1}). However, excluding some special cases, we can construct the solution $u(x,y,t)$ only modulo smooth functions and only for $t:|t|\le T$, where usually $T$ is a small constant. It leads to
\begin{multline}
F_{t\to \tau} \bigl(\bar{\chi}_T(t)u(x,x,t)\bigr) =
T \int \widehat{\bar{\chi}} ((\lambda - \tau)T)\, d_\lambda e(x,x,\lambda^2) =\\
c_0(x) \lambda^{d-1} + c_1(x)\lambda^{d-2}+ O(\lambda^{d-3})
\label{1-2-4}
\end{multline}
where $F$ denotes the Fourier transform, $\bar{\chi}\in \sC_0^\infty (-1,1)$, $\bar{\chi}(0)=1$, $\bar{\chi}'(0)=0$ and $\bar{\chi}_T(t)=\bar{\chi}(t/T)$\,\footnote{\label{foot-8} In fact, there is a complete decomposition.}.

Then using H\"ormander's Tauberian theorem\footnote{\label{foot-9} Which was already known  to Boris~Levitan.}, we can recover
\begin{equation}
e(x,x,\lambda^2)= c_0(x) d^{-1} \lambda^{d} +O(\lambda^{d-1}T^{-1}).
\label{1-2-5}
\end{equation}

To get the remainder estimate $o(\lambda^{d-1})$ instead, we need some extra arguments. First, the asymptotics (\ref{1-2-4}) holds with a cut-off:
\begin{multline}
F_{t\to \tau} \bigl(\bar{\chi}_T(t) (Q_x u)(x,x,t)\bigr) =
T \int \widehat{\bar{\chi}} ((\lambda - \tau)T)
\, d_\lambda (Q_x e)(x,x,\lambda^2) =\\
c_{0Q}(x) \lambda^{d-1} + c_{1Q}(x)\lambda^{d-2}+ O_T(\lambda^{d-3})
\label{1-2-6}
\end{multline}
where $Q_x=Q(x,D_x)$ is a $0$-order pseudo-differewntial operator (acting with respect to $x$ only, before we set $x=y$; and  $T=T_0$ is a small enough constant. Then the Tauberian theory implies that
\begin{multline}
(Q_xe)(x,x,\lambda^2)= c_{0Q}(x) d^{-1} \lambda^{d} +
c_{1Q}(x) (d-1)^{-1} \lambda^{d-1}+\\
O\bigl(\lambda^{d-1} T^{-1}\upmu (\supp (Q))\bigr)+
o_{Q,T} \bigl(\lambda^{d-1}\bigr)
\label{1-2-7}
\end{multline}
where $\upmu=\frac{dxd\xi }{ dg}$ is a natural measure on the \emph{energy level surface\/} $\Sigma=\{(x,\xi):\, g(x,\xi)=1\}$ and we denote by $\supp(Q)$ the support of the symbol $Q(x,\xi)$.

On the other hand, propagation of singularities (which we discuss in more details later) implies that if for any point $(x,\xi)\in \supp(Q)$ geodesics starting there are not periodic with periods $\le T$ then asymptotics (\ref{1-2-6}) and (\ref{1-2-7}) hold with $T$.

Now, under the assumption (\ref{1-1-4}), for any $T\ge T_0$ and $\varepsilon>0$, we can select $Q_1$ and $Q_2$, such that $Q_1+Q_2=I$,
$\upmu(\supp (Q_1))\le \varepsilon$ and for $(x,\xi)\in \supp(Q_2)$ geodesics starting from it are not periodic with periods $\le T$. Then, combining (\ref{1-2-7}) with $Q_1,T_0$ and with $Q_2,T$, we arrive to
\begin{multline}
e (x,x,\lambda^2)= c_{0}(x) d^{-1} \lambda^{d} +
c_{1}(x) (d-1)^{-1} \lambda^{d-1}+\\
O\bigl(\lambda^{d-1} (T^{-1}+\varepsilon)\bigr)+o_{\varepsilon,T} (\lambda^{d-1})
\label{1-2-8}
\end{multline}
with arbitrarily large $T$ and arbitrarily small $\varepsilon>0$ and therefore
\begin{equation}
e(x,x,\lambda^2)= c_0(x) d^{-1} \lambda^{d} +
c_1 (d-1)^{-1}\lambda^{d-1}+O(\lambda^{d-1}T^{-1}).
\label{1-2-9}
\end{equation}
holds. In these settings, $c_1=0$.

More delicate analysis of the propagation of singularities allows under certain very restrictive assumptions to the geodesic flow to boost the remainder estimate to $O(\lambda^{d-1}/\log \lambda)$ and even to $O(\lambda^{d-1-\delta})$ with a sufficiently small exponent $\delta>0$.

\chapter{Local semiclassical spectral asymptotics}
\label{sect-2}

\section{Asymptotics inside the domain}
\label{sect-2-1}

As we mentioned, the approach described above was based on the representation of the solution $u(x,y,t)$ by an oscillatory integral and does not fare well in (i)  domains with  boundaries because of the trajectories tangent to the boundary and (ii) for matrix operators whose principal symbols have eigenvalues of variable multiplicity. Let us describe our main method. We start by discussing  matrix operators  on closed manifolds.

So, let us consider a self-adjoint elliptic matrix operator $A(x,D)$ of order $m$. For simplicity, let us assume that this operator is semibounded from below and we are interested in $N(\lambda)$, the number of eigenvalues not exceeding $\lambda$, as $\lambda\to+\infty$. In other words, we are looking for the number $\N^-(h)$ of negative eigenvalues of the operator
$\lambda^{-1}A(x,D)-I= H(x,hD,h)$ with $h=\lambda^{-1/m}$\,\footnote{\label{foot-10} If operator is not semi-bounded we consider the number of eigenvalues in the interval $(0,\lambda)$ (or $(-\lambda,0)$) which could be reduced to the asymptotics of the number of eigenvalues in the interval $(-1,0)$ (or $(0,1)$) of $H(x,hD,h)$.}.

\subsection{Propagation of singularities}
\label{sect-2-1-1}

Thus, we are now dealing with the semiclassical asymptotics. Therefore, instead of individual functions, we should consider families of functions depending on the \emph{semiclassical parameter\/} $h$\,\footnote{\label{foot-11} Which in quantum mechanics is called Planck constant and usually is denoted by $\hbar$.} and we need a \emph{semiclassical microlocal analysis\/}. We call such family \emph{temperate\/} if $\|u_h\|\le Ch^{-M}$ where $\|\cdot \|$ denotes usual $\sL^2$-norm.

We say that $u\Def u_h$ is \emph{$s$-negligible\/} at
$(\bar{x},\bar{\xi})\in T^*\bR^d$ if there exists a symbol $\phi(x,\xi)$, $\phi(\bar{x},\bar{\xi})=1$ such that $\|\phi(x,hD)u_h\|=O(h^s)$. We call the \emph{wave front set of $u_h$\/} the set of points at which $u_h$ is not negligible  and denote by $\WF^s(u_h)$; this is a closed set. Here, $-\infty<s\le \infty$.

Our first result is rather trivial: if $P=P(x,hD,h)$,
\begin{equation}
WF^s(u) \subset WF^s(Pu) \cup \Char(P)
\label{2-1-1}
\end{equation}
where $\Char(P)=\{(x,\xi), \det P^0 (x,\xi)=0\}$; we call
$P^0 (x,\xi)\Def P(x,\xi,0)$ the \emph{principal symbol\/} of $P$ and $\Char(P)$ the \emph{characteristic set\/} of $L$.

We need to study the propagation of singularities (wave front sets). To do this, we need the following definition:

\begin{definition}\label{def-2-1-1}
Let $P^0$ be a Hermitian matrix. Then $P$ is \emph{microhyperbolic at $(x,\xi)$ in the direction $\ell\in T(T^*\bR^d)$\/}, $|\ell|\asymp 1$  if
\begin{equation}
\langle (\ell P^0)(x,\xi)v,v\rangle \ge \epsilon |v|^2- C|P^0(x,\xi)v|^2
\qquad \forall v
\label{2-1-2}
\end{equation}
with constants $\epsilon, C > 0$\,\footnote{\label{foot-12} Here and below $\ell P^0$ in  is  the action of the vector field $\ell$ upon $P^0$.}.
\end{definition}

Then we have the following statement which can be proven by the \emph{method of the positive commutator\/}:

\begin{theorem}\label{thm-2-1-2}
Let $P=P(x,hD,h)$ be an $h$-pseudodifferential operator with a Hermitian principal symbol. Let $\Omega \Subset T^*\bR^d$ and let $\phi_j\in \sC^\infty $ be real-valued functions such that $P$ is microhyperbolic in $\Omega$ in the directions $\nabla^\# \phi_j$,  $j=1,\ldots, J$ where
$\nabla^\# \phi = \langle (\nabla _\xi \phi),
\nabla _x\rangle -\langle (\nabla _x \phi),\nabla _\xi\rangle $
is the Hamiltonian field generated by $\phi$.

Let $u$ be tempered and suppose that
\begin{align}
\WF^{s+1}(Pu) \cap \Omega \cap \{\phi _1 \le 0 \} \cap \dots \cap
\{\phi _J \le 0\}&=\emptyset,
\label{2-1-3}\\[2pt]
\WF^s(u) \cap \partial \Omega \cap \{\phi _1 \le 0 \} \cap \dots \cap
\{\phi _J \le 0\}&=\emptyset.
\label{2-1-4}\\[2pt]
\shortintertext{Then,}
\WF^s(u) \cap \Omega \cap \{\phi _1 \le 0 \} \cap \dots \cap
\{\phi _J \le 0\}&=\emptyset.
\label{2-1-5}
\end{align}
\end{theorem}

\begin{proof}
This is Theorem~\ref{book_new-thm-2-1-2} from~\cite{futurebook}. See the proof and discussion there.
\end{proof}

The above  theorem immediately implies:

\begin{corollary}\label{cor-2-1-3}
Let $H=H(x,hD,h)$ be an $h$-pseudodifferential operator with a Hermitian principal symbol and let $P=hD_t-H$. Let us assume that
\begin{equation}
|\partial _{x,\xi} H^0 v|\le C_0 |v| + C|(H^0-\bar{\tau})v| \qquad \forall v.
\label{2-1-6}
\end{equation}
Let $u(x,y,t)$ be the Schwartz kernel of $e^{ih^{-1}tH}$.

\begin{enumerate}[label=(\roman*), wide, labelindent=0pt]
\item\label{cor-2-1-3-i}
For a small constant $T^*>0$,
\begin{equation}
WF (u  )\cap \{|t|\le T^*,\ \tau=\bar{\tau}\}\subset
\{|x-y|^2+|\xi+\eta|^2\le (C_0t)^2\}.
\label{2-1-7}
\end{equation}

\item\label{cor-2-1-3-ii}
Assume that $H$ is microhyperbolic in some direction $\ell=\ell(x,\xi)$ at the point $(x,\xi)$ at the energy level $\bar{\tau}$\,\footnote{\label{foot-13} Which means that $H-\bar{\tau}$ is microhyperbolic in the sense of Definition~\ref{def-2-1-1}.}. Then for a small constant $T^*>0$,
\begin{multline}
WF (u  )\cap \{0\le \pm t \le T^*,\ \tau=\bar{\tau}\}\subset\\
\{\pm (\langle \ell_x, x-y\rangle  + \langle \ell_\xi, \xi+\eta\rangle )\ge
\pm \epsilon_0 t\}.
\label{2-1-8}
\end{multline}
\end{enumerate}
\end{corollary}

\begin{proof}
It is sufficient to prove the above statements  for $t\ge 0$. We apply Theorem~\ref{thm-2-1-2} with
\begin{enumerate}[label=(\roman*), wide, labelindent=0pt]
\item\label{pf-2-1-3-i}
$\phi_1= t$ and
$\phi_2= t- C_0^{-1}(|x-\bar{x}|^2+\epsilon^2)^{\frac{1}{2}}+\varepsilon$,
\item\label{pf-2-1-3-ii}
$\phi_1= t$ and
$\phi_2= (\langle \ell_x, x-y\rangle  + \langle \ell_\xi, \xi+\eta\rangle ) -\epsilon_0t +\varepsilon$,
\end{enumerate}
where  $\varepsilon>0$ is arbitrarily small.
\end{proof}

\begin{corollary}\label{cor-2-1-4}
\begin{enumerate}[label=(\roman*), wide, labelindent=0pt]
\item\label{cor-2-1-4-i}
In the framework of Corollary~\ref{cor-2-1-3}\ref{cor-2-1-3-ii} with $\ell=(\ell_x,0)$, the inequality
\begin{equation}
|F_{t\to h^{-1}\tau} \chi_T(t)(Q_{1x} u\,^t\!Q_{2y})(x,x,t)|\le
C_s h^{-d}(h/|t|)^s
\label{2-1-9}
\end{equation}
holds for all $s$,  $\tau:\,|\tau-\bar{\tau}|\le \epsilon$,
$h\le |t|\lesssim T \le T^*$  where $Q_{1x}=Q_1(x,hD_x)$, $Q_{2y}=Q_2(y,hD_y)$ are operators with  compact supports, $^t\!Q_2$ is the dual rather than the adjoint operator and we write it to the right of the function,
$\chi\in \sC_0^\infty ([-1,-\frac{1}{2}]\cup[\frac{1}{2},1])$,
$\chi_T(t)=\chi (t/T)$, and $\epsilon$, $T^*$ are small positive constants.

\item\label{cor-2-1-4-ii}
In particular, we get the estimate  $O(h^s)$ as
$T_*\Def h^{1-\delta}\le |t|\le T\le T^*$.

\item\label{cor-2-1-4-iii}
More generally, when $\ell=(\ell_x,\ell_\xi)$, the same estimates hold for the distribution
$\sigma_{Q_1,Q_2}(t)= \int (Q_{1x} u\,^t\!Q_{2y})(x,x,t)\,dx$.
\end{enumerate}
\end{corollary}

\begin{proof}
\begin{enumerate}[label=(\roman*), wide, labelindent=0pt]
\item\label{pf-2-1-4-i}
If $t\asymp 1$, (\ref{2-1-9}) immediately follows from Corollary~\ref{cor-2-1-3}\ref{cor-2-1-3-ii}. Consider $t\asymp T$ with
$h\le T\le T^*$ and make the rescaling $t\mapsto t/T$, $x\mapsto (x-y)/T$, $h\mapsto h/T$. We arrive to the same estimate (with $T^{-d}(h/T)^s$ in the right-hand expression where the factor $T^{-d}$ is due to the fact that $u(x,y,t)$ is a density with respect to $y$). The transition from $|t|\asymp T$ to
$|t|\lesssim T$ is trivial.
\item\label{pf-2-1-4-ii}
Statement \ref{cor-2-1-4-ii} follows immediately from Statement~\ref{cor-2-1-4-i}.

\item\label{pf-2-1-4-iii}
Statement \ref{cor-2-1-4-iii} follows immediately from Statements~\ref{cor-2-1-4-i} and~\ref{cor-2-1-4-ii} if we apply the metaplectic transformation $(x,\xi)\mapsto (x-B\xi, \xi)$ with a symmetric real matrix $B$.
\end{enumerate}
\vskip-\baselineskip\end{proof}

Therefore under the corresponding microhyperbolicity condition, we can construct
$(Q_{1x} u\,^t\!Q_{2y})(x,x,t)$ or $\sigma_{Q_1,Q_2}(t)$ for $|t|\le T_*$ and then we automatically get it for
$|t|\le T^*$. Since the time interval $|t|\le T_*$ is very short, we are able to apply the successive approximation method.

\subsection{Successive approximation method}
\label{sect-2-1-2}

Let us consider the propagator $u(x,y,t)$. Recall that it satisfies the equations
\begin{align}
&(hD_t - H)u=0,\label{2-1-10}\\
&u|_{t=0}=\updelta (x-y)I
\label{2-1-11}\\
\shortintertext{and therefore,}
&(hD_t - H)u^\pm \,^t\!Q_{2y}= \mp ih \updelta (t)\updelta (x-y)\,^t\!Q_{2y},
\label{2-1-12}
\end{align}
where $u^\pm =u \uptheta (\pm t)$, $\uptheta $ is the Heaviside function, $I$ is the unit matrix, $Q_{1x}=Q_1 (x,hD_x)$, $Q_{2y}=Q_2(y,hD_y)$ have  compact supports, $^t\!Q$ is the \emph{dual operator\/}\footnote{\label{foot-14} I.e. $^t\!Q v= (Q^*v^\dag)^\dag$ where $v^\dag$ is the complex conjugate to $v$.} and we write operators with respect to $y$ on the right from $u$ in accordance with the  notations of  matrix theory.

Then,
\begin{equation}
(hD_t - \bar{H})u^\pm \,^t\!Q_{2y}=H'u \mp ih \updelta(t)\updelta (x-y)\,^t\!Q_{2y} I
\label{2-1-13}
\end{equation}
with $\bar{H}=H(y,hD_x,0)$ obtained from $H$ by freezing $x=y$ and skipping lower order terms and $H'=H'(x,y,hD_x,h)= H-\bar{H}$. Therefore,
\begin{align}
&u^\pm \,^t\!Q_{2y}= \bar{G}^\pm ih H'u^\pm \,^t\!Q_{2y} \pm
ih \bar{G}^\mp \updelta(t)\updelta (x-y)\,^t\!Q_{2y}I.\label{2-1-14}\\
\intertext{Iterating, we conclude that}
&u^\pm \,^t\!Q_{2y}=
\sum_{0\le n \le N-1} (\bar{G}^\pm ih H')^n \bar{u}^\pm \,^t\!Q_{2y}
+ (\bar{G}^\pm ih H')^N  u^\pm \,^t\!Q_{2y},
\label{2-1-15}\\
&\bar{u}^\pm = \mp ih \bar{G}^\pm \updelta(t)\updelta (x-y)\,^t\!Q_{2y}I
\label{2-1-16}
\end{align}
where $\bar{G}^\pm$ is a parametrix of the problem
\begin{equation}
(ihD_t-\bar{H})v =f,\qquad \supp (v)\subset \{\pm t\ge 0\}
\label{2-1-17}
\end{equation}
and $G^\pm$ is a parametrix of the same problem problem albeit for $H$.

Observe that
\begin{multline}
H'= \sum _{1\le |\alpha|+m \le N-1} (x-y)^\alpha h^m R_{\alpha,m} (y,hD_x) + \\
\sum _{|\alpha|+m = N} (x-y)^\alpha h^m R_{\alpha,m} (x,y,hD_x);
\label{2-1-18}
\end{multline}
therefore due to the finite speed of propagation, its norm does not exceed $CT$ as long as we only consider  strips $\Pi^\pm_T\Def \{0\le \pm t\le T\}$. Meanwhile, due to the Duhamel's integral, the operator norms of $G^\pm $ and $\bar{G}^\pm$ from $\sL^2(\Pi^\pm_T)$ to $\sL^2(\Pi^\pm_T)$ do not exceed $Ch^{-1}T$ and therefore each next term in the successive approximations (\ref{2-1-15}) acquires an extra factor $Ch^{-1}T^2= O(h^\delta)$ as long as
$T\le h^{\frac{1}{2}(1+\delta)}$ and the remainder term is $O(h^s)$ if $N$ is large enough.

To calculate the terms of the successive approximations, let us apply $h$-Fourier transform $F_{(x,t)\to h^{-1}(\xi, \tau)}$ with $\xi\in \bR^d$,
$\tau \in \bC_\mp\Def\{\tau: \mp \Im \tau>0\}$ and observe that
$\updelta(t)\updelta (x-y)\mapsto
(2\pi)^{-d-1} e^{-ih^{-1}\langle y,\eta\rangle}$,
$^t\!Q_{2y}$ and $R_{\alpha ,m}$ become multiplication by $Q_2(y,\eta)$ and
$R_{\alpha ,m}(y,\xi)$ respectively, and $\bar{G}^\pm $ becomes  multiplication by $(\tau -H^0(y,\xi))^{-1}$. Meanwhile, $(x_j-y_j)$ becomes
$-ih \partial_{\xi_j} $.

Therefore the right-hand expression of (\ref{2-1-15}) without the remainder term becomes  a sum of terms
$\mp  i \cF_m(y,\xi,\tau)h^{m+1} e^{-ih^{-1}\langle y,\eta\rangle}$ with
$m\ge 0$ and $\cF_m(y,\xi,\tau)$ the sum of terms of the type
\begin{multline}
(\tau -H^0(y,\xi))^{-1} b _*(y,\xi) (\tau -H^0(y,\xi))^{-1} b_* (y,\xi) \cdots
b_* (y,\xi) \times \\(\tau -H^0(y,\xi))^{-1} Q_{2}(y,\eta)
\label{2-1-19}
\end{multline}
with no more than $2m+1$ factors $(\tau -H^0(y,\xi))^{-1}$. Here, the $b_*$ are regular symbols. In particular,
\begin{equation}
\cF_0 (y,\xi,\tau) = (2\pi)^{-d-1} (\tau -H^0(y,\xi))^{-1} Q_{2}(y,\eta).
\label{2-1-20}
\end{equation}
If we add the expressions for $u^+$ and $u^-$ instead of $\cF_m(y,\xi,\tau)$ with $\tau \in \bC_\mp$, we get the distributions
$\bigl( \cF_m(y,\xi,\tau+i0)-\cF_m(y,\xi,\tau-i0)\bigr)$ with $\tau \in \bR$.

Applying the inverse $h$-Fourier transform with respect to $x$, operator $Q_{1x}$, and setting $x=y$, we cancel the factor $e^{-ih^{-1}\langle y,\eta\rangle}$ and gain a factor of $h^{-d}$. Thus we arrive to the Proposition~\ref{prop-2-1-5}\ref{prop-2-1-5-i} below; applying Corollary~\ref{cor-2-1-4}\ref{cor-2-1-4-ii} and \ref{cor-2-1-4-iii}, we arrive to its Statements~\ref{prop-2-1-5-ii} and~\ref{prop-2-1-5-iii}. We also need to use
\begin{equation}
u(x,y,t)= \int e^{ih^{-1}t\tau }\,d_\tau e(x,y,\tau).
\label{2-1-21}
\end{equation}

\begin{proposition}\label{prop-2-1-5}
\begin{enumerate}[label=(\roman*), wide, labelindent=0pt]
\item\label{prop-2-1-5-i}
As $T_*=h^{1-\delta}\le T\le h^{\frac{1}{2}+\delta}$ and
$\bar{\chi} \in \sC_0^\infty ([-1,1])$
\begin{multline}
T\int \widehat{\bar{\chi}}\bigl((\lambda -\tau)Th^{-1}\bigr)\,
d_\tau (Q_{1x}e \,^t\!Q_{2y} )(y,y,\tau)
\sim\\
\sum _{m\ge 0}  h^{-d+m} T\int \widehat{\bar{\chi}}
\bigl((\lambda -\tau)Th^{-1}\bigr)
\kappa'_m(y,\tau) d\tau,
\label{2-1-22}
\end{multline}
where $\widehat{\bar{\chi}}$ is the Fourier transform of $\bar{\chi}$ and
\begin{equation}
\kappa'_m(y)=\int \bigl( \cF_m(y,\xi,\tau+i0)-\cF_m(y,\xi,\tau-i0)\bigr)\,d\eta.
\label{2-1-23}
\end{equation}
\item\label{prop-2-1-5-ii}
If $H$ is microhyperbolic on the energy level $\bar{\tau}$ on $\supp (Q_2)$ in some direction $\ell$ with $\ell_x=0$ then \textup{(\ref{2-1-21})} holds with
$T_*\le T\le T^*$, $|\lambda-\bar{\tau}|\le \epsilon$, where $T^*$ is a small constant.

\item\label{prop-2-1-5-iii}
On the other hand, if $\ell_x\ne 0$, then  \textup{(\ref{2-1-21})} still holds with $T\le T^*$, albeit only after integration with respect to $y$:
\begin{multline}
T\int \widehat{\bar{\chi}}\bigl((\lambda -\tau)Th^{-1}\bigr)\,
d_\tau \bigl(\int(Q_{1x}e \,^t\!Q_{2y} )(y,y,\tau)\,dy\bigr) \sim\\
\sum _{m\ge 0}  h^{-d+m} T\int \widehat{\bar{\chi}}
\bigl((\lambda -\tau)Th^{-1}\bigr)
\varkappa'_m(\tau)\,d\tau
\label{2-1-24}
\end{multline}
with
\begin{equation}
\varkappa'_m(\tau)=\iint \bigl( \cF_m(y,\xi,\tau+i0)-\cF_m(y,\xi,\tau-i0)\bigr)\,dyd\eta.
\label{2-1-25}
\end{equation}
\end{enumerate}
\end{proposition}
For details, proofs and generalizations, see Section~\ref{book_new-sect-4-3} of \cite{futurebook}.

\subsection{Recovering spectral asymptotics}
\label{sect-2-1-3}

Let $\alpha (\tau)$ denote  $(Q_{1x}e \,^t\!Q_{2y} )(y,y,\tau)$ (which may be integrated with respect to $y$) and $\beta(\tau)$ denote the convolution of its derivative $\alpha'(\tau)$ with  $T\widehat{\bar{\chi}}(\tau T/h)$. To recover $\alpha (\tau)$ from $\beta(\tau)$, we apply \emph{Tauberian methods\/}. First of all, we observe that under the corresponding microhyperbolicity condition the distribution $\kappa'_m(y,\tau)$ or $\varkappa_m'(\tau)$ is smooth and the right-hand side expression  of (\ref{2-1-22}) or (\ref{2-1-24}) does not exceed $Ch^{-d+1}$.

Let us take $Q_1=Q_2$; then $\alpha (y,\tau)$ or $\alpha(\tau)$ is a monotone non-decreasing matrix  function of $\tau$. We choose  a H\"ormander function\footnote{\label{foot-15} I.e. a compactly supported function with positive Fourier transform.} $\bar{\chi}(t)$ and estimate the left-hand expressions of (\ref{2-1-22}) or (\ref{2-1-24}) from below by
\begin{equation*}
\epsilon_0 T\bigl(\alpha (\lambda +hT^{-1})-\alpha (\lambda -hT^{-1})\bigr),
\end{equation*}
which implies that
$\bigl(\alpha (\lambda +hT^{-1})-\alpha (\lambda -hT^{-1})\bigr)\le CT^{-1}h^{-d+1}$ and therefore
\begin{equation}
|\alpha (\lambda)-\alpha (\mu)| \le Ch^{-d+1}|\lambda-\mu| +
CT^{-1}h^{-d+1}
\label{2-1-26}
\end{equation}
as $\lambda,\mu\in (\bar{\tau}-\epsilon,\bar{\tau}+\epsilon)$. Then (\ref{2-1-26}) automatically  holds, even if $Q_1$ and $Q_2$ are not necessarily equal.

Further, (\ref{2-1-26}) implies that
\begin{gather}
|\alpha (\lambda)-\alpha (\mu)- h^{-1}\int_\mu^\lambda \beta(\tau)\,d\tau|\le
CT^{-1}h^{-d+1}
\label{2-1-27}\\
\shortintertext{and therefore}
|\int \Bigl(\alpha (\lambda)-\alpha (\mu)-
h^{-1}\int_\mu^\lambda \beta(\tau)\,d\tau\Bigr)\phi(\mu)\,d\mu|\le
CT^{-1}h^{-d+1}
\label{2-1-28}
\end{gather}
if $\bar{\chi}=1$ on $[-\frac{1}{2},\frac{1}{2}]$,
$\lambda,\mu\in (\bar{\tau}-\epsilon,\bar{\tau}+\epsilon)$ and
$\phi\in \sC_0^\infty((\bar{\tau}-\epsilon,\bar{\tau}+\epsilon))$ with $\int\phi(\tau)\,d\tau=1$.

On the other hand, even without the microhyperbolicity condition, our successive approximation construction is not entirely useless. Let us apply $\varphi_L(hD_t-\lambda)$ with $\varphi\in \sC_0^\infty ([-1,1])$ and
$L\ge h^{\frac{1}{2}-\delta}$, and then set $t=0$. We arrive to
\begin{equation}
\int  \varphi ((\tau-\lambda)L^{-1}) \bigl(\alpha '(\tau)-\beta (\tau)\bigr)\,d\tau=O(h^\infty).
\label{2-1-29}
\end{equation}
This allows us to extend (\ref{2-1-28}) to $\phi\in \sC_0^\infty(bR))$ with $\int\phi(\tau)\,d\tau=1$. For full details and generalizations, see Section~\ref{book_new-sect-4-4} of \cite{futurebook}.

Thus, we have proved:

\begin{theorem}\label{thm-2-1-6}
Let $H=H(x,hD,h)$ be a self-adjoint operator. Then,
\begin{enumerate}[label=(\roman*), wide, labelindent=0pt]
\item\label{thm-2-1-6-i}
The following asymptotics holds for $L\ge h^{\frac{1}{2}-\delta}$:
\begin{multline}
\int  \phi ((\tau-\lambda)L^{-1}) \Bigl( d_\tau (Q_{1x}e\,^t\!Q_{2y})(y,y,\tau)-
\sum _{m\ge 0} h^{-d+m}\kappa '(y,\tau)\,d\tau \Bigr)\\ = O(h^\infty).
\label{2-1-30}
\end{multline}
\item\label{thm-2-1-6-ii}
Let $H$ be microhyperbolic on the energy level $\bar{\tau}$ in some direction $\ell$ with $\ell_x=0$. Then for $|\lambda-\bar{\tau}|\le \epsilon$,
\begin{equation}
(Q_{1x}e\,^t\!Q_{2y})(y,y,\lambda)= h^{-d} \kappa_0 (y,\lambda)+O(h^{-d+1})
\label{2-1-31}
\end{equation}
with $\kappa_m(y,\lambda)\Def\int _{-\infty}^\lambda \kappa'_m (y,\tau)\,d\tau$.
\item\label{thm-2-1-6-iii}
Let $H$ be microhyperbolic on the energy level $\bar{\tau}$ in some direction $\ell$. Then for $|\lambda-\bar{\tau}|\le \epsilon$,
\begin{equation}
\int(Q_{1x}e\,^t\!Q_{2y})(y,y,\lambda)\,dy=
h^{-d} \varkappa_0 (\lambda)+ O(h^{-d+1})
\label{2-1-32}
\end{equation}
with $\varkappa_m(y)\Def \int _{-\infty}^\lambda \varkappa'_m (\tau)\,d\tau$.
\item\label{thm-2-1-6-iv}
In particular, it follows from \textup{(\ref{2-1-20})} that
\begin{align}
\kappa _0(\lambda,x) &=
(2\pi)^{-d}\int q_1^0 (x,\xi)\uptheta (\lambda- H^0(x,\xi))q_2^0 (x,\xi)\,d\xi
\label{2-1-33}\\
 \shortintertext{and}
 \varkappa _0(\lambda)&=
(2\pi)^{-d}\int q_1^0 (x,\xi)\uptheta (\lambda- H^0(x,\xi))q_2^0 (x,\xi)\,dxd\xi
\label{2-1-34}
\end{align}
\end{enumerate}
\end{theorem}

\begin{remark}\label{rem-2-1-7}
\begin{enumerate}[label=(\roman*), wide, labelindent=0pt]
\item\label{rem-2-1-7-i}
So far we have assumed that $Q_1,Q_2$ had compactly supported symbols in $(x,\xi)$. Assuming that these symbols are compactly supported with respect to $x$ only, in particular when $Q_1=\psi(x)$, $Q_2=1$, with $\psi\in \sC_0^\infty (X)$, we need to assume that \begin{claim}\label{2-1-35}
$\{\xi:\,\exists x\in X:\, \Spec H^0(x,\xi)\cap
(-\infty,  \lambda+\epsilon_0]\ne \emptyset\}$ is a compact set.
\end{claim}
\item\label{rem-2-1-7-ii}
If we assume only that
\begin{claim}\label{2-1-36}
$\{\xi:\,\exists x\in X:\, \Spec H^0(x,\xi)\cap
(\mu -\epsilon_0,  \lambda+\epsilon_0]\ne \emptyset\}$ is a compact set,
\end{claim}
 instead  of (\ref{2-1-31}) and (\ref{2-1-32}),  we get
\begin{gather}
| (Q_{1x}e\,^t\!Q_{2y})(y,y,\lambda,\mu)- h^{-1} \kappa_0 (y,\lambda,\mu)|\le Ch^{-d+1}
\label{2-1-37}\\
\shortintertext{and}
|  \iint(Q_{1x}e\,^t\!Q_{2y})(y,y,\lambda,\mu)\,dy-
h^{-1} \varkappa_0 (\lambda,\mu)|\le Ch^{-d+1},
\label{2-1-38}
\end{gather}
where $\mu \le \lambda$, $e(x,y,\lambda,\mu)\Def e(x,y,\lambda)-e(x,y,\mu)$,
$\kappa_m(y,\lambda)\Def \int _{\mu}^\lambda \kappa'_m (y,\tau)\,d\tau$,
$\varkappa_m(y,\lambda)\Def \int _{\mu}^\lambda \varkappa'_m (y,\tau)\,d\tau$
and we assume that the corresponding microhyperbolicity assumption is fulfilled on both energy levels $\mu$ and $\lambda$.

\item\label{rem-2-1-7-iii}
If $H^0(x,\xi)$ is an elliptic symbol which is positively homogeneous of degree $m>0$ with respect to $\xi$, then the microhyperbolicity condition is fulfilled with $\ell=(0,\pm \xi)$ on energy levels $\tau\ne 0$. Furthermore, the compactness condition of \ref{rem-2-1-7-ii} is fulfilled, and if $H^0(x,\xi)$ is also positive-definite, then the compactness condition of \ref{rem-2-1-7-i} is also fulfilled.
\end{enumerate}
\end{remark}

\subsection{Second term and dynamics}
\label{sect-2-1-4}

\subsubsection{Propagation of singularities}
\label{sect-2-1-4-1}
To derive two-term asymptotics, one can use the scheme described in Section~\ref{sect-1-2}, albeit one needs to describe the propagation of singularities. For matrix operators, this may be slightly tricky.

Let us introduce the \emph{characteristic symbol\/}
$g(\x,\upxi)\Def \det (\tau- H^0(\x,\xi))$ where $\x=(x_0,x)$, $\upxi=(\xi_0,\xi)$ etc; then $\Char(\xi_0-H(x,\xi))=\{(x,\upxi):\, g(x,\upxi)=0\}$. Let $\xi_0$ be a root of multiplicity $r$ of $g(x,\xi_0,\xi)$; then $g^{(\alpha)}_{(\beta)}(x,\upxi)=0$ for all $\alpha,\beta:|\alpha|+|\beta|<r$. Let us consider the $r$-jet of $g$ at such a point:
\begin{equation}
g_{(x,\upxi)} (y,\upeta)\Def \sum_{\alpha,\beta:|\alpha|+|\beta|<r}
\frac{1}{\alpha!\beta!} g^{(\alpha)}_{(\beta)}(x,\upxi) y^\beta\upeta^\alpha;
\label{2-1-39}
\end{equation}
it is a hyperbolic polynomial with respect to $\eta_0$. Consider its \emph{hyperbolicity cone\/} $K(x,\upxi)$, which is the connected component of $\{(\y;\upeta)\in \bR^{2d+2}:\, g_{(x,\upxi)} (\y,\upeta)\ne 0\}$ containing
$\{(y,\eta):\,\eta_0=1, y=\eta=0\}$ and the \emph{dual hyperbolicity cone\/}
\begin{equation}
K^\#(x,\upxi)=\{(\y',\upeta'): \langle \y',\upeta\rangle - \langle \y,\upeta'\rangle >0\}\subset \{y_0=0\}.
\label{2-1-40}
\end{equation}

\begin{definition}\label{def-2-1-8}
\begin{enumerate}[label=(\roman*), wide, labelindent=0pt]
\item\label{def-2-1-8-i}
An absolutely continuous curve $(\x(t),\upxi(t))$ (with $x_0=t$) is called a \emph{generalized Hamiltonian trajectory\/} if a.e.
\begin{equation}
(1,\frac{dx}{dt}; \frac{d\upxi}{dt})\in K^\#(x,\xi_0,\xi).
\label{2-1-41}
\end{equation}
Note that $\xi_0=\tau$ remains constant  along the trajectory.
\item\label{def-2-1-8-ii}
Let $\cK^\pm (x,\upxi)$ denote the union of all generalized Hamiltonian trajectories issued from $(x,\upxi)$ in the direction of increasing/decreasing $t$.
\end{enumerate}
\end{definition}

If $g=\alpha g_1^r$ where $\alpha\ne 0$ and  $g_1=0\implies \nabla g_1\ne 0$, the generalized Hamiltonian trajectories are just (ordinary) Hamiltonian trajectories of $g_1$ and $\cK^\pm (x,\upxi)$ are just half-trajectories\footnote{\label{foot-16} Since $e^{ih^{-1}tH}$ describes evolution with revert time, time is also reverted along (generalized) Hamiltonian trajectories.}.

The following theorem follows from Theorem~\ref{thm-2-1-2}:

\begin{theorem}\label{thm-2-1-9}
If $u(x,y,t)$ is the Schwartz kernel of $e^{ih^{-1}tH}$, then
\begin{equation}
\WF(u) \subset \{(x,\xi;y,-\eta;t,\tau): \pm t>0, (t,x;\tau,\xi)\in
K^\pm (0, y;\tau,\eta)\}.
\label{2-1-42}
\end{equation}
\end{theorem}

Then, we obtain:

\begin{corollary}\label{cor-2-1-10}
In the framework of Theorem~\ref{thm-2-1-9},
\begin{align}
&\WF (\sigma_{Q_1,Q_2}(t))\subset
\{(t,\tau):\, \exists (x,\xi): \, (t,x;\tau,\xi)\in \cK^\pm (0, x;\tau,\xi)\},
\label{2-1-43}\\
\intertext{and for any $x$,}
&\WF(Q_{1x}u\,^t\!Q_{2y})\subset
\{(t,\tau):\, \exists \xi,\eta: \, (t,x;\tau,\xi)\in \cK^\pm (0, x;\tau,\eta)\}.
\label{2-1-44}
\end{align}
\end{corollary}

\begin{definition}\label{def-2-1-11}
\begin{enumerate}[label=(\roman*), wide, labelindent=0pt]
\item\label{def-2-1-11-i}
A \emph{periodic point\/} is a point $(x,\xi)$ which satisfies $(t,x;\tau,\xi)\in \cK^\pm (0, x;\tau,\xi)$ for some $t\ne 0$.
\item\label{def-2-1-11-ii}
A \emph{loop point\/} is a point $x$ which satisfies $(t,x;\tau,\xi)\in \cK^\pm (0, x;\tau,\eta)$ for some $t\ne 0$, $\xi,\eta$; we call $\eta$  a \emph{loop direction\/}.
\end{enumerate}
\end{definition}

\subsubsection{Application to spectral asymptotics}
\label{sect-2-1-4-2}
Combining Corollary~\ref{cor-2-1-10} with the arguments of Section~\ref{sect-1-2}, we arrive to

\begin{theorem}\label{thm-2-1-12}
\begin{enumerate}[label=(\roman*), wide, labelindent=0pt]
\item\label{thm-2-1-12-ii}
In the framework of Theorem~\ref{thm-2-1-6}\ref{thm-2-1-6-ii} let for some $x$ the set of all loop directions at point $x$ on energy level $\lambda$ have measure $0$\,\footnote{\label{foot-17} There exists a natural measure $\upmu_{\lambda,x}$ on $\{\xi: \det (\lambda-H^0(x,\xi))=0\}$.}. Then,
\begin{equation}
(Q_{1x}e\,^t\!Q_{2y})(y,y,\lambda)= h^{-d} \kappa_0 (y,\lambda)+
h^{1-d} \kappa_1 (y,\lambda)+o(h^{-d+1}).
\label{2-1-45}
\end{equation}
\item\label{thm-2-1-12-iii}
In the framework of Theorem~\ref{thm-2-1-6}\ref{thm-2-1-6-iii}, suppose that the set of all periodic points on energy level $\lambda$ has measure $0$\,\footnote{\label{foot-18} There exists a natural measure $\upmu_\lambda$ on $\{(x,\xi): \det (\lambda-H^0(x,\xi))=0\}$.}. Then,
\begin{equation}
\int(Q_{1x}e\,^t\!Q_{2y})(y,y,\lambda)\,dy=
h^{-d} \varkappa_0 (\lambda)+ h^{1-d} \varkappa_1 (\lambda) + o(h^{-d+1}).
\label{2-1-46}
\end{equation}
\end{enumerate}
\end{theorem}

\begin{remark}\label{rem-2-1-13}
\begin{enumerate}[label=(\roman*), wide, labelindent=0pt]
\item\label{rem-2-1-13-i}
When studying propagation, we can allow $H$ to also depend  on $x_0=t$; for all details and proofs, see Sections~\ref{book_new-sect-2-1} and~\ref{book_new-sect-2-2} of \cite{futurebook}.

\item\label{rem-2-1-13-ii}
Recall that $e(x,y,\lambda)$ is the Schwartz kernel of $\uptheta (\lambda-H)$. We can also consider $e_\nu (x,y,\tau)$ which is the Schwartz kernel of
$(\lambda-H)^\nu _+\Def (\lambda-H)^\nu\uptheta (\lambda-H)$ with $\nu\ge 0$. Then in the Tauberian arguments, $h^{-d}\times (h/T)$ is replaced by
$h^{-d}\times (h/T)^{1+\nu}$ and then in the framework of Theorem~\ref{thm-2-1-6}\ref{thm-2-1-6-ii} and \ref{thm-2-1-6-iii} remainder estimates are $O(h^{-d+1+\nu})$ and in the framework of Theorem~\ref{thm-2-1-6}\ref{thm-2-1-12-ii} and \ref{thm-2-1-12-iii}, the remainder estimates are $o(h^{-d+1+\nu})$; sure, in the asymptotics one should include  all the necessary terms $\kappa_m h^{-d+m}$ or $\varkappa_m h^{-d+m}$\,\footnote{\label{foot-19} Here, we need to assume that $H$ is semi-bounded from below; otherwise some modifications are required.}.

\item\label{rem-2-1-13-iii}
Under more restrictive conditions on Hamiltonian trajectories instead of $T$ an arbitrarily large constant, we can take $T$ depending on $h$\,\footnote{\label{foot-20} Usually these restrictions are $T\le h^{-\delta}$ and $|D\Psi _t(z)|\le h^{-\delta}$ with sufficiently small $\delta>0$.}; see Section~\ref{book_new-sect-2-4} of \cite{futurebook}. Usually, we can take $T=\epsilon |\log h|$ or even $T=h^{-\delta}$.

Then in the remainder estimate, the main term is
\begin{equation*}
C\bigl(\upmu (\Pi_{T,\gamma})h^{-d+1} +  h^{-d+1+\nu}T^{-1-\nu}\bigr),
\end{equation*}
where $\Pi_{T,\gamma}$ is the set of all points $z=(x,\xi)$ (on the given energy level) such that $\dist (\Psi_t(z), z)\le \gamma$ for some
$t\in (\epsilon, T)$  and $\gamma=h^{1/2-\delta'}$. Here, however, we assume that either $H^0$ is scalar or its eigenvalues have constant multiplicities and apply the Heisenberg approach to the long-term evolution.

Then the remainder estimates could be improved to $O(h^{-d+1+\nu}|\log h|^{-1-\nu})$ or even to $O(h^{-d+1+\nu+\delta})$ respectively. As examples, we can consider the geodesic flow on a Riemannian manifold with negative sectional curvature (log case) and the completely integrable non-periodic Hamiltonian flow (power case). For all details and proofs, see Section~\ref{book_new-sect-4-5} of \cite{futurebook}.
\end{enumerate}
\end{remark}

\subsection{Rescaling technique}
\label{sect-2-1-5}

The results we proved are very uniform: as long as we know that operator in question is self-adjoint and that the smoothness and non-degeneracy conditions are fulfilled uniformly in $B(\bar{x},1)$, then all asymptotics are also uniform (as $x\in B(\bar{x},\frac{1}{2})$ or  $\supp(\psi)\subset B(\bar{x},\frac{1}{2})$). Then these results could self-improve.

Here we consider only the Schr\"odinger operator  away from the boundary; but the approach could be generalized for a wider class of operators. For generalizations, details and proofs, see Chapter~\ref{book_new-sect-5} of \cite{futurebook}.

\begin{proposition}\label{prop-2-1-14}
Consider the Schr\"odinger operator. Assume that
$\rho\gamma\ge h$ and in $B(\bar{x},\gamma)\subset X$,
\begin{phantomequation}\label{2-1-47}\end{phantomequation}
\begin{equation}
|\partial ^\alpha g^{jk}|\le c_\alpha \gamma^{-|\alpha|},\qquad
|\partial ^\alpha V|\le c_\alpha \rho^2\gamma^{-|\alpha|}.
\tag*{$\textup{(\ref{2-1-47})}_{1,2}$}\label{2-1-47-*}
\end{equation}
Then,
\begin{enumerate}[label=(\roman*), wide, labelindent=0pt]
\item\label{prop-2-1-14-i}
In $B(\bar{x},\frac{1}{2}\gamma)$,
\begin{equation}
e(x,x,0)\le C\rho^d h^{-d}.
\label{2-1-48}
\end{equation}
\item\label{prop-2-1-14-ii}
If in addition $|V|+|\nabla V|\gamma  \ge \epsilon \rho^2$, then for $\supp(\psi)\subset B(\bar{x},\frac{1}{2}\gamma)$
such that  $|\partial^\alpha \psi |\le c_\alpha \gamma^{-|\alpha|}$,
\begin{equation}
\biggl |\int \bigl(e(x,x,0)-\kappa_0 V_-^{d/2}\bigr)\,dx \biggr | \le C\rho^{d-1}\gamma^{d-1} h^{1-d};
\label{2-1-49}
\end{equation}
\item\label{prop-2-1-14-iii}
If in addition $|V| \ge \epsilon \rho^2$ in $B(\bar{x},\gamma)$ then
\begin{equation}
|e(x,x,0)-\kappa_0 V_-^{d/2}| \le C\rho^{d-1} \gamma^{-1}h^{1-d};
\label{2-1-50}
\end{equation}
\item\label{prop-2-1-14-iv}
If in addition $V \ge \epsilon \rho^2$ in $B(\bar{x},\gamma)$, then for any $s$,
\begin{equation}
|e(x,x,0)| \le C\rho^{d-s} \gamma^{-s}h^{s-d}.
\label{2-1-51}
\end{equation}
\end{enumerate}
\end{proposition}

\begin{proof}
Indeed, we have already proved this in the special case $\rho=\gamma=1$, $h\le 1$. In the general case, we can reduce the problem to the special case by rescaling $x\mapsto x\gamma^{-1}$, $\tau\mapsto \tau\rho^{-2}$ (so we multiply operator by $\rho^{-2}$) and then automatically $h\mapsto \hbar= h\rho^{-1}\gamma^{-1}$. Recall that $e(x,y,\tau)$ is a function with respect to $x$ but a density with respect to $y$ so an extra factor $\gamma^{-d}$ appears in the right-hand expressions.
\end{proof}

Let us assume that the conditions \ref{2-1-47-*} are fulfilled with $\rho=\gamma=1$. We want to get rid of the non-degeneracy assumption $|V|\asymp 1$ in the pointwise asymptotics. Let us introduce the \emph{scaling function\/} $\gamma(x)$ and also $\rho(x)$
\begin{gather}
\gamma (x)= \epsilon |V(x)| + \bar{\gamma}\quad\text{with\ \ } \bar{\gamma}=h^{\frac{2}{3}}, \qquad \rho(x)=\gamma(x)^{\frac{1}{2}}.
\label{2-1-52}\\
\intertext{One can easily see  that}
|\nabla \gamma|\le \frac{1}{2},\qquad \rho\gamma\ge h,
\label{2-1-53}
\end{gather}
\ref{2-1-47-*} are fulfilled and either $|V| \ge \epsilon \rho^2$ or
$\rho\gamma \asymp h$ and therefore (\ref{2-1-50}) holds ($\hbar\asymp 1$ as  $\rho\gamma \asymp h$ and no non-degeneracy condition is needed). Note that for $d\ge 3$, the right-hand expression of (\ref{2-1-50})  is $O(h^{1-d})$ and for $d=1,2$, it is $O(h^{-\frac{2}{3}d})$. So,  we got rid of the non-degeneracy assumption $|V|\asymp 1$, and the remainder estimate deteriorated only for $d=1,2$.

\begin{remark}\label{rem-2-1-15}
\begin{enumerate}[label=(\roman*), wide, labelindent=0pt]
\item\label{rem-2-1-15-i}
We can improve the estimates for $d=1,2$ to $O(h^{\frac{1}{3}-\frac{2}{3}d})$,  but then we will need to add some correction terms first under the assumption
$|V|+|\nabla V|\asymp 1$ and then get rid of it by rescaling; these correction  terms are of  boundary-layer type (near $V=0$) and are $O(h^{-\frac{2}{3}d})$  and are due to short loops. For details, see Theorems~\ref{book_new-thm-5-3-11} and~\ref{book_new-thm-5-3-16} of \cite{futurebook}.

\item\label{rem-2-1-15-ii}
If $d=2$, then under the assumption $|V|+|\nabla V|\asymp 1$, the weight $\rho^{-1}\gamma^{-1}$ is integrable, and we arrive to the local asymptotics with the remainder estimate $O(h^{1-d})$.

\item\label{rem-2-1-15-iii}
We want to get rid of the non-degeneracy assumption $|V|+|\nabla V|\asymp 1$ in the local asymptotics. We can do it with the scaling function
\begin{equation}
\gamma (x)= \epsilon \bigl(|V(x)|+|\nabla V|^2)^{\frac{1}{2}} + \bar{\gamma}\quad\text{with\ \ } \bar{\gamma}=h^{\frac{1}{2}}, \qquad \rho(x)=\gamma(x).
\label{2-1-54}
\end{equation}
Then for $d=2$, we recover remainder the estimate $O(h^{-1})$; while for $d=1$, the remainder estimate $O(h^{-\frac{1}{2}})$ which could be improved further up to $O(1)$ under some extremely weak non-degeneracy assumption or to $O(h^{-\delta})$ with an arbitrarily small exponent $\delta>0$ without it.

\item\label{rem-2-1-15-iv}
If $d\ge 2$, then in the framework of Theorem~\ref{thm-2-1-12}\ref{thm-2-1-12-ii}, we can get rid of  the non-degeneracy assumption as well. This is true for the magnetic Schr\"odinger operator as well if $d\ge 2$; when $d=2$, some modification of the statement is required; see Remark~\ref{book_new-rem-5-3-4}  of \cite{futurebook}.

\item\label{rem-2-1-15-v}
Furthermore, if we consider asymptotics for $\Tr ((\lambda-H)_+^\nu\psi )$ (see Remark~\ref{rem-2-1-13}\ref{rem-2-1-13-ii}) with $\nu>0$ then in the local asymptotics, we get the remainder  estimate $O(h^{1-d+s})$ without any non-degeneracy assumptions. For details,  see Theorem~\ref{book_new-thm-5-3-5}  of \cite{futurebook}.
\end{enumerate}
\end{remark}

\subsection{Operators with periodic trajectories}
\label{sect-2-1-6}

\subsubsection{Preliminary analysis}
\label{sect-2-1-6-1}
Consider a scalar operator $H$. For simplicity, assume that $X$ is a compact closed manifold. Assume that all the Hamiltonian trajectories are periodic (with periods not exceeding $C(\mu)$ on the energy levels $\lambda \le \mu$). Then the period depends only on the energy level and let $T(\lambda)$ be the minimal period such that all trajectories on the energy level $\lambda$ are $T(\lambda)$ periodic\footnote{\label{foot-21} However, there could be \emph{subperiodic trajectories\/}, i.e. trajectories periodic with period $T(\lambda)/p$ with $p=2, 3,\ldots$. It is known that the set $\Lambda_p$ of subperiodic trajectories with subperiod $T(\lambda)/p$ is a union of symplectic submanifolds $\Lambda_{p,r}$ of codimension $2r$.}.

Without any loss of the generality, one can assume that $T(\lambda)=1$. Indeed, we can replace $H$ by $f(H)$ with $f'(\lambda)= 1/T(\lambda)$. Then,
\begin{gather}
e^{\nabla ^\#H^0}=I
\label{2-1-55}\\
\shortintertext{and therefore,}
e^{ih^{-1}H}= e^{i\varepsilon h^{-1}B},
\label{2-1-56}
\end{gather}
where $B=B(x,hD,h)$ is an $h$-pseudo-differential operator which could be selected to commute with $H$, at this point, $\varepsilon=h$. Then,
$H_0 = H-\varepsilon B$ satisfies
\begin{equation}
e^{ih^{-1}H_0}=I \implies \Spec (H_0)\subset 2\pi h \bZ;
\label{2-1-57}
\end{equation}
we call this \emph{quantum periodicity\/} in contrast to the \emph{classical periodicity\/} (\ref{2-1-55}).

We can calculate the multiplicity $\N_{k, h}=O(h^{1-d})$ of the eigenvalue $2\pi hk$ with $k\in \bZ$ modulo $O(h^\infty)$. The formula is rather complicated especially since subperiodic trajectories\footref{foot-21} cause the redistribution of multiplicities between eigenvalues (however, this causes no more than $O(h^{1-d+r})$ error).

We consider $H\Def H_\varepsilon = H_0+\varepsilon B$ as a perturbation of $H_0$ and we assume only that $\varepsilon \ll 1$. If $\varepsilon \le \epsilon_0 h$, the spectrum of $H$ consists of \emph{eigenvalue clusters\/} of the width $C_0\varepsilon$ separated by \emph{spectral gaps\/} of the width $\asymp h$, but if $\varepsilon \ge \epsilon_0 h$, these clusters may overlap.

\subsubsection{Long range evolution}
\label{sect-2-1-6-2}

Consider
\begin{equation}
e^{ih^{-1}t H}= e^{i h^{-1} tH_0} e^{-ih^{-1}t\varepsilon B}=
e^{ih^{-1} t' H_0} e^{ih^{-1}t'' B}
\label{2-1-58}
\end{equation}
with $t''=\varepsilon t$, $t'=t-\lfloor t\rfloor $. We now have a \emph{fast evolution\/} $e^{ih^{-1} t' H_0}$ and a \emph{slow evolution\/}
$e^{ih^{-1}t'' B}$ and both $t',t''$ are bounded as
$|t|\le T^*\Def \varepsilon ^{-1}$. Therefore, we can trace the evolution up to time $T^*$.

Let the following non-degeneracy assumption be fulfilled:
\begin{equation}
|\nabla_{\Sigma (\lambda)} b|\ge \epsilon_0,
\label{2-1-59}
\end{equation}
where $b$ is the principal symbol of $B$,
$\Sigma(\lambda)\Def \{(x,\xi):\, H^0(x,\xi)=\lambda\}$ and
$\nabla_{\Sigma (\lambda)}$ is the gradient along $\Sigma (\lambda)$. Then using our methods, we can prove that
\begin{gather}
|F_{t\to h^{-1}\tau} \chi_T(t)\int u(x,x,t)\psi(x)\,dx|\le
CTh^{1-d} (h/\varepsilon T)^s,
\label{2-1-60}\\
\shortintertext{and therefore}
F_{t\to h^{-1}\tau} \bar{\chi}_T(t)\int u(x,x,t)\psi(x)\,dx|\le
Ch^{1-d}(\varepsilon^{-1}h +1)
\label{2-1-61}
\end{gather}
for $\epsilon_0(\varepsilon ^{-1}h +1) \le T \le \epsilon_0 \varepsilon^{-1}$; recall that $\chi\in \sC_0^\infty([-1,-\frac{1}{2}]\cup [\frac{1}{2},1]$ and
$\bar{\chi}\in \sC_0^\infty([-1,1]$, $\bar{\chi}=1$ on $[-\frac{1}{2},\frac{1}{2}]$.

Then the Tauberian error does not exceed the right-hand expression of (\ref{2-1-61}) multiplied by $T^{*\,-1}\asymp\varepsilon$, i.e. $Ch^{1-d}(\varepsilon+h)$. In the Tauberian expression, we need to take
$T=\epsilon_0(\varepsilon ^{-1}h^{1-\delta} +1)$.

\subsubsection{Calculations}
\label{sect-2-1-6-3}

We can pass from Tauberian expression to a more explicit one. Observe that the contribution to the former are produced only by time intervals
$t\in [n-h^{1-\delta}, n +h^{1-\delta}]$ with $|n|\le T_*$; contribution of the remaining interval will be either negligible (if there are no subperiodic trajectories) or $O(h^{2-d})$ (if such trajectories exist). Such an interval with $n=0$ produces the standard Weyl expression.

Consider $n\ne 0$. Then the contribution of such intervals  lead to a correction term
\begin{equation}
\N_{\corr,Q_1,Q_2}(\lambda) \Def (2\pi)^{-d}h^{1-d}
\int_{\Sigma_\tau} q_1^0 \Upsilon_1 \bigl(h^{-1}(H^0-\varepsilon b) \bigr)
\,d\upmu_\tau q_2^0,
\label{2-1-62}%{6-2-95}
\end{equation}
where $\Upsilon_1(t)= 2\pi \lceil \frac{t}{2\pi}\rceil -t+\frac{1}{2}$.

\begin{theorem}\label{thm-2-1-16}
Under assumptions \textup{(\ref{2-1-55})}, \textup{(\ref{2-1-56})},  \textup{(\ref{2-1-59})} and  ${\varepsilon\ge h^M}$,
\begin{multline}
\int(Q_{1x}e\,^t\!Q_{2y})(y,y,\lambda)\,dy=
h^{-d} \varkappa_{0,Q_1,Q_2} (\lambda)\\
+ h^{1-d} \varkappa_{1,Q_1,Q_2} (\lambda) +
\N_{\corr,Q_1,Q_2}(\lambda) + O\bigl(h^{1-d}(\varepsilon +h)).
\label{2-1-63}
\end{multline}
\end{theorem}

For a more general statement with (\ref{2-1-59}) replaced by a weaker non-degeneration assumption, see Theorem~\ref{book_new-thm-6-2-24} of \cite{futurebook}. Further, we can skip a correction term (\ref{2-1-62}) if $\varepsilon \ge h^{1-\delta}$; while if $h^M\le \varepsilon \le h^{1-\delta}$, this term  is $O(h^{1-d}(h/\varepsilon)^s)$ for $\varepsilon \ge h$ and of magnitude $h^{1-d}$ for $h^M\le \varepsilon \le h$.

For further generalizations, details and proofs, see Sections~\ref{book_new-sect-6-2} and \ref{book_new-sect-6-3} of \cite{futurebook}. For related spectral asymptotics for a family of commuting operators, see Section~\ref{book_new-sect-6-1}  of \cite{futurebook}.

One can also consider  the case when there is a massive set of periodic trajectories, yet non-periodic trajectories exist. For details, see \cite{safarov:diff} and Subsection~\ref{book_new-sect-6-3-7} of \cite{futurebook}.

\section{Boundary value problems}
\label{sect-2-2}

\subsection{Preliminary analysis}
\label{sect-2-2-1}

Let  $X$ be a domain in $\bR^d$ with boundary $\partial X$  and $H$ an $h$-differential matrix operator which is self-adjoint in $\sL^2(X)$ under the $h$-differential boundary conditions. Again, we are interested in the local and  pointwise spectral asymptotics, i.e. those of $\int e(x,x,0)\psi(x)\,dx$ with $\psi \in \sC_0^\infty (B(0,\frac{1}{2}))$ and of $e(x,x,0)$ with $x\in B(0,\frac{1}{2})$.

Assume that in $B(0,1)$, everything is good: $\partial X$ and coefficients of $H$ are smooth, $H$ is $\xi$-microhyperbolic on the energy levels $\lambda_{1,2}$ ($\lambda_1<\lambda_2$) and also $H$ is elliptic as a differential operator, i.e.
\begin{gather}
\1 H(x,\xi) v \1\ge (\epsilon_0 |\xi|^m -C_0) \1 v \1\qquad \forall v\quad\forall x\in B(0,1)\;\forall\xi.
\label{2-2-1}\\
\shortintertext{Then,}
e(x,x,\lambda_1,\lambda_2)= \kappa_0 (x,\lambda_1,\lambda_2) h^{-d} +
O(h^{1-d}\gamma(x)^{-1})
\label{2-2-2}\\
\intertext{for $x\in B(0,\frac{1}{2})$ and $\gamma(x)\ge h$,}
\kappa_0 (x,\lambda_1,\lambda_2)= (2\pi)^{-d} \int
\bigl(\uptheta (\lambda_2-H^0(x,\xi))-
\uptheta (\lambda_1-H^0(x,\xi))\bigr)\,d\xi
\label{2-2-3}
\end{gather}
and $\gamma(x)=\frac{1}{2}\dist (x,\partial X)$.

Indeed, the scaling $x\mapsto (x-y)/\gamma$ and $h\mapsto \hbar /\gamma$ brings us into the framework of Theorem~\ref{thm-2-1-6}\ref{thm-2-1-6-ii} because $\xi$-microhyperbolicity (in  contrast to the $(x,\xi)$-microhyperbolicity) survives such rescaling. Then,
\begin{multline}
\int_{\{x:\,\gamma(x)\ge h\}}  \bigl(e(x,x,\lambda_1,\lambda_2) -h^{-d}\kappa_0 (x,\lambda_1,\lambda_2) \bigr)\psi(x)\,dx\\=
O(h^{1-d}\log h),
\label{2-2-4}
\end{multline}
since $\int_{\{x:\,\gamma(x)\ge h\}} \gamma(x)^{-1}\,dx \asymp |\log h|$.

One can easily show that if the boundary value problem for $H$ is elliptic then
\begin{gather}
e(x,x,\lambda_1,\lambda_2)=O(h^{-d})
\label{2-2-5}
\intertext{and therefore,}
\int   \bigl(e(x,x,\lambda_1,\lambda_2) -h^{-d}\kappa_0 (x,\lambda_1,\lambda_2) \bigr)\psi(x)\,dx=
O(h^{1-d}\log h).
\label{2-2-6}
\end{gather}
To improve this remainder estimate, one needs to improve (\ref{2-2-4}) rather than (\ref{2-2-2}) but to get sharper asymptotics, we need to improve both. We will implement the same scheme as inside  the domain.\enlargethispage{\baselineskip}

\subsection{Propagation of singularities}
\label{sect-2-2-2}

\subsubsection{Toy model: Schr\"odinger operator}
\label{sect-2-2-2-1}
Let us consider the Schr\"odinger operator
\begin{gather}
H\Def h^2\Delta + V(x)
\label{2-2-7}\\
\shortintertext{with the boundary condition}
\bigl(\alpha(x)h\partial_\nu +\beta (x)\bigr)v|_{\partial X}= 0,
\label{2-2-8}
\shortintertext{where}
\Delta = \sum_{j,k} D_j g^{jk}D_k, \qquad \partial_\nu=
\sum_j  g^{j1}\partial_j
\label{2-2-9}
\end{gather}
is derivative in the direction of the inner normal $\nu$ (we assume that $X={\{x:\, x_1>0\}}$ locally\footnote{\label{foot-22} I.e. in intersection with $B(0,1)$.}),
$\alpha$ and $\beta$ are real-valued and do not vanish simultaneously. Without any loss of the generality, we can assume that locally
\begin{equation}
g^{j1}=\updelta_{j1}.
\label{2-2-10}
\end{equation}

First of all, near the boundary, we can study the propagation of singularities using the same scheme as in Subsection~\ref{sect-2-1-1} as long as $\phi_j(x,\xi)=\phi_j(x,\xi')$ do not depend on the component of $\xi$ which is ``normal to the boundary''.  The intuitive way to explain why one needs this is that at reflections, $\xi_1$ changes by a jump.

For the Schr\"odinger operator, it is sufficient for our needs: near \emph{glancing points\/} $(x,\xi')$ (which are points such that $x_1=0$ and the set
$\{\xi_1:\, \, H^0(x',\xi',\xi_1)=\tau\}$ consists of exactly one point), we can apply this method. On the other hand, near other points, we can construct the solution by  traditional methods of oscillatory integrals.

It is convenient to decompose $u(x,y,t)$ into the sum
\begin{equation}
u=u^0(x,y,t)+u^1(x,y,t),
\label{2-2-11}
\end{equation}
where $u^0(x,y,t)$ is a \emph{free space solution\/} (without boundary) which we studied in Subsection~\ref{sect-2-1-1} and $u^1\Def u-u^0$ is a \emph{reflected wave\/}.

Observe that even for the Schr\"odinger operator, we cannot claim that the singularity of $u(x,x,t)$ at $t=0$ is isolated. The reason are \emph{short loops\/} made by trajectories which reflect from the boundary in the normal direction and follow the same path in the opposite direction. However, these short loops affect neither $u(x,x,t)$ at the points of the boundary nor $u(x,x,t)$ integrated in any direction transversal to the boundary (and thus do not affect $\sigma _\psi (t)$ defined below).

Furthermore, they do not affect
$(Q_{1x}u\,^t\!Q_{2y})(x,x,t)$ as long as at least one of operators $Q_j=Q_j(x,hD',hD_t)$ cuts them off. Then we get the estimate (\ref{2-1-9}).  Consider $Q_1=Q_2=1$. Then,  if $V(x)-\lambda > 0$, we get the same estimate  at the point $x\in \partial X$.  On the other hand,  if either
$V(x)-\lambda<0$ or $V(x)=\lambda,\ \nabla_{\partial X} V(x)\ne 0$ (where $\nabla_{\partial X}$ means ``along $\partial X$'') at each point of $\supp(\psi)$, we get the same estimate for $\sigma_\psi (x)=\int u(x,x,t)\,dx$. As usual, $\lambda$ is an energy level.

Moreover, $\sigma^1_\psi(t)= \int u^1(x,x,t)\psi(x)\,dx$ satisfies
\begin{equation}
|F_{t\to h^{-1}\tau} \chi_T(t)\sigma^1_{\psi}(t)|\le C_s h^{1-d}(h/|t|)^s.
\label{2-2-12}
\end{equation}

In contrast to the Dirichlet ($\alpha=0$, $\beta=1$) or Neumann ($\alpha=1$, $\beta=0$) conditions, under the more general boundary condition (\ref{2-2-8}), the classically forbidden level $\lambda$ (i.e. with $\lambda < \inf_{B(0,1)} V$) may be not forbidden after all. Namely, in this zone, the operator $hD_t-H$ is elliptic and we can construct the \emph{Dirichlet-to-Neumann operator\/}
$L: v|_{\partial X}\to h\partial_1 v|_{\partial X}$ as $(hD_t-H)v\equiv 0$. This is an $h$-pseudo-differential operator on $\partial X$ with  principal symbol
\begin{gather}
L^0(x',\xi',\tau) =
-\bigl(V+\sum_{j,k\ge 2} g^{jk}\xi_j\xi_k-\tau\bigr)^{\frac{1}{2}}.
\label{2-2-13}\\
\intertext{Then the boundary condition (\ref{2-2-8}) becomes}
M w \Def (\alpha L+\beta) w\equiv 0,\qquad w=v|_{\partial X}.
\label{2-2-14}
\end{gather}
The energy level $\lambda<V(x)$ is indeed forbidden if the operator $M$ is elliptic as $\tau=\lambda$, i.e. if
$M^0(x',\xi',\lambda)=\alpha L^0(x',\xi',\lambda)+\beta\ne 0$ for all $\xi'$; it happens as either $\alpha^{-1}\beta<0$ or $W\Def V-\alpha^{-2}\beta^2 >\lambda$.
Otherwise, to recover (\ref{2-2-12})\footnote{\label{foot-23} With $d$ replaced by $d-1$.}, we assume that $M$ is either $\xi'$-microhyperbolic  or $(x',\xi')$-microhyperbolic ($W>\lambda$ and $W=\lambda \implies \nabla_{\partial X}W\ne 0$ respectively).

\subsubsection{General operators}
\label{sect-2-2-2-2}

For more general operators and boundary value problems, we use similar arguments albeit not relying upon the representation of $u(x,y,t)$ via oscillatory integrals. It follows from (\ref{2-2-11}) that
\begin{align}
&(hD_t -H)u^{1\,\pm} =0,
\label{2-2-15}\\
&Bu^{1\,\pm}|_{x_1=0}=-Bu^{0\,\pm}|_{x_1=0},
\label{2-2-16}
\end{align}
where as before, $u^{k\,\pm}=u^k\uptheta (\pm t)$, $k=0,1$. Assuming that $H$ satisfies (\ref{2-2-1}), we reduce (\ref{2-2-15})--(\ref{2-2-16})  to the problem
\begin{align}
&\cA U^{1\,\pm}\Def (\cA_0 hD_1 + \cA_1)U^{1\,\pm}\equiv 0,
\label{2-2-17}\\
&\cB U^{1\,\pm}|_{x_1=0}=-\cB U^{0\,\pm}
\label{2-2-18}
\end{align}
with $\cA_k= \cA_k(x,hD',hD_t)$, $\cB= \cB(x,hD',hD_t)$ and $U= S_xu \,^t\!S_y$ with $S=S(x,hD_x,hD_t)$ etc\footnote{\label{foot-24} If $H$ is a $\D\times\D$ matrix operator of order $m$ then $\cA$ and $S$ are $m\D\times m\D$ and $m\D\times\D$ matrix operators, $B$ and $\cB$ are  $\frac{1}{2}m\D\times \D$ and $\frac{1}{2}m\D\times m\D$ matrix operators  respectively.}.

In a neighbourhood of any point $(\bar{x}',\bar{\xi}',\lambda)$,  the operator $\cA$ could be reduced to the block-diagonal form with blocks $\cA_{kj}$ ($k=0,1$, $j=1,\ldots, N$) such that
\begin{enumerate}[label=(\alph*), wide, labelindent=0pt]
\item\label{sect-2-2-2-2-a}
For each $j=1,\ldots, N-1$, the equation
$\det (\cA^0_{0j} \eta + \cA^0_{1j})=0$ has a single real root $\eta_j$ (at the point $(\bar{x}',\bar{\xi}',\lambda)$ only), $\eta_j$ are distinct,  and
\item\label{sect-2-2-2-2-b}
$\cA_{kN}= \left(\begin{smallmatrix} 0 & \cA'_{kN}\\
\cA''_{kN} &0 \end{smallmatrix}\right)$
with $\det (\cA^{\prime\,0}_{0N} \eta + \cA^{\prime\,0}_{1N})=0$ and
$\det (\cA^{\prime\prime\,0}_{0N} \eta + \cA^{\prime\prime\,0}_{1N})=0$
has only roots with $\Im \eta <0$ and with $\Im \eta >0$ respectively.
\end{enumerate}

We can prove a statement similar to Theorem~\ref{thm-2-1-2}, but instead of  functions $\phi_*(x,\xi)$, we now have arrays of functions $\phi_{*j}(x,t,\xi',\tau)$ ($j=1,\ldots, N-1$) coinciding  with $\phi_{*N}(x',t,\xi',\tau)$ as $x_1=0$. Respectively, instead of  microhyperbolicity of the operator in the direction $\ell\in T (T^*(X\times \bR))$, we now have  the \emph{microhyperbolicty of the boundary value problem in the multidirection\/} $(\ell', \nu_1,\ldots ,\nu_{N-1})\in
T (T^*(\partial X\times \bR))\times \bR^{N-1}$; see Definition~\ref{book_new-def-3-1-4} of \cite{futurebook}. It includes the microhyperbolicity of $\cA_j$ in the direction
$(\ell',0, \nu_j)$ for $j=1,\ldots, N-1$ and a condition invoking $\cA_N$ and $\cB$ and generalizing the microhyperbolicity of operator $M$ for the Schr\"odinger operator. Respectively, instead of the microhyperbolicity of an operator in the direction $\nabla^\#\phi_*$, we want the microhyperbolicity in the multidirection
$(\nabla^{'\#}\phi_*,\partial_1 \phi_{*1},\ldots, \partial_1 \phi_{*(N-1)})$.

As a corollary,  under the microhyperbolicity assumption on the energy level $\lambda$, we prove estimates (\ref{2-1-9}) for $\sigma^0_\psi(t)$, $\sigma_\psi(t)$ and (\ref{2-2-12}) for $\sigma^1_\psi(t)$ as $\tau$ is close to $\lambda$. Furthermore,  if the operator $H$ is  elliptic on this energy level then $\sigma^0_\psi(t)$ is negligible and (\ref{2-2-12}) holds for $\sigma^1_\psi(t)$ and $\sigma_\psi(t)$.

For details, proofs and generalizations, see Chapter~\ref{book_new-sect-3} of \cite{futurebook}.

\subsection{Successive approximations method}
\label{sect-2-2-3}

After the (\ref{2-1-9}) and (\ref{2-2-12})-type estimates are established, we can apply the successive approximations method like in Subsection~\ref{sect-2-1-2} but with some modifications: to construct $B u^{0\,\pm}|_{x_1=0}$ and from it to construct $u^{1\,\pm}$, we freeze coefficients in $(y',0)$ rather than in $y$. As a result, we can calculate all terms in the asymptotics and under microhyperbolicity in the multidirection condition, we arrive to the formulae (\ref{2-1-24}) for $e^0 (.,.,\tau)$, $e^1(.,.,\tau)$ and $e(.,.,\tau)$\,\footnote{\label{foot-25} With the obvious definitions of $e^0(.,.,\tau)$ and $e^1(.,.,\tau)$.} with $m\ge 1$ for $e^1(.,.,\tau)$.

The formulae for $\varkappa^1_m(\tau)$ (and thus for $\varkappa_m(\tau)=\varkappa^0_m(\tau)+\varkappa^1_m(\tau)$ are  however  rather complicated and we do not write  them here. For the Schr\"odinger operator with $V=0$ and boundary condition (\ref{2-2-8}), the  calculation of $\varkappa^1_1(\tau)$ is done in Subsection~\ref{book_new-sect-11-9-4} of \cite{futurebook}.

Similar formulae  also hold if we take $x_1=y_1=0$ and integrate over $\partial X$ (but in this case $m\ge 0$ even for $e^1(.,.,\tau)$).

Furthermore, if $\ell'_x=0$ and $\nu_1=\ldots=\nu_{N-1}$ in the condition of microhyperbolicity, we are able to get formulae for $e^0(x,x,\tau)$, $e^1(x,x,\tau)$ and $e(x,x,\tau)$ without setting $x_1=0$ and without integrating but $e^1(x,x,\tau)$ is a boundary-layer type term.
\enlargethispage{\baselineskip}

For details and proofs, see Section~\ref{book_new-sect-7-2} of \cite{futurebook}.

\subsection{Recovering spectral asymptotics}
\label{sect-2-2-4}

Repeating the arguments of Subsection~\ref{sect-2-1-3}, we can recover the local spectral asymptotics:

\begin{theorem}\label{thm-2-2-1}
\begin{enumerate}[label=(\roman*), wide, labelindent=0pt]
\item\label{thm-2-2-1-i}
Let  an operator $H$ be microhyperbolic on $\supp(\psi)$ on the energy levels $\lambda_1$ and $\lambda_2$ ($\lambda_1<\lambda_2$) and the boundary value $(H,B)$ problem be microhyperbolic  on $\supp(\psi)\cap \partial X$ on these energy levels. Then,
\begin{multline}
\int_X e(y,y,\lambda_1,\lambda_2)\psi(y)\,dy \\
= h^{-d} \int_X \kappa_0 (y,\lambda_1,\lambda_2)\psi(y)\,dy+ O(h^{-d+1}).
\label{2-2-19}
\end{multline}
\item\label{thm-2-2-1-ii}
Suppose that an operator $H$ is elliptic on $\supp(\psi)$ on the energy levels $\lambda_1$ and $\lambda_2$ ($\lambda_1<\lambda_2$)\footnote{\label{foot-26} Then, it is elliptic on all energy levels $\tau\in [\lambda_1, \lambda_2]$.} and the boundary value $(H,B)$ problem is microhyperbolic  on
$\supp(\psi)\cap \partial X$ on these energy levels. Then,
\begin{multline}
\int_X e(y,y,\lambda_1,\lambda_2)\psi(y)\,dy\\
= h^{1-d} \int_X \kappa_1 (y,\lambda_1,\lambda_2)\psi(y)\,dy+ O(h^{-d+2}).
\label{2-2-20}
\end{multline}
\end{enumerate}
\end{theorem}

On the other hand, for the Schr\"odinger operator, we can calculate the contributions of near normal trajectories explicitly and then we arrive to:

\begin{theorem}\label{thm-2-2-2}
Let $(H,B)$ be the Schr\"odinger operator \textup{(\ref{2-2-7})}--\textup{(\ref{2-2-8})} and let $|V|\ne \lambda$ on $\supp(\psi)$. Then,
\begin{equation}
e(y,y,\lambda)=
h^{-d}\bigl(\kappa_0(x,\lambda) +
\cQ(x',\lambda; h^{-1}x_1) \bigr) + O(h^{-d+1})
\label{2-2-21}
\end{equation}
where $\cQ$ depends on the ``normal variables'' $(x',\lambda)$ and a ``fast variable'' $s=h^{-1}x_1$ and decays as $O(s^{-{d+1}/2})$ as $s\to +\infty$. Here, $x_1=\dist(x,\partial X)$.
\end{theorem}

For details, exact statement and proofs, see Section~\ref{book_new-sect-8-1} of \cite{futurebook}.

\subsection{Second term and dynamics}
\label{sect-2-2-5}

As in Subsection~\ref{sect-2-1-4}, we can improve our asymptotics under certain conditions to the dynamics of propagation of singularities. However, in the case that the manifold has a non-empty boundary, propagation becomes really complicated. For Schr\"odinger operators, we can prove that singularities propagate along Hamiltonian billiards unless they ``behave badly'' that is become tangent to $\partial X$ at some point or make an infinite number of reflections in finite time. However, the measure of \emph{dead-end points\/}\footnote{\label{foot-27} I.e. points $z\in \Sigma(\lambda)$ the billiard passing through which behaves badly.}  is $0$.

Thus, applying the arguments of Section~\ref{sect-1-2} we arrive to

\begin{theorem}\label{thm-2-2-3}
Let $d\ge 2$, $|V-\lambda|+|\nabla V|\ne 0$ on $\supp(\psi)$ and
$|V-\lambda|+|\nabla_{\partial X}  V|\ne 0$ on $\supp(\psi)\cap \partial X$. Further, assume that the measure of periodic Hamiltonian billiards passing through points of $\{H^{0}(x,\xi)=0\}\cap\supp(\psi)$ is $0$\,\footnote{\label{foot-28} There is a natural measure $dxd\xi:dH^0$.}. Then,
\begin{equation}
\int e(y,y,\lambda)\psi(y)\,dy=
h^{-d} \int \kappa_0 (y,\lambda)\psi(y)\,dy+ o(h^{-d+1}).
\label{2-2-22}
\end{equation}
\end{theorem}

\begin{remark}\label{rem-2-2-4}
If we are interested in the propagation of singularities without applications to spectral asymptotics, the answer is ``singularities propagate along the generalized Hamiltonian billiards'' (see Definition~\ref{book_new-def-3-2-2} in \cite{futurebook}).
\end{remark}

One can easily show:

\begin{theorem}\label{thm-2-2-5}
Let $d\ge 3$. Assume that we are in the framework of Theorem~\ref{thm-2-2-1}\ref{thm-2-2-1-ii}. Further  assume that the set of periodic trajectories of the Schr\"odinger operator on $\partial X$ with potential $W$ introduced after \textup{(\ref{2-2-14})} has measure $0$. Then,
\begin{equation}
\int e(y,y,\lambda)\psi(y)\,dy=
h^{1-d}  \varkappa_{1,\psi} (\lambda)+
h^{2-d} \varkappa_{2,\psi} (\lambda)+o(h^{-d+2}).
\label{2-2-23}
\end{equation}
\end{theorem}

\begin{remark}\label{rem-2-2-6}
Analysis becomes much more complicated for more general operators even if we assume that the inner propagation is simple. For example, if the operator in question is essentially a collection of $m$ Schr\"odinger operators intertwined through boundary conditions then every incidence ray after reflection generates up to $m$ reflected rays and we have \emph{branching Hamiltonian billiards\/}. Here, a \emph{dead-end point\/} is a point $z\in \Sigma(\lambda)$ such that \emph{some of the branches\/} behave badly and a \emph{periodic point\/}  is a point
$z\in \Sigma(\lambda)$ such that \emph{some of the branches\/} return to it.

Assume that the sets of all periodic points and all dead-end points on the energy level $\Sigma(\lambda)$ have measure $0$ (as shown in \cite{safarov:pdo}, the set of all dead-end points may have positive measure). Then, the two-term asymptotics could be recovered. However, the   investigation of branching Hamiltonian billiards is a rather daunting task.
\end{remark}

\subsection{Rescaling technique}
\label{sect-2-2-6}

The rescaling technique could be applied near $\partial X$ as well. Assume that $\lambda=0$. Then to get rid of the non-degeneracy assumption $V(x)\le -\epsilon$, we use scaling functions $\gamma(x)$ and $\rho(x)$ as in Subsection~\ref{sect-2-1-5}. It may happen that $B(x,\gamma(x))\subset X$  or it may  happen that $B(x,\gamma(x))$ intersects $\partial X$. In the former case, we are obviously done and in the latter case we are done as well because in the condition (\ref{2-2-8}) we scale $\alpha\mapsto \alpha \rho\nu$, $\beta \mapsto \beta\nu$ where $\nu>0$ is a parameter of our choice. Thus, in the pointwise asymptotics, we can get rid of this assumption for $d\ge 3$, and in the local asymptotics for $d\ge 2$ assuming that $|V|+|\nabla V|\asymp 1$ because the total measure of the balls of radii $\le \gamma$ which intersect $\partial X$ is $O(\gamma)$.
For details, exact statements and proofs, see Section~\ref{book_new-sect-8-2} of \cite{futurebook}.

\subsection{Operators with periodic billiards}
\label{sect-2-2-7}

\subsubsection{Simple billiards}
\label{sect-2-2-7-1}
Consider an operator on a manifold with boundary. Assume first that all the billiard trajectories (on energy levels close to $\lambda$) are \emph{simple\/} (i.e. without branching) and  periodic with a period bounded from above; then the period depends only on the energy level. Example: the Laplace-Beltrami operator on the semisphere. Under some non-degeneracy assumptions similar to (\ref{2-1-59}), we can derive asymptotics similar to (\ref{2-1-63}) but with two major differences:

\begin{enumerate}[label=(\roman*), wide, labelindent=0pt]
\item\label{sect-2-2-7-1-i}
We assume that $\varepsilon \asymp  h$ and recover remainder estimate only $O(h^{1-d+\delta})$;  it is still good enough to have the second term of the non-standard type.

\item\label{sect-2-2-7-1-ii}
We can consider $b(x,\xi)$ (which is invariant with respect to the  Hamiltonian billard flow) as a phase shift for one period. Now, however, it could be a result not only of the quantum drift as in Subsection~\ref{sect-2-1-6}, but also of an instant change of phase at the moment of the reflection.
\end{enumerate}

For exact statements, details and proofs, see Subsection~\ref{book_new-sect-8-3-2} of \cite{futurebook}.

\subsubsection{Branching billiards with ``scattering''}
\label{sect-2-2-7-2}

We now assume  that the billiard branches but only one (``main'') branch is typically periodic. For example, consider two Laplace-Beltrami operators intertwined through boundary conditions: one of them is an operator on the semisphere $X_1$ and another on the disk $X_2$ with $\partial X_1$ and $\partial X_2$ glued together. Then all billiards on $X_1$ are periodic but there exist nowhere dense sets $\Lambda_j(\lambda)$ of measure $0$, such  that the billiards passing through $\Sigma_j(\lambda)\setminus \Lambda_j(\lambda)$  and containing at least one segment in $X_2$  are not periodic. Assume also that the boundary conditions guarantee that at reflection, the ``observable'' part of energy escapes into $X_2$. Then to recover the sharp remainder estimates, we do not need a phase shift because for time $T\gg 1$, we have
\begin{equation}
T |F_{t\to h^{-1}\tau}\bar{\chi}_T(t)\, d_\tau e (y,y,\tau)\,dy|\le
C_0h^{1-d} \sum_{|n|\le T} q^n + o_{T} (h^{1-d}),
\label{2-2-24}
\end{equation}
where $q\le 1$ estimates from above the ``portion of energy'' which goes back to $X_1$ at each reflection; if $q<1$, as we have assumed the right-hand expression does not exceed $C_1h^{1-d}+ o_T(h^{1-d})$ and we recover asymptotics similar to (\ref{2-1-63})  with the remainder estimate $o(h^{1-d})$.

For exact statements, details and proofs, see Subsection~\ref{book_new-sect-8-3-3} of \cite{futurebook}.

\subsubsection{Two periodic billiards}
\label{sect-2-2-7-3}

We can also consider  the case when the billiards flows in $X_1$ and $X_2$ are both periodic but ``magic'' happens at reflections. For exact statements, details and proofs, see Subsection~\ref{book_new-sect-8-3-4} of \cite{futurebook}.

\chapter{Global asymptotics}
\label{sect-3}

In this section, we consider global spectral asymptotics. Here we are mainly interested in the asymptotics with respect to the spectral parameter $\lambda$. We consider mainly examples.\enlargethispage{2\baselineskip}

\section{Weyl asymptotics}
\label{sect-3-1}

\subsection{Regular theory}
\label{sect-3-1-1}

We start from examples in which we apply only the results of the previous Chapter~\ref{sect-2} which may be combined with  \emph{Birman-Schwinger principle\/} and the rescaling technique.

\subsubsection{Simple results}
\label{sect-3-1-1-1}

\begin{example}\label{example-3-1-1}
Consider a self-adjoint  operator $A$ with domain $\fD(A)=\{u: \, Bu|_{\partial X}=0\}$.  We assume that $A$ is elliptic and the boundary value problem $(A,B)$ is elliptic as well.
\begin{enumerate}[label=(\roman*), wide, labelindent=0pt]
\item\label{example-3-1-1-i}
 We are interested in $\N(0,\lambda)$, the number of eigenvalues of $A$ in $[0,\lambda)$. Instead we consider $\N(\lambda/2,\lambda)$, which is obviously equal to $\N_h(\frac{1}{2},1)$,  the number of eigenvalues of
$A_h=\lambda ^{-1}A$ that lie in $[\frac{1}{2},1)$, with $h=\lambda^{-1/m}$ where $m$ is the order of $A$. In fact, more is true: the principal symbols of semiclassical operators $A_h$ and $B_h$  coincide with the senior symbols of $A$ and $B$. Then the microhyperbolicity conditions are satisfied and the semiclassical asymptotics with the remainder estimate $O(h^{1-d})$ hold which could be improved to two-term asymptotics under our standard non-periodicity condition. As a result, we obtain
\begin{align}
&\N(0,\lambda)=\varkappa_0 \lambda^{\frac{d}{m}} + O(\lambda^{\frac{d-1}{m}})
\label{3-1-1}\\
\shortintertext{and}
&\N(0,\lambda)=\varkappa_0 \lambda^{\frac{d}{m}} +
\varkappa_1 \lambda^{\frac{d-1}{m}}+ o(\lambda^{\frac{d-1}{m}}),
\label{3-1-2}
\end{align}
as $\lambda\to +\infty$ in the general case and under the standard non-periodicity condition respectively. Here,
\begin{equation}
\varkappa_0 = (2\pi)^{-d} \iint \mathbf{n}(x,\xi)\,dx d\xi
\label{3-1-3}
\end{equation}
where $\mathbf{n}(x,\xi)$ is the number of eigenvalues of $A^0(x,\xi)$ in $(0,1)$ and $m=m_A$ is the order of $A$.

\item\label{example-3-1-1-ii}
Suppose that $A_B$ is positive definite (then $m_A\ge 2$) and $V$  is an operator of the order $m_B<m_A$, symmetric under the same boundary conditions. We are interested in $\N(0,\lambda)$, the number of eigenvalues of $VA^{-1}_B$ in $(\lambda^{-1},\infty)$. Using the Birman-Schwinger principle, we can again reduce the problem to the semiclassical one with $H= h^{m_A}A - h^{m_V}V$, $h=\lambda^{-1/m}$, $m=m_A-m_V$. The microhyperbolicity condition is fulfilled automatically unless $\xi =0$ and $V^0(x,\xi)$ is degenerate. Still under certain appropriate assumptions about $V^0$, we can ensure microhyperbolicity (for $m_B=0,1$ only). Then (\ref{3-1-1}) and (\ref{3-1-2}) (the latter under standard non-periodicity condition) hold with $\mathbf{n}(x,\xi)$  the number of eigenvalues of $V^0(x,\xi)(A^0(x,\xi)^{-1})$ in $(1,\infty)$.

\item\label{example-3-1-1-iii}
Alternatively, we can consider the case when $V$ is positively defined (and $A_B$ may be not).

\item\label{example-3-1-1-iv}
For scalar operators, one can replace microhyperbolicity by a weaker non-degeneracy assumption. Furthermore, without any non-degeneracy assumption we arrive to one-term asymptotics with the remainder estimate $O(\lambda^{(d-1+\delta)/m})$.

\item\label{example-3-1-1-v}
Also one can consider operators whose  all Hamiltonian trajectories are periodic; in this case the oscillatory correction term appears.

\item\label{example-3-1-1-vi}
Suppose the operator $A_B$ has negative definite principal symbol but $A_B$ is not semi-bounded from above and $V$ is positive definite. Then  instead of (\ref{3-1-1}) or (\ref{3-1-2}), we arrive to
\begin{align}
&\N(0,\lambda)=\varkappa_1 \lambda^{\frac{d-1}{m}} + O(\lambda^{\frac{d-2}{m}})
\label{3-1-4}\\
\shortintertext{and}
&\N(0,\lambda)=\varkappa_1 \lambda^{\frac{d-1}{m}} +
\varkappa_2 \lambda^{\frac{d-2}{m}}+ o(\lambda^{\frac{d-2}{m}}),
\label{3-1-5}
\end{align}
(the latter under an appropriate non-periodicity assumption).
\end{enumerate}
\end{example}

\subsubsection{Fractional Laplacians}
\label{sect-3-1-1-2}

The fractional Laplacian $\Lambda_{m,X}$ appears in the theory of stochastic processes. For $m>0$, it is defined first on $\bR^d$ as $\Delta^{m/2}$, and in a domain $X\subset \bR^d$, it is defined as
$\Lambda_{m,X} u=R_X\Delta^{m/2}(\uptheta_X u)$ where $R_X$ is the restriction to $X$ and $\uptheta_X$ is the characteristic function of $X$. It differs from the $m/2$-th power of the Dirichlet Laplacian in $X$ and for $m\notin 2\bZ$, it does not belong to the Boutet de Monvel's algebra. In particular, even if $X$ is a bounded domain with $\partial X\in \sC^\infty$ and $u\in \sC^\infty(\bar{X})$, $\Lambda_{m,X} u $ does not necessarily belong to $\sC^\infty(\bar{X})$ (smoothness is violated in the direction normal to $\partial X$).

Then the standard Weyl asymptotics (\ref{3-1-1}) and (\ref{3-1-2}) hold
(the latter under standard non-periodicity condition) with the standard coefficient $\varkappa_0=(2\pi)^{-d}\omega_{d-1} \vol_{d}(X)$ and with
\begin{equation}
\varkappa_{1,m}= (2\pi)^{1-d}\omega_{d-1} \sigma_{m}\vol_{d-1}(\partial X),
\label{3-1-6}
\end{equation}
\begin{multline}
\sigma_m  =\\
=\frac{d-1}{m}
\iint_1^\infty \tau  ^{-(d-1)/m-1}
\Bigl(\mathbf{e}_m (x_1,x_1,\tau) - \pi ^{-1}(\tau-1)^{1/m}\Bigr)\,dx_1 d\tau
\label{3-1-7}
\end{multline}
where $\mathbf{e}_{m}(x_1,y_1,\tau)$ is the Schwartz kernel of the spectral projector of operator
\begin{equation}
\mathbf{a}_m = ((  D_x^2+1)^{m/2})_\D
\label{3-1-8}
\end{equation}
on $\bR^+$. To prove this, we need to redo some analysis of Chapter~\ref{sect-2}. While tangent rays are treated exactly as for the ordinary Laplacian, normal rays require some extra work. However, we can show that the singularities coming along transversal rays do not stall at the boundary but reflect according the standard law. For  exact statements, details and proofs, see Section~\ref{book_new-sect-8-5} of \cite{futurebook}.

\subsubsection{Semiclassical Dirichlet-to-Neumann operator}
\label{sect-3-1-1-3}

Consider the Laplacian $\Delta$ in $X$. Assuming that $\lambda$ is not an eigenvalue of $\Delta_\D$, we can introduce the Dirichlet-to-Neumann operator $L_\lambda:v\mapsto \lambda^{-\frac{1}{2}}\partial_\nu u|_{\partial X}$ where $u$ is defined as $(\Delta -\lambda)u=0$, $u|_{\partial X}=v$ and $\nu$ is the  inner unit normal. Here, $L_\lambda$ is a self-adjoint operator and we are interested in $\N_\lambda (a_1,a_2)$, the number of its eigenvalues in the interval $[a_1,a_2)$. Due to the Birman-Schwinger principle, it is equal to $N^-_h(a_1)-N^-_h(a_2)$ where $N^-_h(a)$ is the number of the negative eigenvalues of $h^2\Delta-1$ under the boundary condition $(h\partial_\nu -a)u|_{\partial X}=0$ and then we arrive to
\begin{align}
&\N_\lambda (a_1,a_2)=O(\lambda^{\frac{d-1}{m}})
\label{3-1-9}\\
\shortintertext{and}
&\N_\lambda (a_1,a_2)=\varkappa_1 (a_1,a_2)\lambda^{\frac{d-1}{2}}
+ o(\lambda^{\frac{d-1}{2}})
\label{3-1-10}
\end{align}
(the latter under a standard non-periodicity condition). For  exact statements, details and proofs, see Section~\ref{book_new-sect-11-9} of \cite{futurebook}.

\subsubsection{Rescaling technique}
\label{sect-3-1-1-4}

We are interested in the asymptotics of either
\begin{align}
&\N ^- (\lambda)= \int  e(x,x,\lambda)\,dx
\label{3-1-11}\\
\shortintertext{or}
&\N (\lambda_1,\lambda_2)= \int  e(x,x,\lambda_1,\lambda_2)\,dx
&&\text{with\ \ } \lambda_1<\lambda_2:
\label{3-1-12}
\end{align}
with respect to either the spectral parameter(s), or semiclassical parameter(s), or some other parameter(s). We assume that there exist \emph{scaling functions\/} $\gamma(x)$ and $\rho(x)$ satisfying
\begin{phantomequation}\label{3-1-13}\end{phantomequation}
\begin{equation}
|\nabla \gamma|\le \frac{1}{2},\qquad
|x-y|\le \gamma(y)\implies c^{-1}\le \rho(x)/\rho(y)\le c,
\tag*{$\textup{(\ref*{3-1-13})}_{1,2}$}\label{3-1-13-*}
\end{equation}
such that after rescaling $x\mapsto x/\gamma(y)$ and
$\xi\mapsto \xi/\rho(y)$ in $B(y,\gamma(y))$, we find ourselves in the framework of the previous chapter with an \emph{effective semiclassical parameter\/} $\hbar \le 1$\,\footnote{\label{foot-29} In purely semiclassical settings, $\hbar=h/\rho\gamma$ and we assume $\rho\gamma \ge h$.}.

To avoid non-degeneracy assumptions, we consider only the Schr\"odinger operator (\ref{2-2-7}) in $\bR^d$, assuming that $g^{jk}=g^{kj}$,
\begin{gather}
|\nabla ^\alpha g^{jk}|\le c_\alpha \gamma ^{-|\alpha|},\qquad
|\nabla^\alpha V|\le c_\alpha \rho^2\gamma^{-|\alpha|}\label{3-1-14}\\
\shortintertext{and}
\sum _{j,k} g^{jk}\xi_j\xi_k \ge \epsilon_0 |\xi|^2 \qquad \forall x,\xi.
\label{3-1-15}
\end{gather}

In the examples below, $h\asymp 1$.

\begin{example}\label{example-3-1-2}
\begin{enumerate}[label=(\roman*), wide, labelindent=0pt]
\item\label{example-3-1-2-i}
Suppose the conditions (\ref{3-1-14}), (\ref{3-1-15}) are fulfilled with
$\gamma (x)= \frac{1}{2}(|x|+1)$  and $\rho(x)=|x|^m$, $m>0$. Further, assume that the \emph{coercivity condition}
\begin{equation}
V(x)\ge \epsilon_0 \rho^2
\label{3-1-16}
\end{equation}
holds for $|x|\ge c_0$. Then if $|x|\le C\lambda^{1/2m}$, for the operator $H-\lambda$, we can use $\rho_\lambda (x)=\lambda^{1/2}$ and then the contribution of the ball $B(x,\gamma(x))$ to the remainder does not exceed
$C\lambda ^{(d-1)/2}\gamma^{d-1}(x)$; summation over these balls results in $O\bigl(\lambda^{(d-1)(m+1)/2m}\bigr)$.

On the other hand, if $|x|\le C\lambda^{\frac{1}{2m}}$, for the operator $H-\lambda$ we can use $\rho(x)=\gamma^{m}(x)$ but  there the ellipticity condition is fulfilled and  then the contribution of the ball $B(x,\gamma(x))$ to the remainder does not exceed $C\gamma^{-s}$;  summation over these balls results in $o\bigl(\lambda^{(d-1)(m+1)/2m}\bigr)$.  Then we arrive to
\begin{equation}
\N(\lambda)= c_0h^{-d}\int (\lambda-V(x))_+^{\frac{d}{2}} +
O\bigl(\lambda^{(d-1)(m+1)/2m}\bigr)
\label{3-1-17}
\end{equation}
as $\lambda\to +\infty$. Obviously the main part of the asymptotics is $\asymp \lambda^{d(m+1)/2m}$.

\item\label{example-3-1-2-ii}
Suppose instead $0>m >-1$. We are interested in its eigenvalues tending to the bottom of the continuous spectrum (which is $0$) from below. We no longer require the assumption (\ref{3-1-16}).

We use the same $\gamma(x)$ but now $\rho_\lambda (x)= \gamma(x)^{m}$ for $|x|\le C|\lambda|^{1/2m}$. Then the contribution of the ball $B(x,\gamma(x))$ to the remainder does not exceed $C\gamma (x)^{(d-1)(m+1)}$; summation over these balls results in $O\bigl(|\lambda|^{(d-1)(m+1)/2m}\bigr)$.

On the other hand, if $|x|\ge C|\lambda|^{1/2m}$, for the operator $H-\lambda$ we can use $\rho_\lambda(x)=|\lambda|^{\frac{1}{2}}$, but there the ellipticity condition is fulfilled and  then the contribution of the ball $B(x,\gamma(x))$ to the remainder does not exceed $C|\lambda|^{-s}\gamma^{d- s}$; summation over these balls results in $o\bigl(|\lambda|^{(d-1)(m+1)/2m}\bigr)$. Then we arrive to asymptotics (\ref{3-1-17}) again as $\lambda\to -0$.

Obviously the main part of the asymptotics is $O( |\lambda|^{d(m+1)/2m})$ and under the assumption $V(x)\le -\epsilon \rho(x)^2$, in some cone it is
$\asymp |\lambda|^{d(m+1)/2m}$.

\item\label{example-3-1-2-iii}
In both cases \ref{example-3-1-2-i} and \ref{example-3-1-2-ii}, the main contribution to the remainder is delivered by the zone
$\{\varepsilon < |x| |\lambda|^{-1/2m} < \varepsilon^{-1}\}$ and assuming that $g^{jk}(x)$ and $V(x)$ stabilize as $|x|\to +\infty$ to $g^{jk0}(x)$  and $V^0(x)$, positively homogeneous functions of degrees $0$ and $2m$
respectively, and that the set of periodic trajectories of the Hamiltonian $\sum_{j,k}g^{jk}(x)\xi_j\xi_k+V^0(x)$ on  energy level $1$ in \ref{example-3-1-2-i} or $-1$ in \ref{example-3-1-2-ii} has measure $0$, we can improve the remainder estimates to $o\bigl(|\lambda|^{(d-1)(m+1)/2m}\bigr)$.
\end{enumerate}
\end{example}

\begin{example}\label{example-3-1-3}
Consider the Dirac operator
\begin{equation}
H=\sum_{1\le j \le d} \upsigma_j D_j + M\sigma_0 + V(x),
\label{3-1-18}
\end{equation}
where $\upsigma_j$ ($j=0,\ldots,d$) are Pauli matrices in the corresponding dimension and $M>0$. Let $V(x)\to 0$ as $|x|\to \infty$. Then the essential spectrum of $H$ is $(-\infty,-M]\cup[M,\infty)$ and for $V$ as in Example~\ref{example-3-1-2}\ref{example-3-1-2-ii}, we can get  similar results for the  asymptotics of eigenvalues tending to $M-0$ or $-M+0$: so instead of $\N(\lambda)$, we consider $\N(0,M-\eta)$ or $\N(M+\eta,0)$ with $\eta\to +0$.
\end{example}

\begin{example}\label{example-3-1-4}
Consider the Schr\"odinger operator, either in a bounded domain $X\ni 0$ or in $\bR^d$ like in Example~\ref{example-3-1-2}\ref{example-3-1-2-i} and assume that $g^{jk}(x)$ and $V(x)$ have a singularity at $0$ satisfying there (\ref{3-1-14})--(\ref{3-1-16}) with $\gamma(x)=|x|$ and $\rho(x)=|x|^m$ with $m<-1$.

Consider the asymptotics of eigenvalues tending to $+\infty$. As in Example~\ref{example-3-1-2}\ref{example-3-1-2-i}, we take $\gamma(x)=\frac{1}{2}|x|$ and $\rho_\lambda(x)=\lambda^{1/2}$ for
$|x|\ge \epsilon_0 \lambda^{1/2m}$ (then the contribution of $B(x,\gamma(x))$ to the remainder does not exceed $\lambda^{(d-1)/2}|x|^{d-1}$) and $\rho_\lambda(x)=|x|^{m}$ as $|x|\le \epsilon_0 \lambda^{1/2m}$ (then due to the ellipticity the contribution of $B(x,\gamma(x))$ to the remainder does not exceed $\rho^{-s}\gamma^{-s-d}$). We conclude that the contribution of $B(0,\epsilon)$ to the remainder does not exceed $C\lambda^{(d-1)/2}\epsilon^{d-1}$ which means that this singularity does not prevent remainder estimate as good as $o(\lambda^{(d-1)/2})$. However, this singularity affects the principal part which should be defined as in (\ref{3-1-17}).
\end{example}

\begin{example}\label{example-3-1-5}
\begin{enumerate}[label=(\roman*), wide, labelindent=0pt]
\item\label{example-3-1-5-i}
When analyzing the asymptotics of the large eigenvalues, we can consider a potential that is either rapidly increasing
(with $\rho = \exp(|x|^m)$, $\gamma(x)= |x|^{1-m}$, $m>0$), or slowly increasing
(with $\rho =(|\log x|)^m$, $\gamma(x)= |x|$, $m>0$) which affects both the magnitude of the main part and the remainder estimate.

\item\label{example-3-1-5-ii}
When analyzing the asymptotics of the  eigenvalues tending to the bottom of the essential spectrum, we can consider a potential that is either rapidly decreasing (with $\rho=|x|^{-1}(\log |x|)^m$ with $m>0$, $\gamma(x)=|x|$, $m>0$)
or slowly decreasing (with $\rho =(|\log x|)^m$, $\gamma(x)= |x|$, $m<0$) which affects both the magnitude of the main part as well as the remainder estimate.
\end{enumerate}
\end{example}

\begin{remark}\label{rem-3-1-6}
We can consider the same examples albeit assuming only that $h\in (0,1)$; then the remainder estimate acquires the factor $h^{-d+1}$.
\end{remark}

\subsection{Singularities}
\label{sect-3-1-2}

Let us consider other types of singularities when there is  \emph{a singular zone\/} where after rescaling $\hbar \le  1$\,\footref{foot-29}. Still, it does not prevent us from using the approach described above to get an estimate from below for (\ref{3-1-11}) or (\ref{3-1-12}): we only need to decrease these expressions by inserting $\psi$ ($0\le \psi \le 1$) that is supported in \emph{the regular zone\/} (aka \emph{the semiclassical zone\/})  defined by
$\hbar \le 2\hbar_0$ after rescaling and equal to $1$ for $\hbar \le \hbar_0$ and applying the rescaling technique there.

Let us discuss an estimate from above. If there was no regular zone at all, we would have no estimate from below at all but there could be some estimate from above of variational nature. The best known is the LCR (Lieb-Cwikel-Rosenblum) estimate
\begin{equation}
\N^-(0)\le Ch^{-d}\int V_-^{\frac{d}{2}}\,dx
\label{3-1-19}
\end{equation}
for the Schr\"odinger operator with Dirichlet boundary conditions as $d\ge 3$. For $d=2$, the estimate is marginally worse  (see \cite{rozenblioum:lcr} for the most general statement for arbitrary orders of operators and dimensions and \cite{shar} for the most general results for the Schr\"odinger operator in dimension $2$).

It occurs that we can split our domain into an overlapping \emph{regular zone\/} $\{x:\,\rho(x)\gamma(x) \ge h\}$ and a \emph{singular zone\/} $\{x:\,\rho(x)\gamma(x) \le 3h\}$, then evaluate the contribution of the regular zone using the rescaling technique and the contribution of the singular zone by the variational estimate as if on the inner boundary of this zone (i.e. a part of its boundary which is not contained in $\partial X$) the Dirichlet boundary condition was imposed, and we add these two estimates:
\begin{multline}
-Ch^{1-d}\int_{\{\rho\gamma \ge h, V\le \epsilon \rho^2\}} \rho^{d-1}\gamma^{-1}\sqrt{g}\,dx \\
 \le \N^-(0)-(2\pi)^{-d}\omega_d h^{-d} \int_{\{\rho\gamma \ge h\}} V_-^{d/2}\,dx\\
\le
Ch^{1-d}\int_{\{\rho\gamma \ge h, V\le \epsilon \rho^2\}} \rho^{d-1}\gamma^{-1}\,dx+
Ch^{-d}\int_{\{\rho\gamma \le h, V\le \epsilon \rho^2\}}\rho^{d}\,dx.
\label{3-1-21}
\end{multline}
See Theorems~\ref{book_new-thm-9-1-7} and \ref{book_new-thm-9-1-13} of \cite{futurebook} for more general statements. Further, similar statements could be proven for operators which are not semi-bounded (see Theorems~\ref{book_new-thm-10-1-7} and \ref{book_new-thm-10-1-8} of \cite{futurebook}).

In particular, we have:

\begin{example}\label{example-3-1-7}
\begin{enumerate}[label=(\roman*), wide, labelindent=0pt]
\item\label{example-3-1-7-i}
Let
\begin{gather}
\int \rho^{d-1}\gamma^{-1}\,dx<\infty.
\label{3-1-22}\\
\shortintertext{Then,}
\N^-(0)=(2\pi)^{-d}\omega_d h^{-d} \int V_-^{d/2}\sqrt{g}\,dx +O(h^{1-d}).
\label{3-1-23}
\end{gather}
\item\label{example-3-1-7-ii}
If in addition the standard non-periodicity condition is satisfied then
\begin{multline}
\N^-(0)=(2\pi)^{-d}\omega_d h^{-d} \int V_-^{d/2}\sqrt{g}\,dx -\\
\frac{1}{4}(2\pi)^{1-d}\omega_{d-1} h^{1-d} \int V_-^{(d-1)/2}\,dS+
o(h^{1-d}),
\label{3-1-24}
\end{multline}
where $dS$ is a corresponding density on $\partial X$.
\end{enumerate}
\end{example}

\begin{example}\label{example-3-1-8}
Consider the Dirichlet Laplacian  in a domain $X$ assuming that  there exists scaling function $\gamma(x)$ such that (\ref{3-1-14}) holds and
\begin{claim}\label{3-1-25}
For each $y\in X$, the connected component of $B(y,\gamma(x))\cap X$ containing $y$   coincides with $\{x\in B(0,1), x_1\le f(x'\}$, where $x'=(x_2,\ldots,x_d)$ and
\begin{equation}
|\nabla ^\alpha f|\le C_\alpha\gamma^{1-|\alpha|}\qquad \forall \alpha,
\label{3-1-26}
\end{equation}
where we rotate the coordinate system  if necessary\footnote{\label{foot-30} It is precisely the condition that we need to impose on the boundary for the rescaling technique to work.}.
\end{claim}
\begin{enumerate}[label=(\roman*), wide, labelindent=0pt]
\item\label{example-3-1-8-i}
Then the principal part of asymptotics is
\begin{gather}
c_0 \lambda^{\frac{d}{2}}h^{-d}
\int _{\{x:\gamma(x)\ge \lambda^{-\frac{1}{2}}\}}\,dx
\label{3-1-27}\\
\intertext{and the remainder does not exceed}
C \lambda^{\frac{d-1}{2}}h^{1-d}\int _{\{x:\gamma(x)\ge \lambda^{-\frac{1}{2}}\}}\gamma(x)^{-1}\,dx +
C \lambda^{\frac{d}{2}}h^{-d}\int _{\{x:\gamma(x)\le \lambda^{-\frac{1}{2}}\}}\,dx.
\label{3-1-28}
\end{gather}
\item\label{example-3-1-8-ii}
In particular, if
\begin{equation}
\int _X\gamma(x)^{-1}\,dx <\infty,
\label{3-1-29}
\end{equation}
then the standard asymptotics with the remainder estimate $O(\lambda^{(d-1)/2}h^{1-d})$ hold. Moreover, under the standard condition (\ref{1-1-3}), we arrive to the two-term asymptotics (\ref{1-1-2}).

These conditions are satisfied for domains with vertices, edges and conical points. In fact, we may allow other singularities including outer and inner spikes and cuts.

Furthermore, these conditions are satisfied for unbounded domains with cusps (exits to infinity) provided these cusps are thin enough (basically having finite volume and area of the boundary).

\item\label{example-3-1-8-iii}
These results hold under the Neumann  or mixed Dirichlet-Neumann boundary condition, but then we need to assume that the domain satisfies the cone condition; for the two-term asymptotics, we also  need to assume that the border between the parts of $\partial X$ with the Dirichlet and Neumann boundary conditions has $(d-1)$-dimensional measure $0$.
\end{enumerate}
\end{example}

\begin{example}\label{example-3-1-9}
\begin{enumerate}[label=(\roman*), wide, labelindent=0pt]
\item\label{example-3-1-9-i}
Suppose that the potential is singular at $0\in X$ like $|x|^{2m}$ with $m \in (-1,0)$. Then this singularity does not affect the asymptotics of large eigenvalues.

\item\label{example-3-1-9-ii}
Let us consider Example~\ref{example-3-1-2}\ref{example-3-1-2-i} albeit allow $V\ge 0$ to vanish along certain directions. Then we have \emph{canyons\/} and $\{x:\,V(x)\le \lambda\}$ are cusps. If the canyons are narrow and steep enough then the same asymptotics (\ref{3-1-17}) hold. Moreover, under the non-periodicity condition, the remainder estimate is ``$o$''.

\item\label{example-3-1-9-iii}
Let us consider Example~\ref{example-3-1-2}\ref{example-3-1-2-ii} albeit allow $V\ge 0$ to be singular along certain directions. Then we have \emph{gorges\/} and $\{X:\,V(x)\le \lambda\}$ are cusps. If the gorges are narrow and shallow enough then the same asymptotic (\ref{3-1-17}) hold. Moreover, under the non-periodicity condition, the remainder estimate is ``$o$''.
\end{enumerate}
\end{example}

\begin{example}\label{example-3-1-10}
We can consider also multiparameter asymptotics, for example with respect to $h\to +0$ and $\lambda$. In addition to what we considered above, the following interesting possibility appears: $\lambda\searrow \lambda_*\Def \inf V(x)$ which is either finite or $-\infty$. Then if $\lambda$ tends to $\lambda_*$  slowly enough so that $\N^-_h(\lambda)\to +\infty$, we get interesting asymptotics.

In particular, as \underline{either} $V(x)\asymp |x|^{2m}$ with $m>0$ and
$\lambda\to +0$ \underline{or} $V(x)\asymp |x|^{2m}$ with $0>m>-1$ and
$\lambda\to -\infty$, then  this condition is $h=o(|\lambda|^{(m+1)/2m})$.
\end{example}

\begin{remark}\label{rem-3-1-11}
We can also consider  $\Tr ((-H)^\nu\uptheta(-H))$ with $\nu>0$. Then in the estimates above, $\rho^d\mapsto \rho^{d+2\nu}$ and
$\rho^{d-1}\mapsto \rho^{d+\nu-1}\gamma^{-\nu-1}$.
\end{remark}

For full details, proofs  and generalizations, see Chapter~\ref{book_new-sect-11} of \cite{futurebook} which covers also non-semibounded operators.

\section{Non-Weyl asymptotics}
\label{sect-3-2}

\subsection{Partially Weyl theory}
\label{sect-3-2-1}

Analyzing the examples of the previous section, one can observe that for some values of the exponents, the condition (\ref{3-1-22}) (or it special case (\ref{3-1-29})) fails but the main term of the asymptotics is still finite and has the same rate of the growth as it had before, while for other values of the exponent, it is infinite. In the former case, we get Weyl asymptotics but with a worse remainder estimate, in the latter case, all we can get is an estimate rather than the asymptotics. Can one save the day?

In many cases, the answer is positive and we can derive either the Weyl asymptotics but with a non-Weyl correction term or completely non-Weyl asymptotics. The main but not the only tool is the spectral asymptotics for operators with operator-valued symbols. Namely, in some zone of the phase space, we  separate the variables\footnote{\label{foot-31} After some transformation, the transformations and separations  in the different zones  are not necessarily the same.} $x=(x';x'')$ and $(\xi';\xi'')$ respectively and consider the variables $(x',\xi')$ as \emph{semiclassical variables\/} (or \emph{Weyl variables\/}), similar to $(x,\xi)$ in the previous scheme. So we get an operator $\mathsf{H}(x',D')$ with an operator-valued symbol which we can study in the same way as the operator $H$ before.

One can say that we have a matrix operator but with a twist: first, instead of finite-dimensional matrices, we have unbounded self-adjoint operators in the auxilary infinite-dimensional Hilbert space $\bH$ (usually $\mathscr{L}^2$ in the variables $x''$); second, we are interested in the asymptotics
\begin{equation}
\int \tr_{\bH} \bigl(\hat{e}(x',x',\lambda)\bigr)\,dx',
\label{3-2-1}
\end{equation}
rather than in the asymptotics without trace where $\hat{e}(x',y';\lambda)$ is an operator in $\bH$ (with  Schwartz kernel $e(x',y'; x'',y'';\lambda)$);  and, finally, the main term in asymptotics is
\begin{equation}
\int \tr_{\bH} \bigl(\mathsf{e}(x',\xi';\lambda)\bigr)\,d\xi'dx',
\label{3-2-2}
\end{equation}
where $\mathsf{e}(x',\xi',\lambda)$ is a spectral projector (in $\bH$) of $\mathsf{H}(x',\xi')$. Here, we need to assume that $\mathsf{H}(x',\xi')$ is microhyperbolic with respect to $(x',\xi')$. Since the operator  $\tr_\bH$ is now  unbounded, both the main term of the asymptotics of (\ref{3-2-1}) and the remainder estimate \emph{may\/} have magnitudes different from what they would be without $\tr_\bH$.

Since the operator $\mathsf{H}(x',\xi')$ is rather complicated, we want to replace it by some simpler operator and add some easy to calculate correction terms.

We consider multiple examples below. Magnetic Schr\"odinger, Schr\"odinger-Pauli and Dirac operators studied in Sections~\ref{sect-5} and~\ref{sect-6} are also of this type.

\subsection{Domains with thick cusps}
\label{sect-3-2-2}

This was done in Section~\ref{book_new-sect-12-1}  of \cite{futurebook} for operators in domains with thick cusps of the form $\{x:\,x''\in f(x')\Omega\}$ where $\Omega$ is a bounded domain in $\bR^{d''}$ with  smooth boundary, defining the cusp crossection. Here again we consider for simplicity the Dirichlet Laplacian. Assume first that the metric is Euclidean and the domain $X=\{x=(x',x''):\,x'\in X'\Def \bR^{d'},\, x''\in f(x')\Omega\}$. Then, the change of variables $x''\mapsto x''/f(x')$ transforms $\Delta$ to the operator
\begin{equation}
P=\sum_{1\le j\le d'}
\Bigl(D_j+g_{x_j}L + \frac{id''}{2}g_{x_j}\Bigr)
\Bigl(D_j+g_{x_j}L - \frac{id''}{2}g_{x_j}\Bigr)
+ \frac{1}{f^2}\Delta''
\label{3-2-3}
\end{equation}
in $\sL^2(X'\times \Omega)=\sL^2(\bR^{d'}, \bH)$ where
$L=\langle x'',D''\rangle$, $g=-\log f$, $\bH=\sL^2(\Omega)$, $\Delta'$ is a Laplacian in $X'$, and $\Delta''=\Delta''_\D$ is a Dirichlet Laplacians in  $\Omega$, and we simultaneously multiply $u$ by $f^{-d''/2}$ to have the standard Euclidean measure rather than the weighted one $f^{d''}dx$.
We consider the operator (\ref{3-2-3}) as a perturbation of the operator
\begin{equation}
\bar{P}\Def \Delta' + \frac{1}{f^2}\Delta'',
\label{3-2-4}
\end{equation}
which is a direct sum of $d'$-dimensional Schr\"odinger operators $P_n=\Delta'+\mu_n f^{-2}$ in $X'$ where $\mu_n>0$ are the eigenvalues of $\Delta''_\D$. Assuming that
\begin{equation}
f\asymp |x|^{-m},\qquad |\nabla f|\asymp |x|^{-m-1}\qquad\text{for\ \ }|x'|\ge c,
\label{3-2-5}
\end{equation}
we can ensure that the microhyperbolicity condition (with respect to $(x';\xi')$) is fulfilled for $P_n$, $\bar{P}$, as well as for $P$.

Then according to the previous section, for $P_n$ the eigenvalue counting function is
\begin{equation}
\N_n (\lambda) =
c_{d'}\int \bigl(\lambda - \mu _n f^{-2}(x')\bigr)_+^{\frac{d'}{2}}\,dx'+
O\bigl(\lambda^{(d'-1)(m+1)/2m}\mu_n^{-(d'-1)/2m}\bigr),
\label{3-2-6}
\end{equation}
where the remainder estimate is uniform  with respect to $n$. Observe that  for $\bar{P}$ the eigenvalue counting function  is
$\bar{\N} (\lambda) =\sum_n \N_n(\lambda)$. Using  $\mu_n \asymp n^{2/d''}$, we arrive to
\begin{gather}
\N_n (\lambda) =
c_{d'}\iint (\lambda - \mu  f^{-2}(x'))_+^{\frac{d'}{2}}\,dx'd_\mu \mathbf{n}(\mu) + O(R(\lambda))
\label{3-2-7}\\
\shortintertext{with}
R(\lambda)=\lambda^{\frac{1}{2}(d-1)}+ \lambda^{\frac{m+1}{2m}(d'-1)} +
\updelta_{(d'-1), md''}\lambda^{\frac{1}{2}(d-1)} \log \lambda,
\label{3-2-8}
\end{gather}
where $\mathbf{n}(\mu)$ is the eigenvalue counting function for $\Delta''_\D$.

We show, moreover, that the same asymptotics holds for our original operator (\ref{3-2-3}). Furthermore, if the first term in (\ref{3-2-7}) is dominant, then under the standard non-periodicity assumption  we can replace $O(\lambda^{(d-1)/2})$ by $o(\lambda^{(d-1)/2})$; we need to add  the standard boundary term to the right-hand expression in (\ref{3-2-7}).

On the other hand, if the second term in (\ref{3-2-7}) dominates, then assuming that $f$ stabilizes as $|x'|\to\infty$ to a positively homogeneous function $f_0$,  under the corresponding non-periodicity assumption (now in $T^*\bR^{d'}$) for $|\xi'|^2 + f_0^{-2}(x')$, we can replace $O(\lambda^{(m+1)(d'-1)/2m})$ by $o(\lambda^{(m+1)(d'-1)/2m})$. Finally, if both powers coincide then under the stabilization condition, the remainder estimate is
$o(\lambda^{(d-1)/2} \log \lambda)$ but we need to add  the modified boundary term to the right-hand expression in (\ref{3-2-7}).

Obviously, the principal part in (\ref{3-2-7}) is of the magnitude
\begin{equation}
S(\lambda)=\lambda^{\frac{1}{2}d}+ \lambda^{\frac{m+1}{2m}d'} +
\updelta_{d', md''}\lambda^{\frac{1}{2}d} \log \lambda.
\label{3-2-9}
\end{equation}

If $X$ is not exactly of the same form and the metric only stabilizes (fast enough) at infinity to $g^{jk0}\Def \updelta_{jk}$, then we can recover the same remainder estimate and reduce the principal part to
\begin{multline}
 c_{d'} \iint (\lambda - \mu f^{-2}(x'))_+^{\frac{d'}{2}}\phi(x')\,
 dx'd_\mu \mathbf{n}(\mu) \\+
c_d \lambda^{d/2}\int_X (\sqrt{\mathstrut{g}}- \sqrt{g^0}\phi (x'))\,dx,
\label{3-2-10}
\end{multline}
where $\supp(\phi)\subset \{|x'|\ge c\}$, $\phi =1$ in $\{|x'|\ge c\}$.
Here, the first part is exactly as above and the second term is actually the sum of two terms; one of them  $c_d \lambda^{d/2}\int \sqrt{g}(1- \phi (x'))\,dx$ is the contribution of the ``finite part of the domain'' (without the cusp) and the second $c_d \lambda^{d/2}\int (\sqrt{g\vphantom{g^0}}-\sqrt{g^0})\phi (x')\,dx$ is a contribution of the cusp in the correction.

Note that now to get the remainder estimate $o(\lambda^{(d-1)/2})$, one needs to include the standard boundary term in the second part of (\ref{3-2-10}).

The crucial part of our arguments is a \emph{multiscale analysis\/}. As long as %\linebreak
${r\le c\lambda^{1/2m-\delta}}$, we can scale  $x\mapsto xr^m$ and consider
$\sigma_0 (t)=\Tr \bigl(e^{ih^{-1}t H}\phi(x'/r)\bigr)$; here $H=\lambda^{-1}P$, $h=\lambda^{-1/2}r^{m}$. From the propagation with respect to $(x,\xi)$, we know that on energy level $1$, the time interval $(h^{1-\delta}, \epsilon)$ contains no singularities of $\sigma_0 (t)$.

On the other hand, for $r\ge c$, we can scale  $x\mapsto x/r$ and consider
$\sigma_1 (t)=\Tr \bigl(e^{i\hbar^{-1}t H}\phi(x'/r)\bigr)$; here
$\hbar= \lambda^{-1/2}r^{-1}$. From the propagation with respect to $(x',\xi')$, we know that on energy level $1$, the time interval $(\hbar^{1-\delta}, \epsilon)$ contains no singularities of $\sigma_1 (t)$.

Observe first that $\sigma_1(t)=\sigma_0(r^{-1-m}t)$ and therefore the time interval $(h^{1-\delta}, \epsilon r^{m+1})$ contains no singularities of $\sigma_0 (t)$. This allows us to improve the remainder estimate in the full Weyl asymptotics but we need to include many terms which are difficult to calculate.

On the other hand, for $\lambda^\delta\le r\le c\lambda^{1/2m}$, we can consider $H$ as a perturbation of $\bar{H}=\lambda^{-1}\bar{P}$. We do it first in the framework of the theory of operators with operator-valued symbols. Then we consider all perturbation terms and apply to them ``full Weyl theory'' and due to the stabilization assumption, the error is less than (\ref{3-2-8}).  This gives us another asymptotics, also with many terms which are difficult to calculate.

Comparing these two asymptotics in their common domain $\lambda^\delta \le r\le \lambda^{1/2m-\delta}$, we conclude that all terms but those present in both must be $0$; it allows us to eliminate almost all the terms and sew these asymptotics resulting in (\ref{3-2-10}).

Using the same approach, we can consider higher order operators, the case when $X'$ is a conical set and there are several cusps $X_k$ which may have  different dimensions $d'_k$ and rates of decay (then both the principal part and the remainder estimate should be modified accordingly).

\subsection{Neumann Laplacian in domains with ultra-thin cusps}
\label{sect-3-2-3}

Consider the Neumann Laplacian in domains with cusps. Recall that since these domains do not satisfy the cone condition, we so far have no results even if the cusp is thin. Applying the same arguments as before, we hit two obstacles. The first (minor) obstacle is that the Neumann boundary condition for the operator (\ref{3-2-3}) coincides with the same condition for $\Delta''$ only asymptotically. The second (major) obstacle is that $\mu_1=0$ and $P_1=\Delta'$ has a continuous spectrum. In fact, we should not reduce $P$ to $\bar{P}$;  from (\ref{3-2-3}) we conclude that
\begin{gather}
P_1= \sum_{1\le j\le d'} \Bigl(D_j+\frac{id''}{2}g_{x_j}\Bigr)\Bigl(D_j-\frac{id''}{2}g_{x_j}\Bigr)=
\Delta' + W
\label{3-2-11}\\
\shortintertext{with}
W=\frac{d''^2}{4}|\nabla g|^2 +\frac{d''}{2}\Delta'g .
\label{3-2-12}
\end{gather}
Still this operator may have a continuous spectrum unless $|\nabla g|\to \infty$ as $|x|\to \infty$. We need to assume that $f$ has  superexponential decay: $f=e^{-g}$ with
\begin{align}
&|\nabla^\alpha g| \le c_\alpha |x|^{1+m-|\alpha|} &&\forall \alpha,
\label{3-2-13}\\[3pt]
&g\asymp |x'|^{m+1}, \quad|\nabla g| \asymp |x'|^{m}
&&\text{for\ \ } |x'|\ge c,
\label{3-2-14}\\[3pt]
&|\nabla |\nabla g|^2 |\asymp |x|^{2m-1} &&\text{for\ \ } |x'|\ge c,
\label{3-2-15}
\end{align}
where $m>0$ and (\ref{3-2-15}) is a microhyperbolicity condition for $P_1$. Then one can prove easily that when $d''\ge 2$,
\begin{gather}
\N(\lambda)= c_d \lambda^{d/2} \int _X \sqrt{g}\,dx +
c_{d'} \int (\lambda -W)_+^{d'/2}\,dx' + O(R(\lambda))
\label{3-2-16}
\shortintertext{with}
R(\lambda)=\lambda^{\frac{1}{2}(d-1)}+ \lambda^{\frac{m+1}{2m}(d'-1)}.
\label{3-2-17}
\end{gather}
Moreover, if the first term in (\ref{3-2-7}) dominates, then under the standard non-periodicity assumption, we can replace $O(\lambda^{(d-1)/2})$ by $o(\lambda^{(d-1)/2})$ (simultaneously including the standard boundary term); if the second term dominates, then assuming that $W$ stabilizes as $|x'|\to\infty$ to a positively homogeneous function $W_0$, under the corresponding assumption for $|\xi'|^2 + W_0(x')$ we can replace
$O(\lambda^{(m+1)(d'-1)/2m})$ by $o(\lambda^{(m+1)(d'-1)/2m})$.

One can see easily that $\N(\lambda)\asymp
S(\lambda)=\lambda^{\frac{1}{2}d}+ \lambda^{\frac{m+1}{2m}d'}$. Observe that in contrast to (\ref{3-2-8}) and (\ref{3-2-9}), even if the exponents coincide, a logarithmic factor does not appear.

The case $d''=1$ is special since even an ultra-thin cusp is also thick (according to the classification of the previous Subsection~\ref{sect-3-2-1}) and the corresponding formulae should include a modified boundary term containing the double logarithm of $\lambda$. For this and other generalizations, see Section~\ref{book_new-sect-12-6} of \cite{futurebook}. Also one can consider \emph{spikes\/} with $\supp (f)=\{|x'|\le L\}$, in which case the standard Weyl asymptotics holds.

\subsection{Operators in $\bR^d$}
\label{sect-3-2-4}

The scheme of Subsection~\ref{sect-3-2-2} is repeated in many similar cases.

First, consider eigenvalues tending to $+\infty$ for the Schr\"odinger operator with potential $V$ which generically is $\asymp |x|^{2m}$ but vanishes along some directions.

For example, consider the toy model $V= |x|^{2m-2m'}|x''|^{2m'}$ with $m>m'>0$. Let $X'=\bR^{d'}\ni x'$ and $X''=\bR^{d''}\ni x''$. Consider only the conical vicinity of $X'$ and here we instead consider  the potential $V= |x'|^{2m-2m'}|x''|^{2m'}$. Consider only the part of operator which is related to $x''$: $\Delta'' + |x'|^{2m-2m'}|x''|^{2m'}$ and after the change of variables $x''\mapsto x'' |x'|^{k}$ with $k=(m-m')/(m'+1)$, it becomes $|x'|^{2k}L$ with $L= \Delta'' + U(x'')$, $U=|x''|^{2m''}$. The condition $m''>0$ ensures that the spectrum of $L$ is discrete  and accummulates to $+\infty$.

So basically we have a mixture of the Schr\"odinger operator on $\bR^d$  with a potential growing as $|x|^{2m}$ and the Schr\"odinger operator with the operator-valued symbol on $\bR^{d''}$ with a potential growing as $|x|^{2k}$ and we recover the asymptotics with the remainder estimate $O(R(\lambda))$, where
\begin{gather}
R(\lambda)= \lambda^{\frac{m(d-1)}{(m+1)}} + \lambda^{\frac{k(d'-1)}{(k+1)}} +
\updelta_{\frac{m(d-1)}{(m+1)}, \frac{k(d'-1)}{(k+1)}} \lambda^{\frac{m(d-1)}{(m+1)}}\log \lambda
\label{3-2-18}
\intertext{and the principal part is $\asymp S(\lambda)$, where}
S(\lambda)= \lambda^{\frac{md}{(m+1)}} + \lambda^{\frac{kd'}{(k+1)}} +
\updelta_{\frac{md/(m+1)},\frac{kd'}{(k+1)}}
\lambda^{\frac{md}{(m+1)}}\log \lambda.
\label{3-2-19}
\end{gather}
In a rather general situation, this principal part is similar to the one in  (\ref{3-2-10}) where $\mathbf{n}(\mu)$ is the eigenvalue counting function for $L$. Further, under similar non-periodicity assumptions,  we can replace ``$O$'' by ``$o$''. For generalizations, details and proofs, see Section~\ref{book_new-sect-12-2} of \cite{futurebook}.

Second,  consider eigenvalues tending to $-0$ for the Schr\"odinger operator with a potential $V$ which generically is $\asymp |x|^{2m}$ with $m\in (-1,0)$ but is singular in some directions. Again, consider a toy model
$V= -|x|^{2m-2m'}|x''|^{2m'}$ with $-1<m<m'<0$. Again, $L= \Delta'' + U(x'')$, $U=-|x''|^{2m''}$ and its negative spectrum is discrete and accummulates to $-0$. The formulae (\ref{3-2-18}) and (\ref{3-2-19}) remain valid (albeit $\lambda\to -0$). For generalizations, details and proofs, see  Section~\ref{book_new-sect-12-3} of \cite{futurebook}.

\subsection{Maximally hypoelliptic operators}
\label{sect-3-2-5}

Third, consider the eigenvalues tending to $+\infty$ for \emph{maximally hypoelliptic operators with a symplectic manifold of degeneration\/}. Consider the toy model $P= \Delta ''+|x''|^{2m}\Delta '$. In this case, after the partial Fourier transform, we get $\Delta ''+|x''|^{2m}|\xi'|^2$ and after the change of variables $x''\mapsto |\xi'|^k x''$, we get $|\xi'|^{2k}L$, $L= \Delta''+ |x''|^{2m}$ and $k=1/(m+1)$.

This toy model is maximally hypoelliptic as the spectrum of $L$ is discrete  and accummulates to $+\infty$. So basically we have a blend of operator of order $2$ on $\bR^d$ and of order $2k$ on $\bR^{d'}$ and we recover the asymptotics with remainder estimate $O(R(\lambda))$ with
\begin{gather}
R(\lambda)= \lambda^{\frac{(d-1)}{2}} + \lambda^{\frac{(d'-1)}{2k}} +
\updelta_{d-1, (d'-1)/k} \lambda^{\frac{(d-1)}{2}}\log \lambda
\label{3-2-20}
\intertext{and principal part $\asymp S(\lambda)$ with}
S(\lambda)= \lambda^{\frac{d}{2}} + \lambda^{\frac{d'}{2k}} +
\updelta_{d,d'/k} \lambda^{\frac{d}{2}}\log \lambda.
\label{3-2-21}
\end{gather}
Further, under similar non-periodicity assumptions,  we can replace ``$O$'' by ``$o$''. For generalizations, details and proofs, see Section~\ref{book_new-sect-12-4} of \cite{futurebook}.

\subsection{Trace asymptotics for operators with singularities}
\label{sect-3-2-6}

Here, we also consider only one example (albeit the most interesting one) of a Schr\"odinger operator $H\Def h^2\Delta -V(x)$ in $\bR^3$ with potential $V(x)$ at $0$  stabilizing to a positive homogeneous  function $V_0$ of degree $-1$:
\begin{equation}
|\nabla^\alpha (V-V_0)|\le c_\alpha|x|^{-|\alpha|}\qquad \forall \alpha.
\label{3-2-22}
\end{equation}
We assume that $V(x)$ decays fast enough at infinity and we are interested in the asymptotics of $\Tr (H_-)$, which is the sum of the negative eigenvalues of $H$. While generalizations are considered in Section~\ref{book_new-sect-12-5} of \cite{futurebook}, exactly this problem with $V_0=|x|^{-1}$ arises in the asymptotics of the ground state energy of heavy atoms and molecules.

It follows from Section~\ref{sect-3-1} that $\N^-_h $ has  purely Weyl asymptotics with the remainder estimate $O(h^{-2})$ and\footnote{\label{foot-32}
Under the standard non-periodicity condition.\/} it could be improved to $o(h^{-2})$ but we have a different object and if the potential had no singularities, the remainder estimate would be $O(h^{-1})$ or even $o(h^{-1})$\,\footref{foot-32}\footnote{\label{foot-33} But then the principal part of asymptotics should include  the third term $c h^{-1}$ while the second  term vanishes.}.

Therefore considering the contribution of the ball $B(x,\gamma(x))$ with $\gamma(x)=\frac{1}{2}|x|$, we have a contribution to the Weyl expression
\begin{equation}
-c h^{-3}\int V_+^{\frac{5}{2}}\,dx
\label{3-2-23}
\end{equation}
of magnitude $C\rho^2 (h/\rho\gamma)^{-3} = Ch^{-3}\rho^5 \gamma^3$, while the contribution to the remainder does not exceed
$C\rho^2 (h/\rho\gamma)^{-1}=Ch^{-1}\rho \gamma$ with $\rho=|x|^{-\frac{1}{2}}$.
We see that the former converges at $0$ and the latter diverges. This analysis could be done for $\rho\gamma\ge h$ i.e.~if $|x|\ge h^2$. Then we conclude that the contribution of the zone $\{x:\,|x|\ge a\}$ to the remainder does not exceed $Ch^{-1}a^{-\frac{1}{2}}$ which as $a=h^2$ is $O(h^{-2})$. On the other hand, one can easily prove that the contribution of $B(0,h^2)$ to the asymptotics is also $O(h^{-2})$.

To improve this estimate, we analyze $B(0,a)$ in more detail. In virtue of (\ref{3-2-22}), we can  easily prove that the contribution of $B(x,\gamma)$ to
$\Tr (H_- - H_{0-})$ (with $H_0=h^2\Delta-V_0$) does not exceed $C(h/\rho\gamma)^{-2}=Ch^{-2}\rho^2\gamma^2$ and therefore the contribution of $B(0,a)$ to the remainder is  $O(h^{-2}a)$. Minimizing the total error
$h^{-2}a + h^{-1}a^{-\frac{1}{2}}$ in $a$, we get $a=h^{\frac{2}{3}}$ and the remainder $O(h^{-\frac{4}{3}})$, which is better than $O(h^{-2})$ but not as good as $O(h^{-1})$.

But then we need to include in the asymptotics  the extra term
\begin{equation}
\int \bigl(e_0 ^1 (x,x,0) - c V_{+}^{\frac{5}{2}}(x)\bigr)\psi(a^{-1}x)\,dx,
\label{3-2-24}
\end{equation}
where $e(\cdot,\cdot,\lambda)$ is the Schwartz kernel of the spectral projectors for $H$,
$e^1(\cdot,\cdot,0)= \int_{-\infty}^0 \lambda\,d_\lambda e(\cdot,\cdot,\lambda)$ and the subscript $0$ means that it is for $H_0$ and $\psi\in \sC_0^\infty (B(0,2))$ and equals $1$ in $B(0,1)$.

Basically, all that we achieved so far was to replace $H$ by $H_0$ in (\ref{3-2-24}). The same arguments allow us to replace $\psi$ by $1$ in this expression with the same error $O(h^{-1}a^{-\frac{1}{2}})$. This time, we cannot decompose it as the difference of two integrals because each of them is diverging at infinity (since $V_0$ decays there not fast enough). Further, due to the homogeneity of $V_0$, one can prove that this remodelled expression (\ref{3-2-24}) is homogeneous of degree $-2$ with respect to $h$ and thus is equal to $\kappa h^{-2}$. Here, $\kappa$ is some unknown constant, but for $V_0=|x|^{-1}$, it could be calculated explicitly.

Therefore, we conclude that with the remainder estimate $O(h^{-\frac{4}{3}})$, $\Tr(H_-)$ is given by the Weyl expression plus the \emph{Scott corretion term\/} $\kappa h^{-2}$.

To improve this remainder estimate, we should carefully study the propagation of singularities. We can prove that if $h^{2-\delta}\le \gamma \le 1$, then  the singularities do not come back ``in real time'' $\asymp 1$, which is a vast improvement over $\asymp \gamma\rho^{-1}\asymp \gamma^{\frac{3}{2}}$. Then the contribution of $B(x,\gamma)$ to the ``trace remainder'' does not exceed $Ch^{-1}\rho^2 \gamma^3$ but then the principal part of asymptotics should have a lot of terms;  the $n$-th term is of the magnitude
$h^{-3+2n} \rho^{2-2n}\gamma^{3-2n}$; however, using (\ref{3-2-22}) we conclude that the difference between such terms for $H$ and $H_0$ is
$O(h^{-3+2n} \rho^{-2n}\gamma^{3-2n})$ which leads to the estimate
\begin{equation}
\biggl |\int \bigl(e^1(x,x,0)- e_0 ^1 (x,x,0) -
c V_{+}^{\frac{5}{2}}(x)+c V_{0\,+}^{\frac{5}{2}}(x) \bigr)\,dx \biggr |\le Ch^{-1}.
\label{3-2-25}
\end{equation}
This estimate implies that with the remainder estimate $O(h^{-1})$, $\Tr(H_-)$ is given by the Weyl expression plus $\kappa h^{-2}$. Moreover, this estimate could be further improved to $o(h^{-1})$\,\footref{foot-32}\footref{foot-33}.

Similar results hold for other singularities (including singularities of the boundary), dimensions and $\Tr (H_-^\nu)$ with $\nu>0$. However, note that there could be  more than one such correction term.

\subsection{Periodic operators}
\label{sect-3-2-7}

Finally, consider an operator $H_0= H_0(x, D)$ with periodic coefficients (with the lattice of periods $\Gamma$). Then its spectrum is  usually absolutely continuous and consists of \emph{spectral bands\/}
$\{ \lambda_k(\xi):\, \xi \in \cQ'\}$ separated by \emph{spectral gaps\/}. Here, $\lambda_k$ are the eigenvalues of operator $H_0$ with \emph{quasiperiodic boundary conditions\/}
\begin{equation}
u(x+n)= T^n e^{i\langle n,\xi\rangle }(x)\qquad \forall n\in \Gamma,
\label{3-2-26}
\end{equation}
$\Gamma^*$ is the dual lattice\footnote{\label{foot-34} I.e. if
$\Gamma= \bZ e_1\oplus \bZ e\_2 \oplus  \ldots \oplus \bZ e_d$ then
$\Gamma^*= \bZ e'_1\oplus \bZ e'_2 \oplus  \ldots \oplus \bZ e'_d$ with
$\langle e_j,e'_k\rangle =\updelta_{jk}$.}, $\cQ$ and $\cQ'$ are corresponding \emph{elementary cells\/}\footnote{\label{foot-35} I.e.
$\cQ =\{x_1e_1+\ldots +x_de_d:\ x\in [0,1]^d\}$ and
$\cQ'* =\{\xi_1e'_1+\ldots +\xi_de'_d:\ \xi\in [0,1]^d\}$.};  $\xi$ is called the \emph{quasimomentum\/}. Here, $T=(T_1,\dots, T_d)$ is a family of commuting unitary matrices.

Let us consider an operator $H_t=H_0 -t W(x)$ with $W(x)>0$ decaying at infinity. Then, while the essential spectra of $H$ and $H_t$ are the same, $H_t$ can have discrete eigenvalues in the spectral gaps and these eigenvalues decrease as $t$ increases.

Let us fix an \emph{observation point\/} $E$ belonging  to either the spectral gap or its boundary and introduce $\N_E (\tau)$, the number of eigenvalues of $H_t$ crossing $E$ as $t$  changes between $0$ and $\tau$. We are interested in the asymptotics of $\N_E (\tau)$ as $t\to  \infty$.

Then using Gelfand's transform,
\begin{equation}
\cF u(\xi,x) =(2\pi)^{\frac{d}{2}} (\vol(\cQ'))^{-1} \sum_{n\in \Gamma} T^n e^{-i\langle n-x,\xi\rangle}u(x-n)
\label{3-2-27}
\end{equation}
with $(x,\xi)\in \cQ\times \cQ'$, this problem is reduced to the problem for operators with operator-valued symbols on $\sL^2 (\cQ', \bH_{\xi, \{T\}})$ where $\bH_{\xi, \{T\}}$ is the space of functions satisfying (\ref{3-2-26}).

After that, different results are obtained in three essentially different cases: when $E$ belongs to the spectral gap, $E$ belongs to the bottom of the spectral gap, and $E$ belongs to the top of the spectral gap. For exact results, proofs and generalizations, see Section~\ref{book_new-sect-12-7} of \cite{futurebook}.

\chapter{Non-smooth theory}
\label{sect-4}

So far we have considered operators with  smooth symbols in domains with smooth boundaries. Singularities were possible but only on  ``lean'' sets. However, it turns out that many results remain true under very modest smoothness assumptions.

\setcounter{section}{1}

\subsection{Non-smooth symbols and rough microlocal analysis}
\label{sect-4-1-1}

To deal with non-smooth symbols, we approximate them by \emph{rough\/} symbols $p\sim \sum_m p_m$,  depending on a small \emph{mollification parameter\/} $\varepsilon$ and satisfying
\begin{gather}
|\nabla ^\alpha _\xi\nabla ^\beta_x p_m(x,\xi)|\le
C _{m\alpha\beta}\boldrho^{-\alpha}\boldgamma^{-\beta}\varepsilon^{-m}
\label{4-1-1}\\
\shortintertext{with}
\min_j\rho_j\gamma_j\ge \varepsilon  \ge Ch^{1-\delta}
\label{4-1-2}
\end{gather}
(\emph{microlocal uncertainty principle}), which could be weakened to
\begin{multline}
|\nabla ^\alpha _\xi\nabla ^\beta_x p_m(x,\xi)|\le
C ^{|\alpha|+|\beta|+m+1} \alpha!\beta!m! \boldrho^{-\alpha}\boldgamma^{-\beta}\varepsilon^{-m} \\
 \forall \alpha,\beta, m:
|\alpha|+|\beta|+2m\le N=C|\log h|^{-1}
\label{4-1-3}
\end{multline}
with
\begin{equation}
\min_j\rho_j\gamma_j\ge \varepsilon  \ge Ch|\log h|
\label{4-1-4}
\end{equation}
(\emph{logarithmic uncertainty principle\/}). At this point, microlocal analysis ends: the assumptions cannot be weakened any further.

Assuming that
\begin{align}
&|\nabla ^\alpha _\xi\nabla ^\beta_x \nabla p_0(x,\xi)|\le
C ^{|\alpha|+|\beta|+1} \alpha!\beta! \boldrho^{-\alpha}\boldgamma^{-\beta}
\label{4-1-5}
\shortintertext{and}
&|\nabla ^\alpha _\xi\nabla ^\beta_x \nabla p_m(x,\xi)|\le
C ^{|\alpha|+|\beta|+m+1} \alpha!\beta!m! \boldrho^{-\alpha}\boldgamma^{-\beta}\varepsilon^{1-m} \quad (m\ge 1),
\label{4-1-6}
\end{align}
we can restore Theorem~\ref{thm-2-1-2} (see Theorem~\ref{book_new-thm-2-3-2} of~\cite{futurebook}) and therefore also the Corollaries~\ref{cor-2-1-3} and~\ref{cor-2-1-4}, assuming $\xi$-microhyperbolicity instead of the usual microhyperbolicity. For proofs and details, see Section~\ref{book_new-sect-2-3} of \cite{futurebook}.

After this, we can than use the successive approximation method like in Subsection~\ref{sect-2-1-2} (definitely some extra twisting required) and then recover the spectral asymptotics -- originally only for operators which are $\xi$-microhyperbolic.

To consider non-smooth symbols, we can bracket them between rough symbols: for example, for the Schr\"odinger operator
$p^- (x,\xi, h)\le p(x,\xi,h)\le p^+ (x,\xi, h)$ where
$p^\pm =p_\varepsilon \pm C\nu(\varepsilon)$ and $p_\varepsilon$ is the symbol $p$, $\varepsilon$-mollified and $\nu(\varepsilon)$ is the modulus of continuity of the metric and potential; $\varepsilon= Ch|\log h|$.

Then for $\nu(\varepsilon)=
O(\varepsilon |\log \varepsilon|^{-1})$\,\footnote{\label{foot-36} Which means that the first partial derivatives are continuous with modulus of continuity $\nu_1(\varepsilon)=\nu(\varepsilon)\varepsilon^{-1}$.}, we can recover the remainder estimate $O(h^{1-d})$; under even weaker regularity conditions by rescaling, we can recover weaker remainder estimates. On the other hand, if $\nu(\varepsilon)=o(\varepsilon |\log \varepsilon|^{-1})$, we can recover the remainder estimate $o(h^{1-d})$ under the standard non-periodicity condition\footnote{\label{foot-37} However, even for the Schr\"odinger operator without boundary, the dynamic equations do not satisfy the Lipschitz condition and thus the flow could be multivalued.}. For proofs and details, see Section~\ref{book_new-sect-4-6} of \cite{futurebook}. There is an alternative to the bracketing construction based on perturbation theory, which works better for the trace asymptotics and also covers the pointwise asymptotics. For an exposition, see Section~\ref{book_new-sect-4-6} of \cite{futurebook}.

Further, for scalar and similar operators, the rescaling technique allows us to replace $\xi$-microhyperbolicity by microhyperbolicity under really weak smoothness assumptions; here we also use  $\varepsilon$ depending on the point so that we can consider scalar symbols under weaker and weaker non-degeneracy assumptions albeit stronger and stronger smoothness assumptions. See Section~\ref{book_new-sect-5-4} of \cite{futurebook}.

\subsection{Non-smooth boundaries}
\label{sect-4-1-2}

Let us  consider a domain with non-smooth boundary (with the Dirichlet boundary condition). Here, the standard trick to flatten out the boundary by the change of variables $x_1\mapsto x_1-\phi(x')$ works very poorly: the operator principal symbol contains the first partial derivatives $\phi$ and therefore we need to require $\phi\in \sC^2$. Fortunately, the method of R.~Seelley \cite{seeley:sharp} can help us. This method was originally developed for the Laplacian with a smooth metric and a smooth boundary.

Here, we consider only the Schr\"odinger operator; assume first  that the metric and potential are smooth. Consider a point $\bar{x}\in X$ and assume that the metric is Euclidean at $\bar{x}$  and nearby, $X$ looks like $\{x:\,x_1\ge \phi(x')\}$ with
$\nabla '\phi(\bar{x'})=0$. Observe that these assumptions do not require any smoothness beyond $\sC^1$.

Consider a trajectory starting from $(\bar{x}, \xi)$.
If $|\xi_1 |< \rho\Def Ch|\log h|/\gamma$, the trajectory starts parallel to $\partial X$ and $\partial X$ can ``catch up'' only at time at least $T=\sigma(\gamma)$ where $\gamma=\frac{1}{2}\dist (x,\partial X)$ and $\sigma$ is the inverse function to $\nu$, which is a modulus of continuity for $\phi$\,\footref{foot-36}.

If $\xi_1 > \rho$ then this trajectory ``runs away from $\partial X$'' and $\partial X$ can ``catch up'' only at time at least $T=\sigma(\gamma)+\sigma_1(\xi_1)$ where  $\sigma_1$ is the inverse function to $\nu_1$\,\footref{foot-36}. On the other hand, if $\xi_1 <-\rho$, then we can revert the trajectory (which works only for local but not pointwise spectral asymptotics).

These arguments allow us to estimate the contribution of $B(x,\gamma(x))$ to the remainder by $Ch^{1-d} \gamma^d h|\log h |\sigma(\gamma)^{-1}$ and then the total remainder by $Ch^{1-d}\int \sigma(\gamma)^{-1}\,dx$. The latter integral converges for $\nu(t)= t |\log t|^{-1-\delta}$.

Sure, this works only when $\gamma \ge \bar{\gamma}=Ch|\log h|$. However, if we smoothen the boundary with a smoothing parameter $C\bar{\gamma}$, for
$\gamma \le \bar{\gamma}$, we will be in the framework of the smooth theory after rescaling and we can take $T=\bar{\gamma}$. The contribution of this strip to the remainder does not exceed $Ch^{1-d} \bar{\gamma}T^{-1}$ as its measure does not exceed $C\bar{\gamma}$. One can easily check  that the variation of $\vol (X)$ due to the smoothing of the boundary is $Ch^{1-d}$ and we can use the bracketing of $X$ as well.

We can even improve the remainder estimate to $o(h^{1-d})$ under the standard non-periodicity condition.

Furthermore, if  the metric and potential are not smooth, we need to mollify them, taking the mollification parameter $\varepsilon$ larger near $\partial X$ and taking $\rho= Ch|\log h|/\varepsilon$, but it works. For systems, we can exploit the fact that most of the cones of dependence are actually trajectories. For  exact statements, proofs and details, see Section~\ref{book_new-sect-7-5} of \cite{futurebook}.

\subsection{Aftermath}
\label{sect-4-1-3}

After the non-smooth local theory is developed, we can use all the arguments of Section~\ref{sect-3-1} and consider ``stronger but more concentrated'' singularities added on the top of the weaker ones.\enlargethispage{\baselineskip}

\chapter{Magnetic Schr\"odinger operator}
\label{sect-5}

\section{Introduction}
\label{sect-5-1}

This Section is entirely devoted to the study of the \emph{magnetic Schr\"{o}dinger operator\/}
\begin{gather}
H= (-ih\nabla - \mu A (x))^2 +V(x)
\label{5-1-1}\\
\intertext{and  of the Schr\"{o}dinger-Pauli operator}
H= ((-ih\nabla - \mu A (x))\cdot \boldupsigma)^2 +V(x)
\label{5-1-2}
\end{gather}
with a small semiclassical parameter $h$ and large \emph{magnetic intensity parameter} (coupling constant) responsible for the interaction of a particle with the magnetic field $\mu$. Here, $\boldupsigma=(\upsigma_1,\ldots,\upsigma_d)$ where
$\upsigma_1,\ldots,\upsigma_d$ are Pauli $\D\times \D$-matrices and $ A $ is \emph{magnetic vector potential}. We are interested in the two-parameter asymptotics (with respect to $h$ and $\mu$) as well as related asymptotics.

\section{Standard theory}
\label{sect-5-2}

\subsection{Preliminaries}
\label{sect-5-2-1}

For a detailed exposition, generalizations  and proofs, see Chapter~\ref{book_new-sect-13} of \cite{futurebook}.

Consider the most interesting cases $d=2,3$ with smooth $V(x)$ and $ A (x)$. If $d=2$, the magnetic field could be described by a single (pseudo)scalar $F_{12}=\partial_1 A_2-\partial_2 A_1$ and by a scalar $F=|F_{12}|$.  If $d=3$, the magnetic field could be described by a (pseudo)vector $\mathbf{F}=\nabla\times  A $ (\emph{vector magnetic  intensity\/}) and by a scalar $F=|\mathbf{F}|$ (\emph{scalar magnetic  intensity\/}). As a toy model, we consider an operator in $\bR^d$ with constant $V$ and $\mathbf{F}$. Then canonical form of the operator (\ref{5-1-1}) is
\begin{equation}
H= h^2D_1^2 + (hD_2-\mu Fx_1)^2 + \underbracket{h^2D_3^2}+V,
\label{5-2-1}
\end{equation}
with the third term omitted when $d=3$. Then we can calculate
\begin{gather}
e(x,x,\tau)= h^{-d} \cN^\MW_d(\tau-V, \mu h F)
\label{5-2-2}\\
\shortintertext{with}
\cN^\MW_d(\tau,  F) =\kappa_d  \sum_{n \ge 0}
\bigl(\tau-(2n+1) F\bigr)_+^{\frac{1}{2}(d-2)} F,
\label{5-2-3}
\end{gather}
where $\kappa_2=1/(2\pi)$, $\kappa_3=1/(2\pi^2)$. In particular, if $d=2$,
$F\ne 0$ this operator has a  pure point spectrum of infinite multiplicity. Eigenvalues $(2m+1)\mu h F$ are called \emph{Landau levels\/}. If $d=3$, this operator has an absolutely continuous spectrum.

In these cases, the operator (\ref{5-1-2}) is a direct sum of $\D/2$ operators $H_-$ and $\D/2$ operators $H_+$ where $H_\mp = H_0\mp \mu h F$, $H_0$ is the operator (\ref{5-1-1}); then
\begin{equation}
\cN^\MW_d(\tau-V, \mu h F)\Def
\kappa_d D  \Bigl( \frac{1}{2}\tau_+^{\frac{1}{2}(d-2)} +
\sum_{n \ge 1}
\bigl(\tau-2n F\bigr)_+^{\frac{1}{2}(d-2)}\Bigr) F
\label{5-2-4}
\end{equation}

Classical dynamics are different as well: when $d=2$, the trajectories are \emph{magnetrons\/}--circles of radii $(\mu F)^{-1}$, while if $d=3$, there is also free movement along \emph{magnetic lines\/}--integral curves of $\mathbf{F}$, so the trajectories are solenoids.

\subsection{Canonical form}
\label{sect-5-2-2}
Using the $\hslash$-Fourier transform, we can  reduce the magnetic Schr\"odinger operator to its microlocal canonical form
\begin{align}
&\mu^2 \sum_{n\ge 0} B_n (x_1,\hslash D_1, \mu^{-2},\hslash )  \cL_0^n &&\text{for\ \ } d=2,
\label{5-2-5}\\
&h^2D_2^2+\mu^2 \sum_{n\ge 0} B_n (x',\hslash D_1, \mu^{-2},\hslash )  \cL_0^n
&&\text{for\ \ } d=3,
\label{5-2-6}
\end{align}
with $\hslash= \mu^{-1}h$,
$\cL_0= x_d^2+ \hslash^2 D_d^2$, $x'=x_1$ and $x'=(x_1,x_2)$ when $d=2,3$ respectively. Further, the principal symbols of the operators $B_0$ and $B_1$ are $\mu^{-2}V\circ \Psi$ and $F\circ \Psi$ respectively, where $\Psi$ is a diffeomorphism $(x',\xi_1)\to x$.

This canonical form allows us to study both the classical trajectories and the propagation of singularities in the general case. When $d=2$, there is still movement along the magnetrons but magnetrons are drifting with the velocity $\mathbf{v}=\mu^{-1}(\nabla ((V-\tau)/F_{12}))^\perp$ where $^\perp$ denotes the   counter-clockwise rotation by $\pi/2$. If $d=3$, trajectories are solenoids winding around magnetic lines and the movement along magnetic lines is described by an $1$-dimensional Hamiltonian but there there is also side-drift as in $d=2$.

We can replace then $\cL_0$ by its eigenvalues which are $(2j+1)\hslash$, thus arriving to a family of $\hslash$-pseudodifferential operators with respect to $x_1$ if $d=2$ and to a family of $\hslash$-pseudodifferential operators with respect to $x_1$ which is also a Schr\"odinger operator with respect to $x_2$.

\subsection{Asymptotics: moderate magnetic field}
\label{sect-5-2-3}

We can always recover the estimate $O(\mu h^{1-d})$ with the standard Weyl principal part simply by using the scaling $x\to \mu x$, $h\mapsto \mu h$, $\mu\mapsto 1$. On the other hand, for $d=2$, we cannot in general improve it as follows from the example with constant $F$ and $V$.

However, under a(THE?) non-degeneracy assumption, the remainder estimate is much better:

\begin{theorem}\label{thm-5-2-1}
Let $d=2$,  $F\asymp 1$, $\mu h\lesssim 1$   and
\begin{gather}
|\nabla VF^{-1}|+|\det \Hess  VF^{-1}|\ge \epsilon.
\label{5-2-7}\\
\intertext{Then}
\int \Bigl(e(x,x,0) -h^{-2}\cN^\MW_2(x,-V,\mu h F)\Bigr)\psi(x)\,dx= O(\mu^{-1}h^{-1}).
\label{5-2-8}
\end{gather}
\end{theorem}

The explanation is simple: each of the non-degenerate $1$-dimensional $\hslash$-pseudodifferential operators contributes $O(1)$ to the remainder estimate and there is $\asymp (\mu h)^{-1}$ of them which should be taken into account. Another explanation is that under the non-degeneracy assumption, the drift of the magnetrons destroys the periodicity but we can follow the evolution for time
$T^*=\epsilon \mu$, so the remainder estimate is $O(T^{*\,-1}h^{-1})$.

If $d=3$, we cannot get the local remainder estimate better than $O(h^{-2})$ without global non-periodicity conditions due to the evolution along magnetic lines. On the other hand, we do not need strong non-degeneracy assumptions:

\begin{theorem}\label{thm-5-2-2}
Let $d=3$, $F\asymp 1$ and $\mu h\lesssim 1$. Then,
\begin{gather}
\int \Bigl(e(x,x,0) -h^{-3}\cN^\MW_3(x,-V,\mu h F)\Bigr)\psi(x)\,dx=
O(h^{-2}+\mu h^{-1-\delta})
\label{5-2-9}\\
\intertext{in the general case and}
\int \Bigl(e(x,x,0) -h^{-3}\cN^\MW_3(x,-V,\mu h F)\Bigr)\psi(x)\,dx= O(h^{-2}),\hphantom{\mu h^{-1-\delta}}
\label{5-2-10}\\
\intertext{provided}
\sum_{\alpha:\,1\le |\alpha|\le K} |\nabla^\alpha VF^{-1}|\ge \epsilon.
\label{5-2-11}
\end{gather}
Further, in the general case, as(FOR?) $\mu \le h^{-\frac{1}{3}}$, we can replace the magnetic Weyl expression $\cN^\MW_3$ by the standard Weyl expression $\cN_3$.
\end{theorem}

\subsection{Asymptotics: strong magnetic field}
\label{sect-5-2-4}
Let us now consider the strong magnetic field case $\mu h \gtrsim 1$. Then the remainder estimates (\ref{5-2-8}), (\ref{5-2-9}) and (\ref{5-2-10}) acquire a factor of $\mu h^{-1}$:

\begin{theorem}\label{thm-5-2-3}
Let $d=2$, $F\asymp 1$ and $\mu h \gtrsim 1$. Then for the operator \textup{(\ref{5-2-1})},
\begin{enumerate}[label=(\roman*), wide, labelindent=0pt]
\item\label{thm-5-2-3-i}
Under the assumption
\begin{gather}
|\tau V-(2j+1)\mu h F|\ge \epsilon_0\qquad \forall j\in \bZ^+,
\label{5-2-12}
\intertext{the following asymptotics holds:}
e(x,x,\tau)-h^{-2}\cN^\MW_2(x,\tau -V, F) = O(\mu^{-s}h^s).
\label{5-2-13}
\end{gather}
\item
Under the assumption
\begin{multline}
|\tau V-(2j+1)\mu h F|+ |\nabla ((V-\tau)F^{-1})| +\\
|\det \Hess ((V-\tau)F^{-1})|\ge \epsilon_0\qquad
 \forall j\in \bZ^+,
 \label{5-2-14}
\end{multline}
the following asymptotics holds:
\begin{equation}
\int \Bigl(e(x,x,\tau)-h^{-2}\cN^\MW_2(x,\tau -V, F) \Bigr)\psi(x)\,dx = O(1).
\label{5-2-15}
\end{equation}
\end{enumerate}
\end{theorem}

\begin{remark}\label{rem-5-2-4}
If $d=2$, we only need that $\mu^{-1}h\ll 1$ rather than $h\ll 1$.
\end{remark}

\begin{theorem}\label{thm-5-2-5}
Let $d=3$, $F\asymp 1$ and $\mu h \gtrsim 1$. Then for the operator \textup{(\ref{5-2-1})},
\begin{align}
&\int \Bigl(e(x,x,0) -h^{-3}\cN^\MW_3(x,-V,\mu h F)\Bigr)\psi(x)\,dx=
O(\mu h^{-1-\delta})
\label{5-2-16}\\
\intertext{in the general case and}
&\int \Bigl(e(x,x,0) -h^{-3}\cN^\MW_3(x,-V,\mu h F)\Bigr)\psi(x)\,dx=O(\mu h^{-1})
\label{5-2-17}
\end{align}
 under the assumption
 \begin{equation}
|V+(2j+1)\mu h F| +\sum_{\alpha:\,1\le |\alpha|\le K} |\nabla^\alpha VF^{-1}|\ge \epsilon \qquad \forall j\in \bZ^+.
\label{5-2-18}
\end{equation}
\end{theorem}

\begin{remark}\label{rem-5-2-6}
\begin{enumerate}[label=(\roman*), wide, labelindent=0pt]
\item\label{rem-5-2-6-i}
$\cN^\MW_d =O(\mu h)$ for $\mu h\gtrsim 1$.
\item\label{rem-5-2-6-ii}
For the Schr\"odinger-Pauli operator \textup{(\ref{5-2-2})}, one only needs to replace ``$(2j+1)$'' by ``$2j$'' in the assumptions above.
\end{enumerate}
\end{remark}

\section{$2\D$ case, degenerating magnetic field}
\label{sect-5-3}

\subsection{Preliminaries}
\label{sect-5-3-1}

Since $\mu F$ plays such a  prominent role when $d=2$, one may ask what happens if $F$ vanishes somewhere? Obviously, one needs to make certain assumptions; it turns out that in the generic case
\begin{gather}
|F|+|\nabla F|\asymp 1,
\label{5-3-1}\\
\intertext{the \emph{degeneration manifold\/} $\Sigma\Def \{x:\,F(x)=0\}$ is a smooth manifold and the operator is modelled by}
h^2D_1^2 +(hD_2-\mu  x_1^2/2)^2 +V(x_2),
\label{5-3-2}
\end{gather}
which we are going to study. We consider the local spectral asymptotics for $\psi$ supported in a small enough vicinity of $\Sigma$. Under the assumption (\ref{5-3-1}) (or, rather more a general one), the complete analysis was done in Chapter~\ref{book_new-sect-14} of \cite{futurebook}.

\subsection{Moderate and strong magnetic field}
\label{sect-5-3-2}

We start from the case $\mu h\lesssim 1$. Without any loss of the generality, one can assume that $\Sigma=\{x:\, x_1=0\}$. Then, the scaling
$x\mapsto x/\gamma(\bar{x})$ (with $\gamma(x)=\frac{1}{2}\dist(x,\Sigma)$), brings us to the case of the non-degenerate magnetic field with $h\mapsto h_1=h/\gamma$ and  $\mu\mapsto \mu_1=\mu \gamma^2$ as long as $\gamma \ge \mu^{-\frac{1}{2}}$. Then the contribution of $B(x,\gamma(x))$ to the remainder does not exceed $C\mu_1^{-1}h_1^{-1}=C\mu^{-1}h^{-1}\gamma^{-1}$ and the total contribution of the \emph{regular zone\/}
$\cZ=\{\gamma(x)\ge C_0\mu^{-\frac{1}{2}}\}$ does not exceed
$C\int \mu^{-1}h^{-1}\gamma^{-3}\,dx=Ch^{-1}$.

On the other hand, in the \emph{degeneration zone\/}
$\cZ_0=\{\gamma(x)\le C_0\mu^{-\frac{1}{2}}\}$, we use $\gamma=\mu^{-\frac{1}{2}}$ and the contribution of $B(x,\gamma(x))$ does not exceed $Ch_1^{-1}= Ch^{-1}\mu^{-\frac{1}{2}}$ and the total contribution of this zone also does not exceed $Ch^{-1}$.

Thus we conclude that the left-hand expression of (\ref{5-2-8}) is now $O(h^{-1})$. Can we do any better than this?

Analysis of the evolution and propagation in the zone $\cZ$ shows that there is a drift of magnetic lines along $\Sigma$ with speed $C\mu^{-1}\gamma^{-1}$ which allows us to improve $T^*\asymp \mu \gamma^2$  to $T^*\asymp \mu \gamma$ (both before rescaling) and improve the estimate of the contribution of $B(x,\gamma(x))$ to $C\mu^{-1}h^{-1}$ and the total contribution of this zone to
$C\int \mu^{-1}h^{-1}\gamma^{-2}\,dx=C\mu^{-\frac{1}{2}}h^{-1}$.

Analysis of evolution and propagation in the zone $\cZ_0$ is more tricky. It turns out that there are \emph{short periodic trajectories\/} with period $\asymp \mu^{-\frac{1}{2}}$, but there are not many of them which allows us to improve the remainder estimate in this zone as well.

\begin{theorem}\label{thm-5-3-1}
Let $d=2$ and suppose the condition \textup{(\ref{5-3-1})} is fulfilled.
Let $\Sigma\Def \{F=0\}=\{x_1=0\}$ and $-V\asymp 1$. Let
 $\mu \le h^{-1}$. Then,
\begin{enumerate}[label=(\roman*), wide, labelindent=0pt]
\item\label{thm-5-3-1-i}
The left-hand expression of \textup{(\ref{5-2-15})} is $O(\mu^{-\frac{1}{2}}h^{-1}+h^{-1}(\mu^{\frac{1}{2}}h|\log h|)^{\frac{1}{2}})$. In particular, for $\mu \lesssim (h|\log h|)^{-\frac{2}{3}}$, it is $O(\mu^{-\frac{1}{2}}h^{-1})$.
\item\label{thm-5-3-1-ii}
Further,
\begin{multline}
\int \Bigl(e(x,x,0)-h^{-2}\cN^\MW_2(x,-V, \mu h F) \Bigr)\psi(x)\,dx -\\
h^{-1}\int \cN^\MW_\corr (x_2,0)\psi (x_2,0)= O(\mu^{-\frac{1}{2}}h^{-1}+ h^{-\delta})
\label{5-3-3}
\end{multline}
with
\begin{multline}
h^{-1}\cN^\MW_\corr\Def \\
(2\pi h)^{-1}\int \mathbf{n}_0(\xi_2,-W(x_2),\hbar)\,d\xi_2 -
h^{-2}\int \cN^\MW  (-W(x_2), \mu h F(x_1))\, dx_1
\label{5-3-4}
\end{multline}
where $\mathbf{n}_0(\xi_2,\tau,\hbar)$ is an eigenvalue counting function for the operator
$\mathbf{a}_0(\xi_2,\hbar)= {\hbar^2 D_1^2 + (\xi_2-  x_1^2/2 )^2}$ on
$\bR^1\ni x_1$, $\hbar=\mu^{\frac{1}{2}}h$  and $W(x_2)=V(0,x_2)$.
\item\label{thm-5-3-1-iii}
Furthermore, under the non-degeneracy assumption
\begin{equation}
\sum_{1\le k\le m} |\partial_{x_2}^k W|\asymp 1
\label{5-3-5}
\end{equation}
(in the framework of assumption \textup{(\ref{5-3-2})}), one can take $\delta=0$ in \textup{(\ref{5-3-3})}.
\end{enumerate}
\end{theorem}

\begin{remark}\label{rem-5-3-2}
\begin{enumerate}[label=(\roman*), wide, labelindent=0pt]
\item\label{rem-5-3-2-i}
Under some non-degeneracy assumptions, Theorem~\ref{thm-5-3-1}\ref{thm-5-3-1-i} could also be improved.
\item\label{rem-5-3-2-ii}
Theorem~\ref{thm-5-3-1} remains valid for $h^{-1}\le \mu\lesssim h^{-2}$ as well but then the zone $\{x:\,\gamma(x) \ge C( \mu h)^{-1}\}$ is forbidden, contribution to the principal part is delivered by the zone
$\{x:\,\gamma(x) \lesssim (\mu h)^{-1}\}$ and it is $\lesssim \mu^{-1}h^{-3}$.
\item\label{rem-5-3-2-iii}
As $\mu \ge Ch^2$, the principal part is $0$ and the remainder is $O(\mu^{-s})$.
\end{enumerate}
\end{remark}

For further details, generalizations and proofs, see Section~\ref{book_new-sect-14-6} of \cite{futurebook}.

\subsection{Strong and superstrong magnetic field}
\label{sect-5-3-3}

Assume now that $\mu \gtrsim h^{-1}$ and replace $V$ by $V-(2j+1)\mu h F_{12}$ with  $j\in \bZ^+$. Then the zone $\{x:\,\gamma(x) \ge C( \mu h)^{-1}\}$ is no longer forbidden, the principal part of asymptotics is of the magnitude
$\mu h^{-1}$ (cf. Remark~\ref{rem-5-3-2}\ref{rem-5-3-2-ii} and the remainder estimate becomes $O(\mu^{-\frac{1}{2}}h^{-1}+h^{-\delta})$  (and under the non-degeneracy assumption one can take $\delta=0$).

Furthermore,  the case $\mu \ge C h^{-2}$ is no longer trivial. First, one needs to change the correction term by replacing $\mathbf{a}_0$ with
$\mathbf{a}_{2j+1}\Def \mathbf{a}_0-(2j+1) \hbar x_1$. Second, the non-degeneracy condition should be relaxed by requiring (\ref{5-3-5}) only if
\begin{equation}
|\hbar^{2/3}\lambda_{j,n} (\eta)+W(x_2)|+ |\partial_\eta \lambda_{j,n} (\eta)
\ge \epsilon\qquad \forall \eta,
\label{5-3-6}
\end{equation}
fails,  where $\lambda_{j,n}$ are the eigenvalues of $\mathbf{a}_{2j+1}$ with $\hbar=1$.

Furthermore, if
\begin{equation}
|\hbar^{2/3}\lambda_{j,n} (\eta)+W(x_2)|\ge \epsilon\qquad \forall \eta,
\label{5-3-7}
\end{equation}
then $0$ belongs to the spectral gap and the remainder estimate is $O(\mu^{-s})$.

For further details, exact statements, generalizations and proofs, see Sections~\ref{book_new-sect-14-7} and~\ref{book_new-sect-14-8} of \cite{futurebook}.

\section{$2\D$ case, near the boundary}
\label{sect-5-4}

\subsection{Moderate magnetic field}
\label{sect-5-4-1}

We now consider the magnetic Schr\"odinger operator with $d=2$, $F\asymp 1$ in a compact domain $X$ with $\sC^\infty$-boundary. While the dynamics inside  the domain do not change, the dynamics in the boundary layer of the width $\asymp \mu^{-1}$ are completely different. When the magnetron hits $\partial X$, it reflects according to the standard ``incidence angle equals reflection angle'' law and thus the ``particle'' propagates along $\partial X$ with speed $O(1)$ rather than $O(\mu^{-1})$. Therefore, physicists distinguish between \emph{bulk\/} and \emph{edge particles\/}.
Note however that in general, this distinction is not as simple as in the case of constant $F$ and $V$. Indeed, a drifting inner trajectory can hit $\partial X$ and become a \emph{hop trajectory\/}, while the latter could leave the boundary and become an inner trajectory.

It follows from Section~\ref{sect-5-2} that the  contribution of $B(x,\gamma(x))$ with $\gamma(x)=\frac{1}{2}\dist (x,\partial X)\ge \bar{\gamma}=C\mu^{-1}$ to the remainder is $O(\mu^{-1}h^{-1}\gamma^2T(x)^{-1})$, where $T(x)$ is the length of the drift trajectory inside  the \emph{bulk zone\/}  $\{x\in X:\,\gamma(x)\ge \bar{\gamma}\}$. Then the total contribution of this zone to the remainder does not exceed
$C\mu^{-1}h^{-1}\int T(x)^{-1}\,dx=O(\mu^{-1}h^{-1})$ since $T(x)\gtrsim \gamma(x)^{\frac{1}{2}}$ (in the proper direction).

On the other hand, due to the rescaling $x\mapsto x/\bar{\gamma}$, the contribution of  $B(x,\bar{\gamma})$ with $\gamma(x)\le \bar{\gamma}$ to the remainder does not exceed $C\mu h^{-1}\bar{\gamma}^2$ and the total contribution of the \emph{edge zone\/}  $\{x\in X:\,\gamma(x)\le \bar{\gamma}\}$  does not exceed
$C\mu h^{-1}\bar{\gamma}= Ch^{-1}$ and the total remainder is $O(h^{-1})$. Thus, if we want a better estimate, we need to study propagation along $\partial X$.

\begin{theorem}\label{thm-5-4-1}
Under the non-degeneracy assumption $|\nabla _{\partial X}(VF^{-1})|\asymp 1$ on $\supp(\psi)$ (contained in the small vicinity of $\partial X$) for $\mu h\lesssim 1$
\begin{multline}
\int_X \Bigl( e(x,x,0)-h^{-2}\cN^\MW (x,0, \mu h)\Bigr)\psi(x)\,dx -\\ h^{-1}\int_{\partial X} \cN^\MW_{*,\bound}(x,0, \mu h )\psi (x)\,ds_g=
O(\mu^{-1}h^{-1}),
\label{5-4-1}
\end{multline}
where $\cN^\MW_{*,\bound}$ is introduced in \ref{5-4-2-D} or  \ref{5-4-2-N} below for the Dirichlet or Neumann boundary conditions respectively with
$\hbar=\mu h F(x)$ and $\tau$ replaced by $-V(x)$.

Here,
\begin{phantomequation}\label{5-4-2}\end{phantomequation}
\begin{multline}
\cN^\MW_{\D,\bound}(\tau, \hbar)\Def\\
(2\pi)^{-1}  \int_0^\infty \sum_{j\ge 0}\Bigl(
\int \uptheta\bigl(\tau -\hbar \lambda_{\D,j}(\eta)\bigr)
\upsilon_{\D ,j}^2 ( x_1,\eta )\,d\eta -
\uptheta \bigl(\tau -(2j+1)\hbar\bigr)\Bigr)\hbar^{\frac{1}{2}}\, dx_1
\tag*{$\textup{(\ref*{5-4-2})}_\D$}\label{5-4-2-D}
\end{multline}
and
\begin{multline}
\cN^\MW_{\N,\bound}(\tau, \hbar)\Def\\
(2\pi)^{-1} \int_0^\infty \sum_{j\ge 0}\Bigl(
\int \uptheta\bigl(\tau -\hbar \lambda_{\N,j}(\eta)\bigr)
\upsilon_{\N,j}^2 ( x_1,\eta )\,d\eta -
\uptheta \bigl(\tau -(2j+1)\hbar+\bigr)\Bigr)\hbar^{\frac{1}{2}}\, dx_1,
\tag*{$\textup{(\ref*{5-4-2})}_\N$}\label{5-4-2-N}
\end{multline}
where $\lambda_{\D,j}(\eta)$ and $\lambda_{\N,j}(\eta)$ are eigenvalues and
$\upsilon_{\D,j}$ and $\upsilon_{\N,j}$ are eigenfunctions of operator \begin{equation}
\mathbf{a}(\eta, x_1,D_1) = D_1^2 +  x_1^2 \qquad\text{as\ \ } x_1<\eta
\label{5-4-3}
\end{equation}
with the Dirichlet or Neumann boundary conditions respectively at $x_1=\eta$.
\end{theorem}

\begin{remark}\label{rem-5-4-2}
Under weaker non-degeneracy assumptions $|\nabla VF^{-1}|\asymp 1$ and
$\nabla _{\partial X}VF^{-1}=0\implies
\pm \nabla^2 _{\partial X}VF^{-1}\ge \epsilon$, less sharp remainder estimates are derived. The sign in the latter inequality matters since it affects the dynamics. It also matters whether Dirichlet or Neumann boundary conditions are considered: for the Dirichlet boundary condition we get a better remainder estimate.
\end{remark}

For exact statements, generalizations and proofs, see Sections~\href{http://www.math.toronto.edu/ivrii/monsterbook.pdf#section.15.2}{15.2} and~\href{http://www.math.toronto.edu/ivrii/monsterbook.pdf#section.15.3}{15.3} of \cite{futurebook}.

\subsection{Strong magnetic field}
\label{sect-5-4-2}
We now consider the strong magnetic field $\mu h\gtrsim 1$ and for simplicity assume that $F=1$. In this case, we need to study an auxillary operator
$D_1^2 +  (x_1-\eta)^2$ as $x_1<0$ with either the Dirichlet or Neumann boundary conditions at $x_1=0$ or equivalently the operator (\ref{5-4-3}) and our operator is basically reduced to a perturbed operator
\begin{equation}
\mathbf{a}(\hslash D_2, x_1,D_1)-(2j+1)-(\mu h)^{-1}W(x_2)\qquad \hslash=\mu^{-1}h,
\label{5-4-4}
\end{equation}
with $W(x_2)=V(0,x_2)$. Then we need to analyze either $\lambda_{\D,j}(\eta)$ or $\lambda_{\N,j}(\eta)$ more carefully. It turns out that:

\begin{proposition}\label{prop-5-4-3}
$\lambda_{\D, n}(\eta)$ and $\lambda_{\N, n}(\eta)$, $n=0,1,2,\dots$ are real analytic functions with the following properties:

\begin{enumerate}[label=(\roman*), wide, labelindent=0pt]
\item\label{prop-5-4-3-i}
$\lambda_{\D, k}(\eta)$ are monotone decreasing for $\eta\in \bR$;
$\lambda_{\D, k}(\eta)\nearrow +\infty$ as $\eta\to -\infty$;
$\lambda_{\D, k}(\eta)\searrow (2k+1)$ as $\eta\to +\infty$;
$\lambda_{\D, k}(0)=(4k+3)$.

\item\label{prop-5-4-3-ii}
$\lambda_{\N, k}(\eta)$ are monotone decreasing for $\eta\in \bR^-$; $\lambda_{\N, k}(\eta)\nearrow +\infty$ as $\eta\to -\infty$;
$\lambda_{\N, k}(\eta)< (2n+1)$ as $\eta\ge (2n+1)^{\frac{1}{2}}$;
$\lambda_{\N, k}(0)=(4n+1)$.

\item\label{prop-5-4-3-iii}
$\lambda_{\N, k}(\eta) < \lambda_{\D, k}(\eta) <\lambda_{\N, (k+1)}(\eta)$;
$\lambda_{\D, k}(\eta)> (2k+1)$, $\lambda_{\N, n}(\eta)> (2k-1)_+$.
\end{enumerate}
\end{proposition}

\begin{proposition}\label{prop-5-4-4}
\begin{enumerate}[label=(\roman*), wide, labelindent=0pt]
\item\label{prop-5-4-4-i}
$\partial_\eta \lambda_{\D,k} (\eta)<0$.

\item\label{prop-5-4-4-ii}
$\partial_\eta \lambda_{\N,n}(\eta) \gtreqqless 0$ if and only if
\begin{equation}
\lambda_{\N,k} (\eta)\lesseqqgtr\eta^2.
\label{5-4-5}
\end{equation}
\item\label{prop-5-4-4-iii}
$\lambda_{\N,k}(\eta)$ has a single stationary point\footnote{\label{foot-38}
And it must have one due to Proposition~\ref{prop-5-4-3}.} $\eta_k$, it is a non-degenerate minimum, and at this point \textup{(\ref{5-4-5})} holds.

\item\label{prop-5-4-4-iv}
In particular, $(\lambda_{\N,k}(\eta)-\eta^2)$ has the same sign as $(\eta_k-\eta)$.
\end{enumerate}
\end{proposition}

We see the difference between the Dirichlet and Neumann cases because we need non-degeneracy for $\lambda_{*,k}(\hslash D_2)-(2j+1)+W(x_2)$. It also means the difference in the propagation of singularities along $\partial X$: in the Dirichlet case, all singularities move in one direction (constant sign of $\lambda'_{\D,k}$), while in the Neumann case, some move in the opposite direction (variable sign of $\lambda'_{\N,k}$); this effect plays a role also in the  case when $\mu h\lesssim 1$.

Assume here that $\tau$ is in an ``inner'' spectral gap:
\begin{equation}
|(2j+1)\mu h  +V-\tau|\ge \epsilon_0\mu h \qquad \forall x\
\forall j\in \bZ^+.
\label{5-4-6}
\end{equation}

\begin{theorem}\label{thm-5-4-5}
Suppose that $\mu h\gtrsim 1$ and the condition  \textup{(\ref{5-4-6})} is fulfilled.  Then,
\begin{enumerate}[label=(\roman*), wide, labelindent=0pt]
\item\label{thm-5-4-i}
In case of the  Dirichlet boundary condition, the left-hand expression in \textup{(\ref{5-4-1})} is $O(1)$.
\item\label{thm-5-4-5-ii}
In case of the Neumann boundary condition, assume additionally that
\begin{phantomequation}\label{5-4-7}\end{phantomequation}
\begin{multline}
|\bigl(\lambda_{\N,j} (\eta)  \mu h +V -\tau| \le \epsilon_0 \mu h,\\[2pt]   |\lambda'_{\N,j} (\eta)|+ |\partial_{x_2} V|\le \epsilon_0 \implies
\pm \partial_{x_2}^2 V \ge \epsilon_0
\qquad
\forall j=0,1,2,\ldots
\tag*{$\textup{(\ref*{5-4-7})}_\pm$}\label{5-4-7-*}
\end{multline}
Then, the left-hand expression in \textup{(\ref{5-4-1})} is $O(1)$ under the assumption $\textup{(\ref{5-4-7})}_+$ and $O(\log h)$ under the
assumption $\textup{(\ref{5-4-7})}_-$.
\end{enumerate}
\end{theorem}

\begin{remark}\label{rem-5-4-6}
\begin{enumerate}[label=(\roman*), wide, labelindent=0pt]
\item\label{rem-5-4-6-i}
If (\ref{5-4-6}) is fulfilled for all $\tau\in[\tau_1,\tau_2]$, then
the asymptotics is ``concentrated'' in the boundary layer.
\item\label{rem-5-4-6-ii}
For a more general statement when (\ref{5-4-6}) fails (i.e.~when $\tau$ is no longer in the ``inner'' spectral gap) and is replaced with the condition
\begin{equation}
|V+(2j+1)\mu h -\tau|+|\nabla V|\asymp 1
\label{5-4-8}
\end{equation}
on $\supp (\psi)$, see  Theorem~\href{http://www.math.toronto.edu/ivrii/monsterbook.pdf#theorem.15.4.18}{15.4.18} of \cite{futurebook}.
\end{enumerate}
\end{remark}

\section{Pointwise asymptotics and short loops}
\label{sect-5-5}

We are now interested in the pointwise asymptotics inside the domain. Surprisingly, it turns out that the standard Weyl formula for this purpose is better than the Magnetic Weyl formula.

\subsection{Case $d=2$}
\label{sect-5-5-1}

We start from the case $d=2$, $F=1$, $|\nabla V|\asymp 1$. One can easily see  that in classical dynamics,  short loops of the lengths $\asymp \mu^{-1} n$ with $n=1, \ldots, N$, $N\asymp \mu$ appear. We would like to understand how these loops affect the asymptotics in question.

\begin{theorem}\label{thm-5-5-1}%\label{thm-16-2-19}
For the magnetic Schr\"odinger operator which satisfies the above assumptions in a domain $X\subset \bR^2$, with $B(0,1)\subset X$, the following estimates hold at a point $x\in B(0,\frac{1}{2})$

\begin{enumerate}[label=(\roman*), wide, labelindent=0pt]
\item\label{thm-5-5-1-i}
For $1\le \mu\le h^{-\frac{1}{2}}$,
\begin{equation}
|e(x,x,\tau)- h^{-2}\cN^\W_x (\tau)|\le C\mu^{-1}h^{-1}+C\mu^{\frac{1}{2}}h^{-\frac{1}{2}}+ C\mu^2 h^{-\frac{1}{2}}
\label{5-5-1}
\end{equation}
and
\begin{multline}
\R^\W_{x(r)}\Def |e(x,x,0)-
h^{-2}\bigl(\cN^\W_x (0)+\cN^\W_{x,\corr(r)} (0)\bigr)|\le\\[3pt]
C\mu^{-1}h^{-1}+ C\mu^{\frac{1}{2}}h^{-\frac{1}{2}} +
C\mu h^{-1} \bigl(\mu^2 h)^{r+\frac{1}{2}}+ \\[3pt]
C\left\{\begin{aligned}
&\Bigr( h^{-1}\bigl(h \mu^{\frac{5}{2}}\bigr)^{r+\frac{1}{2}} +
\mu ^{\frac{1}{3}} h^{-\frac{2}{3}}\Bigr)\qquad&&\text{as \ \ } \mu\le h^{-\frac{2}{5}},\\
&\mu^{\frac{5}{3}}h^{-\frac{1}{3}}\qquad&&\text{as \ \ } \mu\ge h^{-\frac{2}{5}}
\end{aligned}\right.
\label{5-5-2}
\end{multline}
where  $\cN^\W_{x,\corr(r)}$ is the $r$-term stationary phase approximation to some explicit oscillatory integral (see Section~\ref{book_new-sect-16-2} of \cite{futurebook}).

\item\label{thm-5-5-1-ii}
For $h^{-\frac{1}{3}} \le \mu\le h^{-1}$,
\begin{multline}
\R^{\W\prime\prime}_{x(r)} \Def \\
| \Bigl(e  (x,x,\tau) -h^{-2}\cN_{x,\corr(r)}-
\bar{e}_x  (x,x,\tau) +h^{-2}\bar{\cN}_{x,\corr(r)}\Bigr)|\le\allowdisplaybreaks \\
C\mu^{\frac{1}{2}}h^{-\frac{1}{2}}+
C\left\{\begin{aligned}
&\mu^{-2} h^{-2}  (\mu^2h )^{r+\frac{1}{2}} \qquad&&\text{for\ \ } \mu \le h^{-\frac{1}{2}}\\
&h^{-1} \qquad
	&&\text{for\ \ }  \mu \ge h^{-\frac{1}{2}}
\end{aligned}\right.\qquad \ \\
+ C\left\{\begin{aligned}
& h^{-1} \bigl(\mu ^{\frac{5}{2}}h \bigr)^{r+\frac{1}{2}}+ \mu^{\frac{1}{3}}  h^{-\frac{2}{3}} \qquad
	&&\text{for\ \ } \mu\le h^{-\frac{2}{5}}\\
&\mu^{\frac{5}{3}}h^{-\frac{1}{3}} \qquad
	&&\text{for\ \ }  h^{-\frac{2}{5}}\le \mu\le h^{-\frac{1}{2}}\\
&\mu^{-\frac{1}{3}}h^{-\frac{4}{3}} \qquad
	&&\text{as\ \ } \mu \ge h^{-\frac{1}{2}}
\end{aligned}\right.
\label{5-5-3}
\end{multline}
while for $r=0$,
\begin{multline}
\R^{\W\prime\prime}_{x(0)}\Def |e  (x,x,\tau) - \bar{e}_x  (x,x,\tau)|\le C\mu^{\frac{1}{2}}h^{-\frac{1}{2}}+ \\
C\left\{\begin{aligned}
& h^{-1}\mu^{\frac{1}{2}}\qquad
&&\text{for\ \ }\mu\le h^{-\frac{1}{2}},\\
&\mu ^{-\frac{1}{2}}h^{-\frac{3}{2}}
&&\text{for\ \ }\mu\ge h^{-\frac{1}{2}}
\end{aligned}\right.
\label{5-5-4}
\end{multline}
where here and in \ref{thm-5-5-1-iii}, $\bar{e}_y$ is constructed for the toy model in $y$ (with $F=F(y)$ and $V(x)=V(y)+ \langle \nabla V(y),x-y\rangle$).

\item\label{thm-5-5-1-iii}
For  $\mu \ge h^{-1}$, $\tau \le c\mu h$,
\begin{equation}
|e (x,x,\tau)-\bar{e}_x  (x,x,\tau)|\le C\mu^{\frac{1}{2}}h^{-\frac{1}{2}}.
\label{5-5-5}
\end{equation}
\end{enumerate}
\end{theorem}

\subsection{Case $d=3$}
\label{sect-5-5-2}

For $d=3$, we cannot expect a remainder estimate better than $O(h^{-1})$. On the other hand, the purely Weyl approximation has a better chance to succeed as the loop condition now includes returning free movement along the magnetic line in addition to the returning circular movement. We formulate only one theorem out of many from  Section~\ref{book_new-sect-16-6} of \cite{futurebook}:

\begin{theorem}\label{thm-5-5-2}%\label{thm-16-6-8}
Let  $d=3$. Then,

\begin{enumerate}[label=(\roman*), wide, labelindent=0pt]
\item\label{thm-5-5-2-i}
In the general case,
\begin{equation}
|e (x,x,\tau)- h^{-3}\cN_x^\W (x,x,\tau)|\le Ch^{-2}+  C\mu^{\frac{3}{2}}h^{-\frac{3}{2}}.
\label{5-5-6}
\end{equation}
\item\label{thm-5-5-2-ii}
Under non-degeneracy condition
\begin{equation}
|\nabla_{\perp \mathbf{F}} (V-\tau)/F|\asymp 1,
\label{5-5-7}
\end{equation}
where $\nabla_\perp$ is the component of the gradient perpendicular to $\mathbf{F}$, for $\mu \le h^{-\frac{1}{2}}$, we have the estimates
\begin{equation}
|e (x,x,\tau)- h^{-3}\cN_x^\W (x,x,\tau)|\le Ch^{-2}+ C\mu^{\frac{5}{2}}h^{-1}
\label{5-5-8}
\end{equation}
and
\begin{multline}
|e (x,x,\tau)- h^{-3}\cN_x^\W (x,x,\tau)-h^{-3}\cN _{x,\corr(r)}|\le\\
Ch^{-2}+ C\mu^{\frac{3}{2}}h^{-\frac{3}{2}}(\mu^2h)^{r+\frac{1}{2}}.
\label{5-5-9}
\end{multline}
Here we use stationary phase approximations again.
\end{enumerate}
\end{theorem}

\begin{remark}\label{rem-5-5-3}
One can also consider the cases $h^{-\frac{1}{2}}\lesssim \mu \lesssim h^{-1}$ and $\mu \gtrsim h^{-1}$. But then one needs  to include the toy model expression (with constant $\mathbf{F}$ and $\nabla V$) into the approximation.
\end{remark}

\subsection{Related asymptotics}
\label{sect-5-5-3}

Apart from the pointwise asymptotics, one can consider the related asymptotics of
\begin{equation}
\int \omega (\frac{1}{2}(x+y),x-y) e(x,y,\tau) e(y,x,\tau)\,dxdy
\label{5-5-10}
\end{equation}
and estimates of
\begin{multline}
\int \omega (\frac{1}{2}(x+y),x-y)\bigl(e(x,x,\tau)-e^\W(x,x,\tau)\bigr)\times\\
\bigl(e(y,y,\tau)-e^\W(y,y,\tau)\bigr)\,dxdy.
\label{5-5-11}
\end{multline}
For all the details, see Chapter~\ref{book_new-sect-16} of \cite{futurebook}. These expressions play important role in Section~\ref{sect-7}.

\section{Magnetic Dirac operators}
\label{sect-5-6}

We discuss the magnetic Dirac operators
\begin{gather}
H = ((-ih\nabla - \mu A (x))\cdot \boldupsigma)+\upsigma_0 M +V(x)
\label{5-6-1}\\
\shortintertext{and}
H = ((-ih\nabla - \mu A (x))\cdot \boldupsigma)+V(x).
\label{5-6-2}
\end{gather}
If $d=3$, for the second operator, we can consider $2\times 2$ matrices rather than $4\times 4$ matrices.

If $V=0$ then $H^2$ equals to the Schr\"odinger-Pauli operator (plus $M^2$) and therefore the theory of magnetic Dirac and Schr\"odinger operators are closely connected. If $V=0$ and $0\ne \mathbf{F}$ is constant then the operator for $d=2$ has a pure point spectrum of infinite multiplicity consisting of
$\pm \sqrt{M^2 +2j\mu h F}$ with $j\in \bZ^+$\,\footnote{\label{foot-39} However, one of the points $\pm M$ is missing depending on whether $F_{12}\gtrless 0$ and $\upsigma_1\upsigma_2 \upsigma_3=\pm i$.} and for $d=3$, this operator has absolutely continuous spectrum $(-\infty,-M]\cup[M,\infty)$. Thus we get corresponding \emph{Landau levels\/}.

The results similar to those of Section~\ref{sect-5-2} hold (for details, exact statements and proofs, see Chapter~\ref{book_new-sect-17} of \cite{futurebook}. Further, the results of Sections~\ref{sect-5-3} and~\ref{sect-5-5} probably are not difficult to generalize, and maybe the results of Section~\ref{sect-5-4} as well under correctly posed boundary conditions.\enlargethispage{\baselineskip}

\chapter{Magnetic Schr\"odinger operator. II}
\label{sect-6}

\section{Higher dimensions}
\label{sect-6-1}

\subsection{General theory}
\label{sect-6-1-1}

We can consider a magnetic Schr\"odinger (and also Schr\"odinger-Pauli and Dirac) operators in higher dimensions. In this case, the magnetic intensity is characterized by the skew-symmetric matrix $F_{jk}=\partial_j A_k-\partial_k A_j$ rather than by a pseudo-scaler $F_{12}$ or a pseudo-vector $\mathbf{F}$. As a result, the magnetic Weyl expression becomes more complicated; as before, it is exactly $e(x,x,\tau)$ for the operator in $\bR^d$ if $F_{jk}$ and $V$ constant:
\begin{multline}
h^{-d}\cN_d^\MW (\tau)\Def  \\
(2\pi )^{-r} \mu ^rh^{-r}
\sum _{\alpha \in \bZ^{+r}}
\Bigl(\tau - \sum_j(2\alpha_j +1) f_j\mu h -V\Bigr)_+^{\frac{1}{2}(d-2r)}
f_1\cdots f_r\sqrt{g},
\label{6-1-1}
\end{multline}
where $2r=\rank (F_{jk})$ and $\pm if_j$ (with $j=1,\ldots, r$, $f_j>0$) are its eigenvalues\footnote{\label{ft-38} In the general case, the eigenvalues of $(F^j_k)=(g^{jl})(F_{lk})$ where $(g^{jk})$ is a metric.} which are not $0$; recall that $z_+^0=\uptheta(z)$.

One can see that $H$ has pure point of infinite multiplicity spectrum when $d=2r$ and $H$ has an absolutely continuous spectrum when $d>2r$. In any case, the bottom of the spectrum is $\mu h(f_1+\ldots+f_r)$.

We are interested in the asymptotics with the sharpest possible remainder estimate like the one for $d=2$ or $d=3$ in the cases $d=2r$ and $d>2r$ respectively\footnote{\label{ft-39} We assume that $(F_{jk}(x))$ has constant rank.}. Asymptotics without a remainder estimate were derived in \cite{Rai1, Rai2}.

As we try to reduce the operator to the canonical form, we immediately run into problem of \emph{resonances\/} when $f_1m_1+\ldots+f_rm_r=0$ at some point with
$m\in \bZ^d$; $|m_1|+\ldots+|m_r|$ is \emph{an order of resonance\/}. If the lowest order of resonances is $k$, then we can reduce the operator to its canonical form modulo $O(\mu^{-k})$ for $\mu h\le 1$ (when $\mu h\ge 1$, this problem is less acute).

It turns out, however, that we can deal with an incomplete canonical form (with a sufficiently small remainder term).

Another problem is that we cannot in general assume that the $f_j$ are constant (for $d=2,3$, we could achieve this by multiplying the operator by $f_1^{-1/2}$) and the microhyperbolicity condition becomes really complicated.

\subsubsection{Case $2r=d$}
\label{sect-6-1-1-1}
In this case, only the resonances of orders $2,3$ matter as we are looking for an error $O(\mu^{-1}h^{1-d})$ when $\mu h\le 1$. If the magnetic field is weak enough, we use the incomplete canonical form only to study propagation of the singularities and if the magnetic field is sufficiently strong, the omitted terms $O(\mu^{-4})$ in the canonical form are small enough to be neglected.

As a result, the indices $j=1,\ldots,r $ are broken into several groups (indices $j$ and $k$ belong to the same group if they ``participate'' in the resonance of order $2$ or $3$ after all the reductions). Then under a certain non-degeneracy assumption called $\fN$-microhyperbolicity  (see Definition~\ref{book_new-def-19-2-5} of \cite{futurebook}) in which this group partition plays a role,  we can recover the remainder estimate $O(\mu^{-1}h^{1-d})$ for $\mu h\lesssim 1$ (in which case, the principal part has magnitude $h^{-d}$) and  $O(\mu^{r-1}h^{1-d+r})$ for $\mu h\gtrsim 1$ (in which case, the principal part has magnitude $\mu^r h^{r-d}$). If we ignore the resonances of order $3$, we get a partition into smaller groups and we need a weaker non-degeneration assumption called microhyperbolicity  (see Definition~\ref{book_new-def-19-2-4} of \cite{futurebook}) but the remainder estimate would be less sharp.

In an important special case of constant $f_1,\ldots, f_r$, both these conditions are equivalent to $|\nabla V|$ disjoint from $0$ (on $\supp(\psi)$) but it could be weakened to $V$ having only non-degenerate critical points (if there  are saddles, we need to  add a logarithmic factor to the remainder estimate).

On the other hand, without any non-degeneracy assumptions, the remainder estimate can be as bad as $O(\mu h^{1-d})$ for $\mu h\lesssim 1$ and as bad as the principal part itself for $\mu h\gtrsim 1$.

For exact statements, details, proofs and generalizations,  see Chapter~\ref{book_new-sect-19} of \cite{futurebook}.

\subsubsection{Case $2r<d$}
\label{sect-6-1-1-2}
In this case, only the resonances of order $2$ matter since we are expecting a larger error than in the previous case. As a result, the indices $j=1,\ldots,r $ are broken into several groups (indices $j$ and $k$ belong to the same group if $f_j=f_k$).

Then under a certain non-degeneracy assumption called microhyperbolicity  (see Definition~\ref{book_new-def-20-1-2} of \cite{futurebook}) in which this group partition plays a role, we can  recover the remainder estimate $O(h^{1-d})$ for
$\mu h\lesssim 1$ (then the principal part is of the magnitude $h^{-d}$) and  $O(\mu^{r}h^{1-d+r})$ for $\mu h\gtrsim 1$ (then the principal part is of the  magnitude $\mu^r h^{r-d}$).

In an important special case of constant $f_1,\ldots, f_r$, this condition is equivalent to $|\nabla V|$ being disjoint from $0$ (on $\supp(\psi)$) but it could be weakened to  ``$\nabla V =0\implies \Hess V$ has a positive eigenvalue'';
if we assume only that ``$\nabla V =0\implies \Hess V$ has a non-zero eigenvalue'', but  we need to  add a logarithmic factor to the remainder estimate.

As expected, for $2r=d-1$, we can recover less sharp remainder estimates even without any non-degeneracy assumptions and for $2r\le d-2$, we do not need any non-degeneracy conditions at all. For exact statements, details, proofs and generalizations,  see Chapter~\ref{book_new-sect-20} of \cite{futurebook}.

\subsection{Case $d=4$: more results}
\label{sect-6-1-2}
This case is simpler than the general one since we have only $f_1$ and $f_2$ and resonance happens if either $f_1=f_2$ or $f_1=2f_2$ (or $f_2=2f_1$).

If we assume that the magnetic intensity matrix $(F_{jk})$ has constant rank $4$, this case is simpler than the general one ($d=2r$) and we can recover sharp remainder estimates under less restrictive conditions. For exact statements, details, proofs and generalizations,  see Chapter~\ref{book_new-sect-22} of \cite{futurebook}.\enlargethispage{\baselineskip}

On the other hand, if we consider $(F_{jk})$ of variable rank, then in the generic case, it has the  eigenvalues $\pm i f_1$ and $\pm i f_2$ and $\Sigma=\{x:\, f_1(x)=0\}$ is a $\sC^\infty$ manifold of dimension $3$, $\nabla f_1 \ne 0$ on $\Sigma$, while $f_2$ is disjoint from $0$. It is similar to a $2$-dimensional operator which we considered in Section~\ref{sect-5-3}, although with a twist: the symplectic form restricted to $\Sigma$ has rank $2$ everywhere except on a $1$-dimensional submanifold $\Lambda$ where it has rank $0$.

Our results are also similar to those of Section~\ref{sect-5-3} with rather obvious modifications but the proofs are more complicated.

For exact statements, details, proofs and generalizations,  see Chapter~\ref{book_new-sect-21} of \cite{futurebook}.

\section{Non-smooth theory}
\label{sect-6-2}

As in Chapter~\ref{sect-4}, we do not need to assume that the coefficients are very smooth. As before, we bracket the operator in question between two ``rough'' operators with the same asymptotics and with sharp remainder estimates. However, the lack of sufficient smoothness affects the reduction to the canonical form: it will be incomplete even if there are no resonances. Because of this, to get as sharp asymptotics as in the smooth case, we need to request more smoothness than in Chapter~\ref{sect-4}.

\subsubsection{Case $d=2$}
\label{sect-6-2-0-1}
For $d=2$, we require smoothness of $F_{12}$, $g^{jk}$ and $V$  marginally larger than $\sC^2$ to recover the same remainder estimate as in the smooth case, but there is a twist: unless the smoothness is $\sC^3$, a correction term needs to be included. This is due to the fact that $V(x)$ and $W(x)$ differ and a more precise formula should use $W(x)$ rather than $V(x)$. Here, $W$ is  $V$ averaged along a magnetron with center $x$.  In fact, it is possible to consider $V$ of the lesser smoothness than $\sC^2$ (but marginally better than $\sC^1$), but one gets a worse remainder estimate. For exact statements, details and proofs, see Chapter~\ref{book_new-sect-18} of \cite{futurebook} and especially Section~\ref{book_new-sect-18-5}.

\subsubsection{Case $d=3$}
\label{sect-6-2-0-2}
Results are similar to those in the smooth case. However, in this case, if we assume no non-degeneracy conditions then the exponent $\delta$ in the  estimates (\ref{5-2-9}) and (\ref{5-2-16}) depends on the smoothness and if we assume a non-degeneracy condition (\ref{5-2-11}) or (\ref{5-2-18}) then obviously $K$ depends on the smoothness. Under the non-degeneracy assumption with $K=1$, we need smoothness marginally better than $\sC^1$ but again unless the smoothness is $\sC^2$, we need to use the averaged potential $W(x)$ rather than $V(x)$.

For exact statements, details and proofs, see Chapter~\ref{book_new-sect-18} of \cite{futurebook} and especially Section~\ref{book_new-sect-18-9}.

\subsubsection{Case $d\ge 4$}
\label{sect-6-2-0-3}

Basically, the results are similar to those for $d=2$ (if $\rank (F_{jk})=d$) or for $d=3$ (if $\rank (F_{jk})<d$), but we cannot recover the sharp remainder estimate if the  smoothness of $V$ is less than $\sC^3$ or $\sC^2$ respectively because we cannot replace $V(x)$ by its average $W(x)$ in the canonical form.

For exact statements, details and proofs, see Chapters~\ref{book_new-sect-19} (if $\rank (F_{jk})=d$) and Chapters~\ref{book_new-sect-20} (if $\rank (F_{jk})<d$) of \cite{futurebook}.

\section{Global asymptotics}
\label{sect-6-3}

For magnetic Schr\"odinger and Dirac operators, one can derive results similar to those of Sections~\ref{sect-3-1} and~\ref{sect-3-2}. We describe here only some results which are very different from those already mentioned and only for the Schr\"odinger operator.

\subsection{Case $d=2r$}
\label{sect-6-3-1}

Assume that $F_{jk}=\const$, $\rank (F_{jk})=2r=d$, $\mu=h=1$ and $V$ decays at infinity. Then instead of an eigenvalue of infinite multiplicity $\lambda_{j,\infty}=(2j_1+1)f_1+\ldots+(2j_r+1)f_r$ (with $j\in \bZ^{+\,r}$), we have  a sequence of eigenvalues $\lambda_{j,n}$ tending to $\lambda_{j,\infty}$ as
$n\to \infty$ and we want to consider the asymptotics of
$\N^-_j (\eta)$ which is the number of eigenvalues in $(\lambda_{j,\infty}-\epsilon, \lambda_{j,\infty}-\eta)$ and
$\N^+_j (\eta)$ which is the number of eigenvalues in
$(\lambda_{j,\infty}+\eta, \lambda_{j,\infty}+\epsilon)$, as $\eta \to +0$.

It turns out that in contrast to the Schr\"odinger operator without a magnetic field, there are meaningful results no matter how fast $V$ decays.

\begin{theorem}\label{thm-6-3-1}
Let us consider a Schr\"odinger operator in $\bR^2$ satisfying the above conditions, $\mu=h=1$ and
\begin{equation}
|\nabla^\alpha V|\le c_\alpha \rho^2\gamma^{-|\alpha|},
\label{6-3-1}
\end{equation}
with $\rho=\langle x\rangle^m$, $\gamma=\langle x\rangle$, $m<0$. Let
\begin{phantomequation}\label{6-3-2}\end{phantomequation}
\begin{equation}
\mp V\ge -\epsilon \rho^2\implies
|\nabla V|\ge \epsilon \rho^2\gamma ^{-1}\qquad
\text{for\ \ } |x|\ge c.
\tag*{$\textup{(\ref*{6-3-2})}_{\mp }$}\label{6-3-2-*}
\end{equation}
\enlargethispage{2\baselineskip}

Then,
\begin{enumerate}[label=(\roman*), wide, labelindent=0pt]
\item\label{thm-6-3-1-i}
The asymptotics
\begin{phantomequation}\label{6-3-3}\end{phantomequation}
\begin{phantomequation}\label{6-3-4}\end{phantomequation}
\begin{gather}
\N^\mp_j (\eta )=\cN ^{\mp}  (\eta )+O(\log \eta )
\tag*{$\textup{(\ref*{6-3-3})}_{\pm}$}\label{6-3-3-*}\\
\shortintertext{hold with}
\cN ^{\mp}  (\eta )=
\frac{1}{2\pi } \int_{\{\mp V>\eta \}} F \,  dx
\tag*{$\textup{(\ref*{6-3-4})}_{\pm}$}\label{6-3-4-*}
\end{gather}
and in our conditions $\cN ^{\mp }(\eta )=O(\eta ^{1/m})$. Moreover,
$\cN ^\mp (\eta )\asymp \eta ^{1/m}$ provided that $\mp V\ge \epsilon \rho^2$ for $x\in \Gamma$, where $\Gamma $ is a non-empty open sector (cone) in $\bR^2$.

\item\label{thm-6-3-1-ii}
Furthermore, if
\begin{phantomequation}\label{6-3-5}\end{phantomequation}
\begin{equation}
\mp V\ge \epsilon \rho^2 \qquad \text{for\ \ } |x|\ge c,
\tag*{$\textup{(\ref*{6-3-5})}_{\pm}$}\label{6-3-5-*}
\end{equation}
then the remainder estimate is $O(1)$. In this case, the points $(2j+1)F \pm 0$ are not limit points of the discrete spectrum.
\end{enumerate}
\end{theorem}

This theorem is proved by rescaling the results of Subsection~\ref{sect-5-2-4} which do not require  $h\ll 1$, but only
$\mu h \gtrsim 1$ and $\mu^{-1}h\ll 1$ (see Remark~\ref{rem-5-2-4}); in our case, after rescaling $\mu =1/\rho\gamma$ and $\mu =\gamma/\rho$, so that  $\mu h=1/\rho^2$ and $\mu^{-1}h=\gamma^{-2}$. Therefore,  the remainder does not exceed
$\int \gamma^{-2}\,dx$, where we integrate over $\{x:\, |V(x)|\ge (1-\epsilon) \eta\}$ in the general case and over\linebreak $\{x:\,  (1+\epsilon) \eta \ge |V(x)|\ge (1-\epsilon) \eta\}$ under the assumption \ref{6-3-5-*}.

\begin{remark}\label{rem-6-3-2}
\begin{enumerate}[label=(\roman*), wide, labelindent=0pt]
\item\label{rem-6-3-2-i}
Similar results hold in the $d$-dimensional case ($d\ge 4$) when $(F)=\const$ and $d=\rank (F)$: the remainder is $O(\eta^{(d-2)/2m})$ and the principal part
is $O(\eta^{d/2m})$.

\item\label{rem-6-3-2-ii}
One can consider the case $\rho = \exp (-\langle x\rangle ^m)$,
$\gamma= \langle x\rangle ^{1-m}$, $0<m<1$, and recover remainder estimate $O(|\log \eta|^{(d-2)/m+2})$ in the general case and
$O(|\log \eta|^{(d-2)/m+1})$ under the assumption \ref{6-3-5-*}
with $\cN^\mp(\eta)=O(|\log \eta|^{d/m})$.
\item\label{rem-6-3-2-iii}
On the other hand, the cases when $V$ decays like $\exp (-2\langle x\rangle)$ or faster,  or is compactly supported,  are out of reach of our methods but the asymptotics (without a remainder estimate) were obtained in \cite{MR, RT, RW}.
\end{enumerate}
\end{remark}

For exact statements, details, proofs and generalizations for arbitrary $\rank (F_{jk})=d=2r$, see Subsection~\ref{book_new-sect-23-4-1}  of \cite{futurebook}.

\subsection{Case $d>2r$. I}
\label{sect-6-3-2}
This case is less ``strange'' than case $d=2$. Here, we can discuss only the eigenvalue counting function  $\N^-_0(\eta)$.
\begin{theorem}\label{thm-6-3-3}
Let us consider a Schr\"odinger operator in $\bR^3$ satisfying  $\mathbf{F} =\const$ and \textup{(\ref{6-3-1})} with  $\rho=\langle x\rangle^m$,
$\gamma=\langle x\rangle$, $m\in (-1,0)$.
Then,
\begin{enumerate}[label=(\roman*), wide, labelindent=0pt]
\item\label{thm-6-3-3-i}
The asymptotics
\begin{phantomequation}\label{6-3-6}\end{phantomequation}
\begin{gather}
\N^-_0 (\eta )=\cN ^{-}  (\eta )+O(\eta^{\frac{1}{m}-\delta})
\tag*{$\textup{(\ref*{6-3-6})}_-$}\label{6-3-6-*}\\
\shortintertext{hold with}
\cN ^-  (\eta )=
\frac{1}{2\pi^2 } \int  F (-V-\eta)_+^{\frac{1}{2}}\,  dx
\label{6-3-7}
\end{gather}
and arbitrarily small $\delta>0$, and furthermore,
$\cN ^-(\eta )=O(\eta ^{\frac{3}{2m} +\frac{1}{2}})$. Moreover,
$\cN ^- (\eta )\asymp \eta ^{\frac{3}{2m}+\frac{1}{2}}$, provided
$-V\ge \epsilon \rho^2$ for $x\in \Gamma $ where $\Gamma $ is a non-empty open cone in $\bR^3$.

\item\label{thm-6-3-3-ii}
Further, under the assumption
\begin{equation}
\sum_{|\alpha|\le K}|\nabla^\alpha V|\cdot \gamma ^{|\alpha|}\ge \epsilon \rho^2
\qquad \text{for\ \ }|x|\ge c,
\label{6-3-8}
\end{equation}
the asymptotics \textup{(\ref{6-3-7})} hold with $\delta=0$.
\end{enumerate}
\end{theorem}

\begin{remark}\label{rem-6-3-4}
\begin{enumerate}[label=(\roman*), wide, labelindent=0pt]
\item\label{rem-6-3-4-i}
Similar results hold in the $d$-dimensional case when $(F_{jk})=\const$ and
$d>2r=\rank (F_{jk})$: the remainder estimate is \\
$O(\eta^{(d-2r-1)/2 + (d-1)/2m-\delta})$\,\footnote{\label{ft-40} Where $\delta=0$ if either $d\ge 2r+2$ or the assumption (\ref{6-3-8}) is fulfilled.} and $\cN^-(\eta)=O(\eta^{(d-2r)/2+d/2m})$.
\item\label{rem-6-3-4-ii}
Observe that for $m=-1$, both the principal part and the remainder estimate have magnitude $\eta^{-r}$.
\item\label{rem-6-3-4-iii}
One can also consider
$\rho=\langle  x\rangle ^{-1}|\log \langle x\rangle|^\alpha$ with $\alpha>0$.
\end{enumerate}
\end{remark}

For exact statements, details, proofs and generalizations, see Subsection~\ref{book_new-sect-24-4-1} of \cite{futurebook}.

\subsection{Case $d>2r$. II}
\label{sect-6-3-3}

We now discuss faster decaying potentials. Assume that $d=2r+1$ (otherwise there will be no interesting results). Assume for simplicity that $g^{jk}=\updelta_{jk}$ and $F_{d k}=0$.  Further, one can assume that $A_d(x)=0$; otherwise one can achieve it by a gauge transformation. Then, $A_j=A_j(x')$ with
$x'=(x_1,\ldots, x_{2r})$ and the operator is of the form
\begin{equation}
D_d^2 +V(x)+ H'_0, \qquad\text{with\ \ }
H'_0\Def \sum_{1\le j\le d-1} (D_j -A_j(x'))^2.
\label{6-3-9}
\end{equation}
For any fixed $x':|x'|\ge c$, consider the one-dimensional operator
\begin{equation}
L\Def D_t^2 +V(x';t)
\label{6-3-10}
\end{equation}
on $\bR\ni t$.
It turns out that under the assumption
\begin{equation}
|V(x';t)|\le \varepsilon t^{-2},
\label{6-3-11}
\end{equation}
with $\varepsilon \le (\frac{1}{4}-\epsilon)$, this operator has no more than one negative eigenvalue $\lambda (x')$; moreover,
it has exactly one negative eigenvalue
\begin{gather}
\lambda (x')= -\frac{1}{4} W(x')^2 + O(\varepsilon^3),
\label{6-3-12}\\
\shortintertext{provided}
W(x')\Def \int_{\bR} V(x';t) <0\qquad\text{and\ \ } -W(x')\asymp \varepsilon.
\label{6-3-13}
\end{gather}
Furthermore, in this case $\lambda (x')$ nicely depends on $x'$.

Let
\begin{equation}
|\nabla^\alpha V|\le c_\alpha \rho^2\gamma_1^{-|\alpha'|}\gamma ^{-\alpha_d},
\label{6-3-14}
\end{equation}
with $\rho=\langle x\rangle^l \langle x'\rangle ^k$,
$\gamma=\langle x\rangle$, $\gamma_1=\langle x'\rangle$ and $l\le -2$,
$m\Def 2l+2k+1<0$
and if
\begin{equation}
W(x') <0, \qquad W(x')\asymp \rho' , \qquad
\rho'\Def \langle x'\rangle ^{m},
\label{6-3-15}
\end{equation}
then we are essentially in the $(d-1)$-dimensional case of an operator
$H'\Def H'_0 +\lambda(x')$, and for $\N^-(\eta)$, we have the corresponding asymptotics of Subsubsection~\ref{sect-6-3-1}.

For exact statements, details, proofs and generalizations, see Subsections~\ref{book_new-sect-24-4-2}  of \cite{futurebook}.
For improvements for slower decaying potentials, see Subsections~\ref{book_new-sect-24-4-3}  of \cite{futurebook}.

\chapter{Applications to multiparticle quantum theory}
\label{sect-7}

\section{Problem set-up}
\label{sect-7-1}

In this Chapter, we discuss an application to Thomas-Fermi Theory.
Consider a large (heavy) atom or molecule; it is described by a \emph{Multiparticle Quantum Hamiltonian\/}
\begin{equation}
\sfH_N= \sum_{1\le n\le N} H_V(x_n)+
\sum_{1\le n< k\le N}\frac{1}{|x_n-x_k|}
\label{7-1-1}
\end{equation}
where $H$ is a \emph{one-particle quantum Hamiltonian\/}, Planck constant $\hbar=1$, electron mass $=\frac{1}{2}$, electron charge $=-1$. This operator acts on the space $\fH\Def \wedge_{1\le j\le N} \sL^2(\bR^3,\bC^2)$ of totally antisymmetric functions $\Psi(x_1,\varsigma_1;\ldots; x_N,\varsigma_N)$ because the electrons are fermions, $x_n=(x^1_n,x^2_n,x^3_n)$ is a coordinate and\linebreak
$\varsigma_n\in \{-\frac{1}{2},\frac{1}{2}\}$ is the \emph{spin\/} of $n$-th particle. We identify the $\bC^2$-valued function $\psi(x)$ on $\bR^3$ with a scalar-valued function $\psi(x,\varsigma)$.

If the electrons did not interact between themselves, but the field potential was $-W(x)$, then they would occupy the lowest eigenvalues and the ground state wave functions would be the anti-symmetrized product $\phi_1(x_1,\varsigma_1)\phi_2(x_2,\varsigma_2)\ldots \phi_N(x_N,\varsigma_N)$, where $\phi_n$ and $\lambda_n$ are the eigenfunctions and eigenvalues of $H_W$ respectively.

Then the local electron density would be
$\rho_\Psi = \sum_{1\le n\le N} |\phi_n(x)|^2$ and according to the pointwise Weyl law (if there is no magnetic field)
\begin{equation}
\rho_\Psi (x) \approx \frac{1}{3\pi^2} (W + \nu)_+ ^{\frac{3}{2}},
\label{7-1-2}
\end{equation}
where $\nu=\lambda_N$. We first assume that there is no magnetic field and therefore, $H_V=-\Delta-V(x)$.

This density would generate the potential $-|x|^{-1}* \rho_\Psi$ and we would have $W\approx V-|x|^{-1}*\rho_\Psi$.

Replacing all approximate equalities by strict ones, we arrive to the \emph{Thomas-Fermi equations\/}:
\begin{align}
&V-W^\TF=|x|^{-1}*\rho^\TF,\label{7-1-3}\\
&\rho^\TF = \frac{1}{3\pi^2} (W^\TF+\nu)_+^{\frac{3}{2}},\label{7-1-4}\\
&\int \rho^\TF\,dx=N,\label{7-1-5}
\end{align}
where $\nu\le 0$ is called the \emph{chemical potential\/} and in fact approximates $\lambda_N$.

Considering atoms and molecules, we assume that
\begin{equation}
V(x)=\sum _{1\le m\le M} \frac{Z_m}{|x-\y_m|},
\label{7-1-6}
\end{equation}
where $\y_m$ is the \emph{position\/} and $Z_m$ is the \emph{charge\/} of the $m$-th nuclei, $M$ is fixed and $Z_1\asymp Z_2\asymp \ldots \asymp Z_M \asymp N\to \infty$.

Thomas-Fermi theory has been rigorously justified (with pretty good error estimates) and we want to explain how.\enlargethispage{2\baselineskip}

\section{Reduction to one-particle problem}
\label{sect-7-2}

\subsection{Estimate from below}
\label{sect-7-2-1}

We start from the estimate from below. The \emph{ground state energy\/}\linebreak
$\E_N\Def\inf \blangle\sfH_N \Psi,\Psi\brangle$, taken over all  $\Psi\in \fH $ with $\|\Psi\|=1$, where $\blangle \cdot,\cdot\brangle $ denotes the inner product in $\fH$.  Classical mathematical physics provides a wealth of results. One of them is the \emph{electrostatic inequality\/} due to E.~H.~Lieb~\cite{lieb:selecta}:\enlargethispage{\baselineskip}
\begin{multline}
\sum_{1\le j<k\le N}
\int |x_j-x_k| ^{-1}|\Psi (x_1,\dots ,x_N)| ^2\, dx_1 \cdots dx_N\ge\\
\frac{1}{2}\D(\rho_\Psi ,\rho_\Psi )-C\int \rho_\Psi ^{\frac{4}{3}}(x)\, dx,
\label{7-2-1}
\end{multline}
with $\rho_\Psi$ defined by (\ref{7-1-2}). This inequality holds for all (not necessarily antisymmetric) functions $\Psi $ with
$\|\Psi \|_{\sL ^2(\bR ^{3N})}=1$. Therefore,
\begin{multline}
\blangle \mathsf{H}_N\Psi ,\Psi \brangle \ge
\sum_{1\le j\le N} \blangle H_{V,x_j}\Psi ,\Psi \brangle +
\frac{1}{2}\D\bigl(\rho_\Psi ,\rho_\Psi )-
C\int \rho_\Psi ^{\frac{4}{3}}(x)\,dx=\\[3pt]
\sum_{1\le j\le N}
\blangle H_{W,x_j}\Psi ,\Psi \brangle +
\frac{1}{2}\D\bigl(\rho_\Psi - \rho ,\rho_\Psi -\rho \bigr)-
\frac{1}{2}\D\bigl( \rho , \rho \bigr)-C\int \rho_\Psi ^{\frac{4}{3}}(x)\,dx
\label{7-2-2}
\end{multline}
and $H_W$ is a one-particle Schr\"odinger operator with potential
\begin{gather}
W=V-|x| ^{-1}*\rho ,
\label{7-2-3}\\
\intertext{where $\rho$ is an arbitrarily chosen real-valued non-negative function and therefore,}
(V-W,\rho_\Psi)= -\D(\rho, \rho_\Psi).
\label{7-2-4}
\end{gather}

The physical sense of the second term in $W$ is transparent: it is a potential created by the charge $-\rho$. Skipping the positive second term in the right-hand expression of (\ref{7-2-2}) and believing that the last term is not very important for the ground state function $\Psi$\,\footnote{\label{foot-40} When we derive the upper estimate for $\E$, we will get an upper estimate for this term as a bonus.}, we see that we need to estimate from below the first term. Since the first term is simply the sum of operators acting with respect to different variables, we can estimate it from below by
\begin{equation}
\blangle (H_{W,x_j}-\nu)\Psi ,\Psi \brangle+\lambda N
\label{7-2-5}
\end{equation}
with arbitrary $\nu$; therefore, it is bounded from below by
$\Tr ((H_W-\nu)_-)$, where $(H_W-\nu)_-$ denotes the negative part of the operator $(H_W-\nu)$, and hence its trace is the sum of the negative eigenvalues.

Here, the assumption that $\Psi $ is antisymmetric is crucial. Namely, for general (or symmetric--does not matter) $\Psi$, the best possible estimate is
$\lambda _1 N$ where $\lambda _1$ is the lowest eigenvalue of $H_W$ (we always assume that there are sufficiently many eigenvalues below the bottom of the essential spectrum of $H_W$) and we cannot apply semiclassical theory.

Thus we arrive to
\begin{equation}
\E_N \ge \Tr ((H_W-\nu)_-) +\nu N -
\frac{1}{2}\D\bigl( \rho , \rho \bigr) - CN^{\frac{5}{3}},
\label{7-2-6}
\end{equation}
where we used another result of E.~H.~Lieb \cite{lieb:selecta}:
$\int \rho_\Psi ^{\frac{4}{3}}(x)\,dx \le CN^{\frac{5}{3}}$ for the ground state $\Psi$.

\subsection{Estimate from above}
\label{sect-7-2-2}

Here, we simply plug in a test function $\Psi$ which is an (anti-symmetrized) product $\phi_1(x_1,\varsigma_1)\phi_2(x_2,\varsigma_2)\ldots \phi_N(x_N,\varsigma_N)$ where $\phi_n$ and $\lambda_n$ are eigenfunctions and eigenvalues of $H_W$ respectively, and we pick $W$ later. It may happen, however, that $H_W$ does not have $N$ negative eigenvalues, then we can reduce $N$ and use the inequality $\E_N\le E_{N'}$ as $N'\le N$.

Then, $\E_N$ is estimated from above by
\begin{multline}
\blangle \mathsf{H}_N\Psi ,\Psi \brangle  = \\
\sum_n (H_{W,x_j}-\lambda)\Psi ,\Psi \brangle + \lambda N -
(V-W, \rho_\Psi) + \frac{1}{2}\D(\rho_\Psi,\rho_\Psi) \\
-\frac{1}{2}\sum_n \iint |x-y|^{-1} |\psi_n(x)|^2|\psi_n(y)|^2\,dxdy
\label{7-2-7}
\end{multline}
and therefore recalling (\ref{7-2-4}), we obtain
\begin{equation}
\E_N\le
\Tr ((H_W-\lambda)_-) + \lambda N
 + \frac{1}{2}\D(\rho_\Psi-\rho,\rho_\Psi-\rho) -\frac{1}{2}\D(\rho,\rho)
\label{7-2-8}
\end{equation}
and $\rho_\Psi = \tr e_N(x,x) $ where $e_N(x,y)$ and $e(x,y,\lambda)$ are the Schwartz kernels of the projector to the $N$ lowest eigenvalues of $H_W$ and of the operator $\uptheta(\lambda-H_W)$ respectively; here $\tr$ denotes the matrix trace, and $\lambda=\lambda_N$ if $\lambda_N<0$ and $\lambda=0$ otherwise. Finally, we conclude that
\begin{multline}
\E_N\le
\Tr ((H_W-\nu)_-) + \nu N + |\lambda -\nu|\cdot |\N^-(H_W-\nu)-N|\\
 + \frac{1}{2}\D(\tr e_N (x,x)-\rho,\tr e_N(x,x)-\rho) -\frac{1}{2}\D(\rho,\rho)
\label{7-2-9}
\end{multline}
with arbitrary $\nu\le 0$.

\section{Semiclassical approximation}
\label{sect-7-3}

\subsection{Estimate from below}
\label{sect-7-3-1}

In the estimate from below (\ref{7-2-6}), we replace $\Tr ((H_W-\nu)_-)$ by its semiclassical approximation
\begin{gather}
\Tr ((H_W-\nu)_-) \approx -\int P(W+\nu)\,dx
\label{7-3-1}\\
\shortintertext{with}
P(W+\nu)\Def \frac{2}{5\pi^2} (W+\nu)_+^{\frac{5}{2}},
\label{7-3-2}
\end{gather}
and also plug in $\rho = \frac{1}{4\pi}\Delta (W-V)$; then we obtain the functional \begin{equation}
\Phi_*(W,\nu)=-\int P(W+\nu)\,dx -\frac{1}{8\pi}\|\nabla (W-V)\|^2 + \nu N;
\label{7-3-3}
\end{equation}
maximizing it, we arrive to the Thomas-Fermi equations and its maximal value is $\cE_N^\TF$, delivered by Thomas-Fermi theory. Then, we need to understand the semiclassical error. To do this, we use the properties of the Thomas-Fermi potential and rescale $x\mapsto xN^{\frac{1}{3}}$ and $\tau\mapsto N^{-\frac{4}{3}}\tau $ (so, $\nu\mapsto  N^{-\frac{4}{3}} \nu$)  with
\begin{equation}
H_W= -h^2 \nabla ^2 - W,
\label{7-3-4}
\end{equation}
where near $\y_m$, the rescaled potential is Coulomb-like: $W\sim z_m|x-y|^{-1}$ with $z_m= Z_mN^{-1}$.

Then, we can apply results Subsection~\ref{sect-3-2-6} (see (\ref{3-2-25})): for the operator (\ref{7-3-4}),
\begin{equation}
\Tr ((H_W-\nu)_-)= -h^{-3} \int P(W+\nu)\,dx + \kappa h^{-2} + O(h^{-1}),
\label{7-3-5}
\end{equation}
where in this case, the numerical value of $\kappa=2\sum_m z_m^2$ is well-known. Scaling back, we obtain $\cE ^\TF _N+ \Scott+O(N^{\frac{5}{3}})$ where the leading term is of magnitude $N^{\frac{7}{3}}$ and the Scott correction term $\Scott=2\sum_m Z_m^2$. Here, we need to assume that $|\y_m-\y_{m'}|\gtrsim 1$ after rescaling (and $|\y_m-\y_{m'}|\gtrsim N^{-\frac{1}{3}}$ before it).

Indeed, after rescaling, we get an operator which is uniformly in the framework of Subsection~\ref{sect-3-2-6} due to the following properties of the Thomas-Fermi potential:
\begin{claim}\label{7-3-6}
Before rescaling, $W^\TF=Z_m |x-\y_m|^{-1}+O(N)$ for
$|x-\y_m|\lesssim N^{-\frac{1}{3}}$ and
$W^\TF\asymp \sum_m \bigl(|x-\y_m|^{-4} +  (Z-N)_+ |x-\y_m|^{-1}\bigr)$ for $|x-\y_m|\gtrsim 1$ for all $m=1,\ldots,M$.
\end{claim}

In fact, the analysis of Subsection~\ref{sect-3-2-6} was mainly motivated by this problem.

\subsection{Estimate from above}
\label{sect-7-3-2}

Again, using the semiclassical approximation (\ref{7-3-1}) for $\Tr ((H_W-\nu)^-)$ and also $e_N(x,x)\approx P'(W+\nu)$ with $P'=\frac{1}{3\pi^2}(W+\nu)^{\frac{3}{2}}$ the derivative of $P(W+\nu)$, we arrive to the functional
\begin{multline}
\Phi^*(W,\nu)=-\int P(W+\nu)\,dx -\frac{1}{8\pi}\|\nabla (W-V)\|^2  + \nu N \\
+\D( P'(W+\nu) -\frac{1}{4\pi}\Delta (W-V),\
P'(W+\nu) -\frac{1}{4\pi}\Delta (W-V));
\label{7-3-7}
\end{multline}
minimizing it, we again arrive to the Thomas-Fermi equations and the minimal value is $\cE_N^\TF$, again delivered by Thomas-Fermi theory.

However, in addition to the semiclassical error for the trace, we have  other  errors from (\ref{7-2-9}):
\begin{gather}
|\lambda -\nu|\cdot |\N^-(H_W-\nu)-N|,
\label{7-3-8}\\
\D( \tr e(x,x,\nu)- P'(W+\nu) ,\, \tr e(x,x,\nu)- P'(W+\nu))
\label{7-3-9}
\shortintertext{and}
\D( \tr e_N(x,x) -\tr  e(x,x,\nu) ,\, \tr  e_N(x,x)-\tr  e(x,x,\nu)).
\label{7-3-10}
\end{gather}
The expression (\ref{7-3-9}) is the semiclassical error and after rescaling it, we can estimate it by $O(h^{-4})$ (due to the pointwise spectral asymptotics). When scaling back, we gain the factor $N^{\frac{1}{3}}$, resulting in $O(N^{\frac{5}{3}})$.

Expressions (\ref{7-3-8}) and (\ref{7-3-10}) can be also estimated by $O(N^{\frac{5}{3}})$ based on another semiclassical error
\begin{equation}
\N^-(H_W-\sigma)-\int P'(W+\sigma)\,dx = O(h^{-2}),
\label{7-3-11}
\end{equation}
(for $\sigma \le 0$) after rescaling and thus, $O(N^{\frac{2}{3}})$ in the original scale, due to the definitions of $\lambda$ and $\nu$. One needs to consider four cases depending on whether $\lambda <0$ (i.e. $N^-(H_W)\ge  N$) or $\lambda =0$  (i.e. $N^-(H_W)<  N$) and whether $\nu <0$ (i.e. $N< Z$) or $\nu =0$ (i.e. $N\ge Z$), where $Z=Z_1+\ldots+Z_M$ is the total charge of the nuclei.

\subsection{More precise estimates}
\label{sect-7-3-3}

If we want to improve the remainder estimate $O(N^{\frac{5}{3}})$, then we need to improve the semiclassical remainder estimates and also deal with $O(N^{\frac{5}{3}})$ in  Lieb's electrostatic inequality (\ref{7-2-1}).

The first task could be done under the assumption
\begin{equation}
a\Def \min_{m\ne m'} |\y_m-\y_{m'}|\gg \bar{a}\Def N^{-\frac{1}{3}},
\label{7-3-12}
\end{equation}
which is completely reasonable (see Section~\ref{sect-7-4}). In this case, in each zone $\cY_m\Def \{x:\, |x-\y_m|\le a^{1-\eta}\bar{a}^\eta  \}$, with $\eta>0$, both $\rho^\TF$ and $W^\TF$ are close to those of a single atom which are spherically symmetric. Then one can prove easily that the standard conditions to the trajectories are fulfilled and we may use the improved remainder estimates.   On the other hand, contributions of the ``outer'' zone $\cY_0\Def\{x:\, |x-\y_m|\ge a^{1-\eta}\bar{a}^\eta\ \forall m=1,\ldots, M\}$ to these remainders is smaller.

Therefore all remainder estimates acquire the factor $(h^\delta + b^{-\delta})$ with
$b= a\bar{a}^{-1}$ before scaling back, i.e.
$(N^{\frac{1}{3}\delta} + (a N^{\frac{1}{3}})^{-\delta})$ after it. However, the trace asymptotics should also include the term $-\kappa_1 h^{-1}$ before scaling back or $-\kappa_1 N^{\frac{5}{3}}$ after it; for the potential $W^\TF$, it is numerically equal to $\Schwinger= -c_1\int \rho^{\TF\,\frac{4}{3}}\,dx$ which is called the \emph{Schwinger correction term\/}.

The second task requires an improvement in Lieb's electrostatic inequality due to \cite{graf:solovej} and \cite{bach}: one can replace the last term in (\ref{7-2-2}) for the ground state energy $\Psi$ by
\begin{equation}
-\frac{1}{2}\iint |x-y|^{-1}\tr \bigl(e^\dag_N(x,y) e_N(x,y)\bigr)\, dxdy - O(N^{\frac{5}{3}-\delta}),
\label{7-3-13}
\end{equation}
where the first term coincides with the last term in (\ref{7-2-7}) (the estimates from above) and again modulo $O(N^{\frac{5}{3}-\delta})$ can be rewritten as
\begin{equation}
-\frac{1}{2}\iint |x-y|^{-1}\tr \bigl(e^\dag (x,y,\nu) e(x,y,\nu)\bigr)\, dxdy.
\label{7-3-14}
\end{equation}
So far, we have not explored such expressions but we can handle them.

For this expression, after rescaling, we can derive the asymptotics with  principal term $-\kappa_2 h^{-4}$ and with remainder estimate as good as $O(h^{-3})$, which after scaling back becomes $O(N^{\frac{4}{3}})$ (which is an overkill). Here, we use the  representation of $e(x,y,\nu)$ by an oscillatory integral modulo a term whose $\sL^2(\bR^6)$ norm does not exceed $Ch^{-2}$.

To calculate $\kappa_2$, we can consider the operator with constant potential $W$, and for this operator, we calculate
$-\frac{1}{2}\int |x-y|^{-1}\tr \bigl(e^\dag (x,y,\nu) e(x,y,\nu)\bigr)\, dy$ obtaining $-\const (W+\nu)^{2} h^{-4}$, then plug in $W=W(x)$ and integrate over $x$. For $W=W^\TF$, after scaling back, we arrive to $\Dirac=-c_2\int \rho^{\TF\,\frac{4}{3}}\,dx$
which is called the \emph{Dirac correction term\/}.

Despite having completely different origins, these correction terms differ only by numerical constants.

We arrive to the theorem:

\begin{theorem}\label{thm-7-3-1}
As $Z=Z_1+\ldots +Z_M\asymp N\to \infty$, $M$ remains bounded,
$a=\min_{m\ne m'} |\y_m-\y_{m'}|\gtrsim N^{-\frac{1}{3}}$ and
\begin{multline}
\E_N =\cE_N^\TF +\Scott +\Schwinger+\Dirac +O(R)
\\
\text{with\ \ } R=N^{\frac{5}{3}}\bigl (N^{-\delta}+(a N^{\frac{1}{3}})^{-\delta} \bigr )
\label{7-3-15}
\end{multline}
where $\delta>0$ is unspecified.
\end{theorem}

As a byproduct of the proof, we obtain
\begin{equation}
\D(\rho_\Psi-\rho^\TF,\,\rho_\Psi-\rho^\TF)=  O(R).
\label{7-3-16}
\end{equation}

For details and proofs, see Sections~\ref{book_new-sect-25-1}--\ref{book_new-sect-25-4} of \cite{futurebook}.
 \enlargethispage{1\baselineskip}

\section{Ramifications}
\label{sect-7-4}

First, instead of the \emph{fixed nuclei model\/}, we can consider the \emph{free nuclei model\/} where we add to both $\E_N$ and $\cE_N^\TF$ the energy of nuclei-to-nuclei interaction
\begin{equation}
\sum_{m<m'} Z_mZ_{m'}|\y_m-\y_{m'}|^{-1}
\label{7-4-1}
\end{equation}
and minimize the results by the position of nuclei $(\y_1,\ldots,\y_m)$; denote the results by $\widehat{\E}_N$ and $\widehat{\cE}_N^\TF$ respectively.

Combining (\ref{7-3-15}) with the non-binding theorem in Thomas-Fermi theory\footnote{\label{ft-42} In the Thomas-Fermi theory, molecules do not exist.}, we obtain that in the free nuclei model (with $Z_1\asymp\ldots \asymp Z_M\asymp Z\asymp N$),
\begin{equation}
a=\min_{m\ne m'} |\y_m-\y_{m'}|\gtrsim N^{-\frac{5}{3}+\delta}
\label{7-4-2}
\end{equation}
 and then  (\ref{7-3-15}) and (\ref{7-3-16}) hold with $R=N^{\frac{5}{3}-\delta}$.

\medskip

Next, using  methods already developed by mathematical physicists before asymptotics (\ref{7-3-15}) and (\ref{7-3-16}) were derived, we can answer several questions with far better precision than before; for simplicity, we assume that $a\ge N^{-\frac{1}{3}+\delta}$.

\begin{enumerate}[label=(\roman*), wide, labelindent=0pt]
\item\label{sect-7-4-i}
How many extra electrons can the system bind? In other words, if $\E_N<\E_{N-1}$, what we can say about $N-Z$? According to a classical  theorem due to G.~Zhislin, the system can bind at least $Z$ electrons. Our answer: $(N-Z)_+ =O(N^{\frac{5}{7}-\delta})$, based on the fact that in the Thomas-Fermi theory, negative ions do not exist.

\item\label{sect-7-4-ii}
What we can say about the \emph{ionization energy\/} $\I_N=\E_{N-1}-\E_N$?
Our answer: $\I_N =O(N^{\frac{20}{21}-\delta})$ if
$N-Z\ge -C N^{\frac{5}{7}-\delta}$ and
$\I_N = -\nu + O((Z-N)^{\frac{17}{18}} Z^{\frac{5}{18}-\delta}$ if
$N-Z\le -C N^{\frac{5}{7}-\delta}$; if
$N\le Z$ $\nu \asymp (Z-N)^{\frac{4}{3}}$.

\item\label{sect-7-4-iii}
In the free nuclei model (with $M\ge 2$), what can we say about $N-Z>0$ if a stable configuration exists? Our answer: $Z-N\le C N^{\frac{5}{7}-\delta}$ (again based on the non-binding theorem).
\end{enumerate}

For details and proofs, see Sections~\ref{book_new-sect-25-5} and~\ref{book_new-sect-25-6} of \cite{futurebook}.

\section{Adding magnetic field}
\label{sect-7-5}

\subsection{Adding external magnetic field}
\label{sect-7-5-1}

Consider  the Schr\"odinger-Pauli operator with magnetic field
\begin{equation}
H_{A,V}= ((-ih\nabla - \mu A (x))\cdot \boldupsigma)^2 -V(x).
\label{7-5-1}
\end{equation}
Then instead of $P(w)$ defined by (\ref{7-3-2}), we need to define it according to (\ref{5-2-3}) by
\begin{equation}
P(w)= \frac{2}{\pi^2}\Bigl(\frac{1}{2} w_+^{\frac{3}{2}} B+\sum_{j=1}^{\infty} (w- 2j B)_+^{\frac{3}{2}}B\Bigr),
\label{7-5-2}
\end{equation}
where $B$ is the scalar intensity of the magnetic field. This changes both the Thomas-Fermi theory and properties of the Thomas-Fermi potential $W^\TF$ and Thomas-Fermi density $\rho^\TF$.

\subsubsection{Case $B\lesssim Z^{4/3}$}
\label{sect-7-5-1-1}

For $B\lesssim Z^{4/3}$, the main contributions to the (approximate) electronic charge $\int \rho^\TF \,dx$ and  the energy $\cE^\TF$ come from the zone
$\{x:\, d(x) \asymp Z^{-1/3}\}$ ($d(x)=\min _m |x-\y_m|$), exactly as for $B=0$.

Furthermore, $W^\TF\asymp Z_m d(x)^{-1}$ if $d(x)\lesssim Z^{-1/3}$ and (for $Z=N$) $W^\TF \asymp d(x)^{-4}$ if $  Z^{-1/3}\le  d(x)\lesssim B^{-1/4}$ but $\rho^\TF=0$ if $d(x)\ge C_0B^{-1/4}$\,\footnote{\label{ft-43} So, the radii of atoms in Thomas-Fermi theory are
$\asymp \min (B^{-1/4},(Z-N)_+^{-1/3})$.}.

Finally, as we using scaling to bring our problem to the standard one, we get that in the zone $\{x:\,d(x)\asymp Z^{-1/3}\}$, the effective semiclassical parameter is $h_{\eff} = Z^{-1/3}$ which leads to $\cE^\TF \asymp Z^{7/3}$ again exactly as for $B=0$.

As a result, assuming that $M=1$, we can recover asymptotics for the ground state energy $\E$ with the Scott  correction term  but with the remainder estimate $O(Z^{5/3}+ Z^{4/3}B^{1/3})$. For $M\ge 2$ and $N\ge Z$, our estimates are almost as good (provided $a= \min_{m\ne m'} |\y_m-\y_{m'}| \ge Z^{-1/3}$), but deteriorate when both  $(Z-N)_+$ and $B$ are large.

Moreover, for $B\ll Z$ assuming (\ref{7-3-12}), we can marginally improve these results and include the Schwinger and Dirac correction terms.

The main obstacles we need to overcome are that now $W^\TF$ is not infinitely smooth but only belongs to the class $\sC^{5/2}$ and that for $M\ge 2$, the nondegeneracy assumption ($|\nabla W|\asymp 1$ after rescaling) fails.
\enlargethispage{\baselineskip}

\subsubsection{Case $B\gtrsim Z^{4/3}$}
\label{sect-7-5-1-2}

On the other hand, for $B\gtrsim Z^{4/3}$, the the main contributions to the (approximate) electronic charge  and  the energy $\cE^\TF$ come from the zone $\{x:\, d(x)\asymp B^{-1/4}\}$ and (for $Z=N$)
$W^\TF \asymp Zd(x)^{-1}$ if  $ d(x)\lesssim B^{-1/4}$ but $W^\TF=0$ if
$d(x)\ge C_0B^{-1/4}$. In this case, $\E^\TF \asymp B^{2/5}Z^{9/5}$.

Further, as we using scaling to bring our problem to the standard one, we see that in the zone $\{x:\,d(x)\asymp B^{-1/4}\}$, the effective semiclassical parameter is $h_{\eff} = B^{1/5}Z^{-3/5}$  and therefore unless $B\ll Z^3$, the semiclassical approximation fails and the correct answer should be expressed in completely different terms \cite{LSY1}.

As a result, assuming that $M=1$ if $Z^{4/3}\le B\le Z^3$, we can recover the asymptotics for $\E$ with the Scott  correction term  but with the remainder estimate $O(B^{4/5}Z^{3/5}+ Z^{4/3}B^{1/3})$.

For $M\ge 2$ and $N\ge Z$, our estimates are almost as good (provided $a= \min_{m\ne m'} |\y_m-\y_{m'}| \ge B^{-1/4}$), but deteriorate when
 $(Z-N)_+$ is large.

Again the main obstacles we need to overcome are that now $W^\TF$ is not infinitely smooth but only belongs to $\sC^{5/2}$ and that for $M\ge 2$, the nondegeneracy assumption ($|\nabla W|\asymp 1$ after rescaling) fails.

For details, exact statement and proofs, see  Sections~\ref{book_new-sect-26-1} and~\ref{book_new-sect-26-6} of \cite{futurebook}. We also estimate the left-hand expression of (\ref{7-3-5}) and are able to obtain results similar to those mentioned in Section~\ref{sect-7-4}. For details and proofs, see Sections~\ref{book_new-sect-26-7}--\ref{book_new-sect-26-8} of \cite{futurebook}.

\subsection{Adding self-generated magnetic field}
\label{sect-7-5-2}

Let
\begin{gather}
\E(A)= \inf \Spec (\sfH_{A,V} ) +
\underbracket{\alpha^{-1} \int |\nabla \times A|^2\,dx}
\label{7-5-3}\\
\shortintertext{and}
\E ^*=\inf_{A\in \sH^1_0} \E(A),
\label{7-5-4}
\end{gather}
where $A$ is an unknown magnetic field and the underlined term is its energy. One can prove that an ``optimal'' magnetic field exists (for given parameters  $Z_1,\ldots, Z_M, \y_1,\ldots,\y_M, N$) but we do not know if it is unique\footnote{\label{ft-44} If it was unique, then for $M=1$,  the spherical symmetry would imply that $A=0$.}.

Using the same arguments as before, we can reduce this problem to the one-particle problem with $\inf \Spec (\sfH_{A,V} )$ replaced by $\Tr ((H_{A,W}+\nu)_-)$ plus some other terms.  However,  in the estimate from below, most of the terms  do not depend on $A$ and in the estimate from above, we pick up $A$.

Then after the usual rescalings, the problem is reduced to the problem of minimizing
\begin{equation}
\Tr ((H_{A,W}+\nu)_-) +
\underbracket{\frac{1}{\kappa h^2} \int |\nabla \times A|^2\,dx}
\label{7-5-5}
\end{equation}
and then the optimal magnetic potential $A$ must satisfy
\begin{multline}
\frac{2}{\kappa h^2} \Delta A_j (x)   = \Phi_j\Def\\
-\Re\tr \Bigl( \upsigma_j\Bigl( (hD -A)_x \cdot \boldupsigma  e (x,y,\tau)+
e (x,y,\tau)\,^t (hD-A)_y \cdot \boldupsigma \Bigr)  \Bigr|_{y=x}\Bigr),
\label{7-5-6}
\end{multline}
where $e(x,y,\tau)$ is the Schwartz kernel of the spectral projector
$\uptheta (- H)$ of $H=H_{A,W}$ and $\tr$ is the matrix trace. As usual, we are mainly interested in $h=Z^{-1/3}$ (and then $\kappa=\alpha Z$).

First, (\ref{7-5-6}) allows us to claim a certain smoothness of $A$. Second, the right-hand expression is something we studied in pointwise spectral asymptotics, and the Weyl expression here is $0$,   so the right-hand expression of (\ref{7-5-6}) is something that we could estimate. Surely, it is not that simple but improving our methods in the case of smooth $W$, we are able to prove that $A$ is so small that the ordinary asymptotics with remainder estimates $O(h^{-2})$ and $O(h^{-1})$ would hold in both the pointwise asymptotics and the trace asymptotics. Moreover, under standard conditions, we would be able to get the remainder estimates $o(h^{-2})$ and $o(h^{-1})$ in the eigenvalue counting and the trace asymptotics respectively.

However, in reality, the above is not exactly true since $W$ has Coulomb-like singularities $W\sim z_m|x-\y_m|^{-1}$ with $z_m\asymp 1$. If $M=1$, $z_m=1$, a singularity leads us to the Scott correction term $S(\kappa)h^{-2}$ derived in the same way as without a self-generated magnetic field. However, we do not have an explicit formula for $S(\kappa)$; we even do not know its properties  except that it is non-increasing function of $\kappa\in [0,\kappa^*)$; we even do not know if we can take $\kappa^*=\infty$. If the optimal magnetic potential $A$ was unique, then $A=0$ and $S(\kappa)=S(0)$, which corresponds to this term without a magnetic field.

Then as $M\ge 2$, the Scott correction term is
$\sum_{1\le m\le M} S(\kappa z_m)z_m^2 h^{-2}$ in the general case. However, as $M\ge2$ we need to decouple singularities as all of them are served by the same $A$ and it leads to decoupling errors depending on the internuclei distance.

For details, exact statements  and proofs, see Sections~\ref{book_new-sect-27-2}--\ref{book_new-sect-27-3} of \cite{futurebook}.

As a result, we derive the ground state asymptotics with the Scott correction term $\sum_{1\le m\le M} S(\alpha Z_m)Z_m^2$. We also estimate the left-hand expression of (\ref{7-3-5}) and are able to obtain results similar to those mentioned in Section~\ref{sect-7-4}. For details, exact statements  and proofs, see Sections~\ref{book_new-sect-27-3}--\ref{book_new-sect-27-4} of \cite{futurebook}.

\subsection{Combining external and self-generated magnetic fields}
\label{sect-7-5-3}

We can also combine a constant strong external magnetic field and a self-generated magnetic field. Results are very similar to those of Subsection~\ref{sect-7-5-1}, but this time, the Scott correction term and the decoupling errors are like in Subsection~\ref{sect-7-5-2}.
 For details, exact statements  and proofs, see Chapter~\ref{book_new-sect-28} of \cite{futurebook}.

\input{BMS.bbl}

\end{document}

%% file: BMS.bbl
\providecommand{\bysame}{\leavevmode\hbox to3em{\hrulefill}\thinspace}

%\vglue .06truein
%
%
%
%\begin{tabular}{rrl}
%&{\hskip 200 pt} &Department of Mathematics,\cr
%&&University of Toronto,\cr
%&&40, St.George Str.,\cr
%&&Toronto, Ontario M5S 2E4\cr
%&&Canada\cr
%&&ivrii@math.toronto.edu\cr
%&&Fax: (416)978-4107\cr
%\end{tabular}